\documentclass[12pt]{amsart}
\usepackage{fullpage}
\usepackage{graphicx}
\usepackage{todonotes}

\newtheorem{theorem}{Theorem}[section]
\newtheorem{corollary}[theorem]{Corollary}
\newtheorem{proposition}[theorem]{Proposition}
\newtheorem{definition}[theorem]{Definition}
\newtheorem{lemma}[theorem]{Lemma}
\newtheorem{remark}[theorem]{Remark}

\numberwithin{equation}{section}
\newcommand\CC {{\mathbb C}}

\newcommand\KK {{\mathbb K}}
\newcommand\NN {{\mathbb N}}

\newcommand\QQ {{\mathbb Q}}
\newcommand\RR {{\mathbb R}}
\newcommand\TT {{\mathbb T}}
\newcommand\ZZ {{\mathbb Z}}

\newcommand\sltwor{{\rm SL(2,\RR)}}

\newcommand\veech{{\rm SL}}

\newcommand\bad{{\rm Bad}}

\newcommand\sys{{\rm Sys^{sc} }}
\newcommand\cyl{{\rm Sys^{cyl} }}
\newcommand\area{{\rm Area }}
\newcommand\leb{{\rm Leb }}

\newcommand\re{{\rm Re }}
\newcommand\im{{\rm Im }}
\newcommand\hol {{\rm Hol}}

\newcommand\cB{{\mathcal{B}  }}
\newcommand\cC{{\mathcal{C}  }}

\newcommand\cG{{\mathcal{G}  }}
\newcommand\cH{{\mathcal{H}  }}
\newcommand\cI{{\mathcal{I}  }}

\newcommand\cK{{\mathcal{K}  }}

\newcommand\cM{{\mathcal{M}  }}

\newcommand\cR{{\mathcal{R}  }}

\newcommand\cT{{\mathcal{T}  }}
\newcommand\cU{{\mathcal{U}  }}
\newcommand\cV{{\mathcal{V}  }}
\newcommand\cW{{\mathcal{W}  }}

\begin{document}

\title{Diophantine approximations for translation surfaces and planar resonant sets}

\author{Luca Marchese, Rodrigo Trevi\~{n}o, Steffen Weil}

\address{Luca Marchese, Universit\'e Paris 13, Sorbonne Paris Cit\'e,
LAGA, UMR 7539, 99 Avenue Jean-Baptiste Cl\'ement, 93430 Villetaneuse, France.}

\email{marchese@math.univ-paris13.fr}

\address{Rodrigo Trevi\~{n}o. Courant Institute of Mathematical Sciences, New York University}

\email{rodrigo@math.nyu.edu}

\address{Steffen Weil, School of Mathematical Sciences, Tel Aviv University, Tel Aviv 69978, Israel.}

\email{steffen.weil@math.uzh.ch}



\begin{abstract}
We consider Teichm\"uller geodesics in strata of translation surfaces. We prove lower and upper bounds for the Hausdorff dimension of the set of parameters generating a geodesic bounded in some compact part of the stratum. Then we compute the dimension of those parameters generating geodesics that make excursions to infinity at a prescribed rate. Finally we compute the dimension of the set of directions in a rational billiard having fast recurrence, which corresponds to a dynamical version of a classical result of Jarn\'ik and Besicovich. Our main tool are planar resonant sets arising from a given translation surface, that is the countable set of directions of its saddle connections or of its closed geodesics, filtered according to length. In an abstract setting, and assuming specific metric properties on a general planar resonant set, we prove a dichotomy for the Hausdorff measure of the set of directions which are well approximable by directions in the resonant set, and we give an estimate on the dimension of the set of badly approximable directions. Then we prove that the resonant sets arising from a translation surface satisfy the required metric properties.
\end{abstract}

\maketitle

\tableofcontents

\section{Introduction}

In this paper we consider a \emph{translation surface} $X$ and we measure the distortion of its flat geometry when we apply the \emph{Teichm\"uller geodesic flow} $g_t$ to the surface $X$ in a given direction $\theta$. In \S~\ref{SecBoundedGeodesicsModuliSpaceIntro} we give estimates on the Hausdorff dimension of the set of directions $\theta$ for which the geometry has uniformly bounded distortion, which is equivalent to saying that $(g_tr_\theta\cdot X)_{t>0}$ is contained in some compact subset of the parameter space with prescribed size. In \S~\ref{SecUnboundedGeodesicsModuliSpaceIntro} we consider directions $\theta$ for which the the flat geometry has unbounded distortion, that is $(g_tr_\theta\cdot X)_{t>0}$ has unbounded excursions to the non compact part of the parameter space, and we state a dichotomy for the Hausdorff measure of the set of directions for which the rate of excursions is prescribed, generalizing some classical results of Jarn\'ik, Besicovich and Khinchin. It's well known that translation surfaces are closely related to \emph{rational billiards}, thus in \S~\ref{SecRecurrenceBilliardsIntro} we consider the \emph{billiard flow} generated by a given direction $\theta$ on a rational polygon $Q$, and we compute the Hausdorff dimension of the set of those $\theta$ for which the \emph{recurrence rate} of the billiard flow has a given value in $(0,1)$. The value of the recurrence rate represents somehow a phase-space counterpart of the rate of excursions in parameter space. All the dynamical properties described above are consequences of specific diophantine conditions. In \S~\ref{SecIntroResonantSet} we describe the relevant diophantine conditions in the abstract setting of \emph{planar resonant sets}, then in \S~\ref{SecIntroHolonomyResonantSets} we state results on translation surfaces which ensures that the abstract diophantine conditions are satisfied on a given surface $X$.

\medskip

A \emph{translation surface} is a genus $g$ closed surface $X$ with a flat metric and a finite set $\Sigma$ of conical singularities $p_1,\dots,p_r$, the angle at each $p_i$ being an integer multiple of $2\pi$. An equivalent definition of translation surface $X$ is the datum $(S,w)$, where $S$ is a compact Riemann surface and $w$ is a holomorphic 1-form on $S$ having a zero at each $p_i$. The relation $k_1+\dots+k_r=2g-2$ holds, where $k_1,\dots,k_r$ are the orders of the zeroes of $w$. In particular the \emph{total multiplicity} at conical singularities of $X$ is the positive integer
$$
m:=2g-2+\sharp(\Sigma).
$$
Any translation surface can be obtained as quotient space $X=P/\sim$ of a suitable polygon $P$ in the complex plane $\CC$ via an equivalence relation $\sim$ on the boundary $\partial P$. More precisely, we assume that boundary $\partial P$ is the union of $2d\geq 4$ segments which come in pairs and are denoted  
$
(\zeta_1,\zeta'_1),\dots,(\zeta_d,\zeta'_d)
$, 
and that there exist complex numbers $z_1,\dots,z_d$ in $\CC$ such that for any $i=1,\dots,d$ the boundary segments $\zeta_i$ and $\zeta'_i$ have the same direction and length of $z_i$, and the opposite orientation induced by the interior of $P$ (that is any $\zeta_i$ touches the interior of $P$ from the opposite side as $\zeta_i'$). The relation $\sim$ is defined on the boundary $\partial P$ identifying for any $i=1,\dots,d$ the sides $\zeta_i$ et $\zeta'_i$ by a translation. This induces identifications of the vertices of $P$, which correspond to conical singularities. The initial polygon $P$ is not necessarily connected, but we assume that this is true for the quotient space $X$. The form $dz$ on $\CC$ projects to the holomorphic 1-form $w$ of $X$. Any surface arising from this construction is a translation surface, the simplest examples being flat tori, which all arise from euclidian parallelograms identifying opposite sides.

A stratum $\cH=\cH(k_1,\dots,k_r)$ is the set of translation surfaces $X$ whose corresponding holomorphic one-form $w$ has $r$ zeros with orders $k_1,\dots,k_r$, where $k_1+\dots+k_r=2g-2$. It is an affine orbifold with complex dimension $2g+r-1$, where affine coordinates around any element $X\in\cH$ are given by the complex numbers $z_1,\dots,z_d$ introduced above, possibly modulo some linear equations with coefficients in $\QQ$. Any stratum admits an action of $\sltwor$, indeed for any translation surface $X=(S,w)$ and any element $G\in\sltwor$ a new translation surface
$
G\cdot X=(G_\ast S,G_\ast w)
$
is defined, where the 1-form $G_\ast w$ is the composition of $w$ with $G$, and $G_\ast S$ is the complex atlas for which $G_\ast w$ is holomorphic. If $X$ is represented as a polygon $P/\sim$ with identified sides then $G\cdot X$ corresponds to the affine image $G\cdot P$ of $P$ with sides pasted according to the same identifications as in $P$. Indeed affine maps preserve parallelism and ratios between lengths. The stabilizer $\veech(X)$ of a translation surface $X$ under this action is known as the \emph{Veech group} of $X$, which is always a discrete subgroup of $\sltwor$. Those surfaces $X$ such that $\veech(X)$ is a lattice in $\sltwor$ are called \emph{Veech surfaces}.

Any $G\in\sltwor$ preserves the euclidian area form
$
dx\wedge dy=i/2dz\wedge d\bar{z}
$
on the plane; therefore we have
$
\area\big(G\cdot X\big)=\area(X)
$,
where for $X=(S,w)$ we set
$$
\area(X):=\frac{i}{2}\int_Xdw\wedge d\bar{w}.
$$
It follows that $\sltwor$ acts on the real sub-orbifold $\cH^{(1)}$ of $\cH$, defined as the set of those translation surfaces $X$ with $\area(X)=1$. It is well-known that $X$ is a Veech surface if and only if its orbit $\cM:=\sltwor\cdot X$ is closed in $\cH^{(1)}$, and in this case $\cM$ is isometric to $\sltwor/\veech(X)$. Relevant subgroups actions are the diagonal group $g_t$, the group of rotations $r_\theta$ and the \emph{horocyclic flow} $u_s$, which are given respectively by
$$
g_t:=
\begin{pmatrix}
e^{t} & 0 \\
0 & e^{-t}
\end{pmatrix}
\textrm{ ; }
u_s:=
\begin{pmatrix}
1 & s \\
0 & 1
\end{pmatrix}
\textrm{ ; }
r_\theta:=
\begin{pmatrix}
\cos\theta & -\sin\theta \\
\sin\theta & \cos\theta
\end{pmatrix}.
$$
The action of the diagonal group $g_t$ is also known as \emph{Teichm\"uller flow}, and corresponds to the geodesic flow for the Teichm\"uller metric, and we refer to $g_t$ orbits as Teichm\"uller geodesics. We refer the reader to \cite{ForniMatheus} and \cite{ZorichFlat} for an exhaustive introduction to translation surfaces and Teichm\"uller dynamics.

\subsection{Bounded geodesics in moduli space}
\label{SecBoundedGeodesicsModuliSpaceIntro}

We identify the complex plane with $\RR^2$. Any segment $\gamma$ of a geodesic for the flat metric of $X$ has a development in the complex plane, also said \emph{holonomy vector}, denoted by $\hol(\gamma,X)\in\RR^2$ and defined by 
$$
\hol(\gamma,X):=\int_\gamma w,
$$ 
where $w$ is the holomorphic one form of $X$. Any such segment $\gamma$ is a geodesic segment also on the surface $G\cdot X$ for any $G\in\sltwor$, and we denote by $\hol(\gamma,G\cdot X)$ its holonomy vector with respect to the surface $G\cdot X$. By definition we have
$$
\hol(\gamma,G\cdot X)=G\big(\hol(\gamma,X)\big).
$$
The length of $\gamma$ on the surface $G\cdot X$ is $|\hol(\gamma,G\cdot X)|$, where $|\cdot|$ denotes the euclidian metric on $\RR^2$.

A \emph{saddle connection} of $X$ is a segment $\gamma$ of a geodesic for the flat metric connecting two conical singularities $p_i$ and $p_j$ and not containing other conical singularities in its interior. The \emph{systole} $\sys(X)$ of $X$ is the length $|\hol(\gamma,X)|$ of the shortest saddle connection $\gamma$ of $X$. According to the \emph{Mumford criterion}, for any fixed $\epsilon>0$ the set of those $X\in\cH^{(1)}$ such that $\sys(X)\geq\epsilon$ is a compact subset of the stratum. 

The set of directions on a translation surface $X$ corresponds to the interval $[-\pi/2,\pi/2[$, with the endpoints identified. The directions
$
\theta\in[-\pi/2,\pi/2[
$
giving rise to positive geodesics whose limit set is contained in $\cK_\epsilon$ are the elements of the set
$$
\bad^{dyn}(X,\epsilon):=
\left\{
\theta
\textrm{ ; }
\liminf_{t\to+\infty}
\sys(g_t r_\theta X) \geq \epsilon
\right\}.
$$
One can consider also the set of all bounded directions
$
\bad^{dyn}(X):=\bigcup_{\epsilon>0}\bad^{dyn}(X,\epsilon)
$.
Although it is a set with zero Lebesgue measure, Kleinbock and Weiss showed it to be \emph{thick}, that is its intersection with any subinterval of $[-\pi/2,\pi/2[$ has full Hausdorff dimension (see \cite{KleinbockWeiss}). Later, Cheung, Chaika and Masur \cite{chaikacheungmasur} improved this result by showing that $\bad^{dyn}(X)$ is an \emph{absolute winning set} for the \emph{absolute Schmidt game} (see \cite{McMullen}), which implies thickness, among other qualitative properties. We also refer to the work of Hubert, Marchese and Ulcigrai \cite{HubertMarcheseUlcigrai}, who studied the \emph{Lagrange spectrum} over the set of bounded directions. Theorem \ref{thmDynamicalJarnickInequality} below develops a quantitative version of the qualitative result in \cite{KleinbockWeiss}, that is thickness. More precisely it establishes non-trivial upper and lower bound for the Hausdorff dimension of $\bad^{dyn}(X,\epsilon)$ in terms of the parameter $\epsilon$. Note that via the \emph{
 Dani correspondence}, in the case of flat tori we obtain similar inequalities as in the classical work of Jarn\'ik on the set of badly approximable numbers (see \cite{jarnik}). Further \emph{Jarn\'ik-type inequalities} are established by Weil in \cite{WeilSteffen}, which is the main source for the techniques used in the proof of Theorem \ref{thmDynamicalJarnickInequality}. 

Fix a translation surface $X$ with $\area(X)=1$ and let $\cH$ be its stratum. Recall that we denote by $m$ the total multiplicity at conical singularities of a translation surface $X$. If $X$ is a Veech surface, let $\cM:=\sltwor\cdot X$ be its closed orbit under the action of $\sltwor$. For any subset $E\subset[-\pi/2,\pi/2[$ let $\dim(E)$ be its Hausdorff dimension.

\begin{theorem}
\label{thmDynamicalJarnickInequality}
There exist positive constants $\epsilon_0$, $c_u$, $c_l$ and $0<\beta\leq 1$, depending only on the integer $m$, such that for any $\epsilon$ with $0<\epsilon< \min\{\epsilon_0,\sys(X)\}$ we have
$$
1-c_l\cdot\frac{\epsilon^{\beta}}{|\log\epsilon|}
\leq
\dim\big(\bad^{dyn}(X,\epsilon)\big)
\leq
1-c_u \cdot\frac{\epsilon^2}{|\log\epsilon|}.
$$
In particular the explicit form of $\beta$ is 
$$
\beta=\frac{1}{3m-1}.
$$
Moreover, if $X$ is a Veech surface, the same inequality holds with $\beta=1$ and with some $\epsilon_0$  which can be chosen uniformly on $\cM$.
\end{theorem}

It is natural to ask whether one can get $\beta=2$ in the lower bound in Theorem \ref{thmDynamicalJarnickInequality}, at least for any Veech surface. We refer to \S~\ref{SecCommentJarnikInequality} for some comments on this question.

\subsection{Unbounded geodesics in moduli space}
\label{SecUnboundedGeodesicsModuliSpaceIntro}

In this paper we also consider geodesics having excursions to the non-compact part of strata at a prescribed rate. The estimates that we prove follow from Theorem \ref{thmabstractkhinchinjarnik} below, which establishes a rather general dichotomy for the size of the set of directions satisfying a given diophantine condition. Unfortunately, while Theorem \ref{thmabstractkhinchinjarnik} admits a very general statement, its dynamical consequences cannot be explicitly stated in full generality. We have first a result on the generic behavior in $\theta$, namely Theorem \ref{thmDynamicalKhinchin} below, which generalizes a previous result of one of the authors (see \cite{MarcheseAsymptoticLaws}). Most of the ideas in the proof of Theorem \ref{thmDynamicalKhinchin} were introduced in \cite{ChaikaHomogeneousApproximationsTranslSurf}.

\begin{theorem}
\label{thmDynamicalKhinchin}
Let $X$ be any translation surface and let $\varphi:\RR^+\to\RR^+$ be a decreasing monotone function.
\begin{enumerate}
\item
If
$
\int_0^\infty \varphi(t)dt
$
converges as $t\to+\infty$, then for almost any $\theta$ we have
$$
\lim_{t\to\infty}
\frac{\sys(g_tr_\theta X)}{\sqrt{\varphi(t)}}=+\infty.
$$
\item
If
$
\int_0^\infty \varphi(t)dt
$
diverges as $t\to+\infty$, then for almost any $\theta$ we have
$$
\liminf_{t\to\infty}
\frac{\sys(g_tr_\theta X)}{\sqrt{\varphi(t)}}=0.
$$
\end{enumerate}
\end{theorem}

In particular, considering the one parameter family of functions
$
\varphi_\epsilon(t):=t^{-(1+\epsilon)}
$ 
and applying both parts of the Theorem, it follows that for almost every $\theta$ we have
\begin{equation}
\label{eqGeneriLogarithmicLaw}
\limsup_{t\to\infty}
\frac{-\log\sys(g_tr_\theta\cdot X)}{\log t}=
\frac{1}{2}.
\end{equation}

Equation \eqref{eqGeneriLogarithmicLaw} above gives the asymptotic maximal size of 
$
-\log\sys(g_tr_\theta\cdot X)
$
along the geodesic in the generic direction $\theta$, and it is inspired by logarithmic laws for geodesics obtained by D. Sullivan and H. Masur, respectively for the case of non-compact hyperbolic manifolds (see \cite{SullivanLogLaw}) and of the moduli space of Riemann surfaces (see \cite{MasurLogLaw}). In \cite{ChaikaTrevino} one can find details on the comparison between the result in \cite{MasurLogLaw} and other analogue logarithmic laws measuring the degeneration of the flat geometry of $g_tr_\theta\cdot X$. Subsets of directions $\theta$ having asymptotic rate for the maximal excursion bigger than in Equation \eqref{eqGeneriLogarithmicLaw} have zero Lebesgue measure, but they can be measured by general \emph{Hausdorff measures} $H^f$ via Theorem \ref{thmabstractkhinchinjarnik} and parts (3) and (4) of Theorem \ref{TheoremMetricResultsHolonomyResonantSets} below, plus an elementary observation corresponding to Equation \eqref{EqInclusionResonantSets}. In particular, for a fixed real number $\alpha$ with $0<\alpha<1$ consider the subset of $[-\pi/2,\pi/2[$ defined by 
$$
S_X(\alpha):=
\left\{
\theta
\textrm{ ; }
\limsup_{t\to\infty}
\frac{-\log\sys(g_tr_\theta\cdot X)-\alpha t}{\log t}
=\frac{1}{2}.
\right\}.
$$
Inspired by the classical Jarn\'ik-Besicovich Theorem on the dimension of the set of real numbers with given diophantine exponent, we develop Theorem \ref{thmLogLaws(Diophantine)} below, which is a version of Jarn\'ik-Besicovich result for the geodesic flow in moduli space. Actually, a natural dynamical behavior corresponding to Jarn\'ik-Besicovich Theorem would be
$$
\limsup_{t\to\infty}
\frac{-\log\sys(g_tr_\theta\cdot X)}{t}=\alpha.
$$
The finer asymptotic that we consider is an adaptation to the geodesic flow on the moduli space of translation surfaces of estimates developed in \S~3.1 of \cite{velaniubiquity}.

\begin{theorem}
\label{thmLogLaws(Diophantine)}
Let $X$ be any translation surface. For any $\alpha\in(0,1)$ we have
$$
\dim\big(S_X(\alpha)\big)= 1-\alpha
\quad
\textrm{ and }
\quad
H^{1-\alpha}\big(S_X(\alpha)\big)=+\infty.
$$
\end{theorem}

\subsection{Recurrence in a rational billiard}
\label{SecRecurrenceBilliardsIntro}

Let $Q$ be a \emph{rational polygon}, that is a polygon in the plane whose angles are rational multiples of $\pi$. The linear part of reflections at the sides of $Q$ generate a finite group of linear isometries of the plane, so that any direction $\theta$ belongs to a finite equivalence class $[\theta]$, which is the orbit of $\theta$ under the action of reflections at sides of $Q$. For any class of directions $[\theta]$, the \emph{billiard flow} $\widehat{\phi}_{[\theta]}$ is well defined. A classical \emph{unfolding construction} of the rational polygon $Q$ defines a translation surface $X=X(Q)$, and for any class $[\theta]$ on $Q$ we have a well defined directional flow $\phi_\theta^t$ on $X$. Fix a class $[\theta]$ of directions on the rational polygon $Q$. The diophantine conditions developed in this paper have a relation with the \emph{recurrence rate function} 
$
\omega_{[\theta]}:Q\to[0,+\infty]
$, 
defined on points $p\in Q$ by  
$$
\omega_{[\theta]}(p):=
\liminf_{r\to0}\frac{\log\big(R_{[\theta]}(p,r)\big)}{-\log r},
$$
where for any $r>0$ the quantity
$
R_{[\theta]}(p,r):=\inf\{t>r\textrm{ ; }|\phi^t_{[\theta]}(p)-p|<r\}
$
denotes the \emph{return time} of $p$ at scale $r$. It is possible to see that $\omega_{[\theta]}(p)$ is defined for all those $p$ whose billiard trajectory never ends in a corner of $Q$, more details can be found in \S~\ref{SecRecurrenceRateFunction}. The function $p\mapsto\omega_{[\theta]}(p)$ is obviously invariant under the billiard flow $\phi_{[\theta]}$. Therefore, when $\phi_{[\theta]}$ is uniquely ergodic, $\omega_{[\theta]}(p)$ is constant for almost any $p\in Q$. By a theorem of Masur (see \cite{MasurUniqErgDir}), the Hausdorff dimension $\lambda=\lambda(Q)$ of the set of directions $\theta$ on $Q$ such that $\phi_{[\theta]}$ is not uniquely ergodic satisfies $0\leq \lambda\leq 1/2$. Fix $\tau\geq 2$ and define the set
$$
S_\tau:=
\left\{\theta\quad;\quad
\phi_{[\theta]}\textrm{ is uniquely ergodic and }
\omega_{[\theta]}(p)=\frac{1}{\tau-1}
\textrm{ for a.e. }p\in Q
\right\}.
$$

In a related setting (see \cite{KimMarmi}), D. H. Kim and S. Marmi prove that for almost any \emph{interval exchange transformation} $T$ the almost everywhere constant value of the recurrence rate function is equal to one. Theorem \ref{ThmRecurrenceBilliads} below is a counterpart of Theorem \ref{thmLogLaws(Diophantine)} for the dynamics of the billiard flow on a rational polygon. Closely related results appear in \cite{KimMarcheseMarmi}.

\begin{theorem}
\label{ThmRecurrenceBilliads}
Let $Q$ be a rational billiard and let $0\leq \lambda\leq 1/2$ be the dimension of the set of non uniquely ergodic directions on $Q$. Then for any $\tau$ with $2\leq\tau<2/\lambda$ we have 
$$
\dim\big(S_\tau\big)=\frac{2}{\tau}.
$$
\end{theorem}

The same result obviously holds for linear flows $\phi_\theta$ on a translation surface $X$. In \cite{CheungHubertMasur}, Y. Cheung, P. Hubert and H. Masur find polygons $Q$ for which $\lambda=0$, so that Theorem \ref{ThmRecurrenceBilliads} applies for any $\tau\geq 2$.

\subsection{Diophantine approximations for planar resonant sets}
\label{SecIntroResonantSet}

We consider diophantine conditions in terms of approximations of a given direction in $\RR^2$ by the directions of a countable set of vectors. Such approach is naturally formalized in polar coordinates, via the notion of \emph{planar resonant set}. We parametrize the set of lines in $\RR^2$ passing through the origin by the angle $\theta\in[-\pi/2,\pi/2[$ that they form with the vertical. Intuitively a \emph{planar resonant set} corresponds to a countable family of vectors $v\in\RR^2$, and for a given direction $\theta$ one considers those directions $\theta_v$ of vectors $v$ in the countable family such that the distance $|\theta-\theta_v|$ is small, compared to the length $|v|$ of $v$. Formal definitions are given below. Denote by $B(\theta,r)$ the open subinterval of $[-\pi/2,\pi/2[$ with length $2r$ centered at $\theta$. For any measurable subset $E\subset[-\pi/2,\pi/2[$ denote by $|E|$ its Lebesgue measure.

\medskip

A \emph{planar resonant set} is the datum $(\cR,l)$, where $\cR$ is a countable subset 
$
\cR\subset[-\pi/2,\pi/2[
$ 
and $l:\cR\to\RR_+$ is a positive function, such that for any $L>0$ the set
$
\{\theta\in\cR\textrm{ ; }l(\theta)<L\}
$
is finite. Given a real number $K>1$, we often consider the partition of $\cR$ into subsets
\begin{eqnarray*}
&&
\cR(K,n):=
\{\theta\in\cR
\textrm{ ; }
K^{n-1}<l(\theta)\leq K^n \}
\textrm{ for }
n\geq 1,
\\
&&
\cR(K,0):=
\{\theta\in\cR
\textrm{ ; }
l(\theta)\leq 1\}.
\end{eqnarray*}

An \emph{approximation function} is a decreasing function  $\psi:\RR_+\to\RR_+$. The set of directions in $[-\pi/2,\pi/2[$ which are \emph{well approximable} by elements in $\cR$ with respect to $\psi$ is
$$
W(\cR,\psi):=
\bigcap_{L>0}
\bigcup_{l(\theta)>L}
B\bigg(\theta,\psi\big(l(\theta)\big)\bigg).
$$
Given  $\epsilon>0$, the set of points in $[-\pi/2,\pi/2[$ which are \emph{$\epsilon$-badly approximable} with respect to $\cR$ is
$$
\bad(\cR,\epsilon):=
\left[\frac{-\pi}{2},\frac{\pi}{2}\right[
\setminus
\bigcap_{L>0}
\bigcup_{l(\theta)>L}
B\left(\theta,\frac{\epsilon^2}{l(\theta)^2}\right).
$$

In the following we consider subintervals $I\subset[-\pi/2,\pi/2[$ and we refer to them simply as \emph{intervals}. We introduce the following metric properties for planar resonant sets.

\begin{definition}
\label{DefMetricPropertiesResonantSets}
Let $(\cR,l)$ be a planar resonant set.
\begin{description}
\item[QG]
The set $(\cR,l)$ has \emph{quadratic growth} if there exists a constant $M>0$ such that for any $R>0$ we have
\begin{equation}\label{EqDefQuadraticGrowth}
\sharp
\{\theta\in\cR
\textrm{ ; }
l(\theta)\leq R\}\leq M\cdot R^2.
\end{equation}
\item[IQG]
The set $(\cR,l)$ has \emph{isotropic quadratic growth} if there exists a constant $M>0$ such that for any interval $I$ and any $R>0$ with $R^2|I|\geq1$ we have
\begin{equation}\label{EqDefIsotropicQuadraticGrowth}
\sharp
\{\theta\in I\cap\cR
\textrm{ ; }
l(\theta)\leq R\}\leq M\cdot |I|\cdot R^2.
\end{equation}
\item[U]
The set $(\cR,l)$ satisfies \emph{ubiquity}, if for any $K>1$ which is big enough there exist $c_1>0$, $c_2>0$ and $a>0$ with 
$
\displaystyle{\frac{a}{c_1}=o(K^2)}
$ 
such that for any $n$ and any interval $I$ with
$$
|I|\geq \frac{c_2}{K^n}
$$
we have
\begin{equation}
\label{EqDefUbiquity}
\left|
I
\cap
\bigcup_{l(\theta)\leq K^n} 
B\left(\theta,\frac{a}{K^{2n}}\right)\right|
\geq c_1 |I|.
\end{equation}
\item[DIR]
The set $(\cR,l)$ satisfies the \emph{$(\epsilon,U,\tau)$-Dirichlet property} for $\epsilon>0$, $U>0$ and $1<\tau<0$ if there exist some $L_0>0$ such that for any $L\geq L_0$ and any interval $I$ with $|I|\geq 2U/L^2$ we have
\begin{equation}
\label{EqDefDirichletUbiquity}
\left|
I\cap
\bigcup_{l(\theta)\leq L}
B\left(\theta,\frac{\epsilon^2}{2l(\theta)^2}\right)
\right|
\geq
\tau|I|.
\end{equation}
\item[DEC]
Fix $0<\epsilon<1$ and $0<\tau<1$ and set $K:=1/\epsilon$. The set $(\cR,l)$ is \emph{$(\epsilon,\tau)$-decaying} if for any $n\geq 1$ and any interval $I$ with
\begin{equation}\label{EqDefDecayingAssumption}
|I|=\frac{1}{K^{2n}}
\quad \textrm{   and   } \quad
I
\cap
\bigcup_{j=0}^{n-1}
\bigcup_{\theta\in\cR(K,j)}
B\left(\theta,\frac{\epsilon^2}{l(\theta)\cdot K^j}\right)
=\emptyset
\end{equation}
we have
\begin{equation}\label{EqDefDecayingConclusion}
\left|
I\cap
\bigcup_{\theta\in\cR(K,n)}
B\left(\theta,\frac{2\epsilon^2}{l(\theta)\cdot K^n}\right)
\right|
\leq \tau \cdot |I|.
\end{equation}
Moreover there exists an interval $I_0$ satisfying Condition \eqref{EqDefDecayingAssumption} for $n=1$.
\end{description}
\end{definition}

\begin{remark}
The notion of ubiquity has already been deployed in several other works, starting from \cite{BeresnevichDickinsonVelani}. Here condition $a/c_1=o(K^2)$ is a technical assumption adapted to our simplified proof of Theorem \ref{thmabstractkhinchinjarnik} in the setting of planar resonant sets. In related settings, the upper bound of the Hausdorff dimension of badly approximable sets is proved with a property which is derived from some version of Dirichlet Theorem, that was first called \emph{Dirichlet property} in \cite{WeilSteffen}. Dirichlet property and Ubiquity are quite similar, indeed for translation surfaces they both follow from Proposition \ref{propdirichletvorobets}. We give two separate abstract definitions because ubiquity is a qualitative property, stated in terms of constants which do not appear in Theorem \ref{thmabstractkhinchinjarnik} below, while the constants in Dirichlet property also appear in the upper bound in Theorem \ref{TheoremAbstractJarnickInequality}. Finally, the name for $(\epsilon,\tau)$-Decaying was chosen because it states a property which is similar to that of \emph{absolutely decaying measures}, which were introduced in \cite{KleinbockLinderstraussWeiss} and proved to be a valuable concept for establishing lower bounds on Hausdorff-dimension of badly approximable sets (see also \S~3.2 and \S~6.5 in \cite{minskyweiss}).
\end{remark}

A \emph{dimension function} is a continuous increasing function $f:\RR_+\to\RR_+$ such that either $f(r)/r$ is decreasing with $\lim_{r\rightarrow0}f(r)/r=\infty$, like for example $f(r)=r^s$ with $0<s<1$, or $f$ is the identity $f(r)=r$. For a fixed subset $E\subset[-\pi/2,\pi/2[$ and for $\rho>0$, a {\em $ \rho $-cover of $E$} is a countable collection $\{B_i\} $ of intervals $B_i$ with length $|B_i|\leq\rho$ for each $i$ such that $E \subset \bigcup_i B_i $. Such a cover exists for every $\rho>0$. For a dimension function $f$ define
$$
H^f_\rho(E):=
\inf\sum_i f\big(|B_i|\big),
$$
where the infimum is taken over all $\rho$-covers of $E$. The {\it Hausdorff $f$-measure} $H^f(E)$ of $E$ with respect to the dimension function $f$ is defined by
$$
H^{f} (E) :=
\lim_{\rho\to0} H^f_\rho(E)
\; = \;
\sup_{\rho>0}H^f_\rho(E).
$$
For the dimension function $f(r)=r^s$ with $0<s\leq 1$, the measure $H^f $ is the usual $s$-dimensional Hausdorff measure $H^s$, which coincides with the Lebesgue measure of $[-\pi/2,\pi/2[$ for $s=1$. The Hausdorff dimension $\dim E$ of a set $E$ is defined by
$$
\dim \, F \, := \,
\inf \left\{ s : H^{s} (E) =0 \right\} =
\sup \left\{ s : H^{s} (E) = \infty \right\}.
$$

In terms of the metric properties introduced in Definition \ref{DefMetricPropertiesResonantSets} we establish the following two results on diophantine approximations for planar resonant sets.

\begin{theorem}
[Abstract Khinchin-Jarn\'ik, after \cite{velanimasstransfer}]
\label{thmabstractkhinchinjarnik}
Consider a planar resonant set $(\cR,l)$ with quadratic growth, an approximation function $\psi$ and a dimension function $f$ such that the function $l\mapsto lf\circ\psi(l)$ for $l>0$ is decreasing monotone.
\begin{enumerate}
\item
If
$
\displaystyle{\sum_{n=1}^{\infty}nf\big(\psi(n)\big)<\infty}
$
then we have
$
\displaystyle{H^f\big(W(\cR,\psi)\big)=0}
$.
\item
If
$
\displaystyle{\sum_{n=1}^{\infty}n f\big(\psi(n)\big)=\infty}
$
and if moreover $(\cR,l)$ is ubiquitous and has isotropic quadratic growth, then we have
$
\displaystyle{H^f\big(W(\cR,\psi)\big)=H^f\big([-\pi/2,\pi/2[\big)}
$.
\end{enumerate}
\end{theorem}

\begin{theorem}
\label{TheoremAbstractJarnickInequality}
Consider a planar resonant set $(\cR,l)$.
\begin{enumerate}
\item
If $(\cR,l)$ satisfies the $(\epsilon,U,\tau)$-Dirichlet property for $\epsilon>0$, $U\geq0$ and $1<\tau<0$ then we have $$
\dim\big(\bad(\cR,\epsilon)\big)
\leq
1-
\frac
{|\log(1-\tau)|}
{|\log(\epsilon^2/(8U))|}.
$$
\item
If $(\cR,l)$ is $(\epsilon,\tau)$-decaying with $\tau<1-\epsilon^{4/3}$, we have
$$
\dim
\big(\bad(\cR,\epsilon)\big)
\geq
1-\frac{|\log(1-\tau-\epsilon^{4/3})|}{4/3|\log(\epsilon)|}.
$$
\end{enumerate}
\end{theorem}

\emph{Note:} Condition $\tau<1-\epsilon^{4/3}$ is a technical assumption. Later on it will be trivially satisfied since for us $\tau=O(\epsilon^\beta)$ for some $\beta>0$.

\subsection{Planar resonant sets of translation surfaces}
\label{SecIntroHolonomyResonantSets}

Let $X$ be a translation surface with $\area(X)=1$ and let $m$ be the total multiplicity at conical singularities of $X$, that is
$$
m=2g-2+\sharp(\Sigma).
$$

If $\gamma$ is a saddle connection of $X$, denote by $\theta_\gamma$ its direction. It is well known that for a given direction $\theta$ there exist at most $4g-4$ saddle connections $\gamma_i$ such that 
$
\theta_{\gamma_i}=\theta
$
for any $i$. For a direction $\theta=\theta_\gamma$ of a saddle connection $\gamma$ let $\gamma^{min}(\theta)$ be the saddle connection parallel to $\gamma$ with minimal length. Define the planar resonant set $\cR^{sc}$ and the length function $l^{sc}:\cR^{sc}\to\RR_+$ by
\begin{eqnarray*}
&&
\cR^{sc}:=\{\theta=\theta_\gamma
\textrm{ ; }
\gamma
\textrm{ saddle connection of }
X\}
\\
&&
l^{sc}(\theta):=|\gamma^{min}(\theta)|.
\end{eqnarray*}

We consider also \emph{closed geodesics} $\sigma$ of $X$, and we denote $\theta_\sigma$ the direction of any such $\sigma$. Given any closed geodesic $\sigma$, there exists a family of closed geodesics which are parallel to $\sigma$ with the same length and the same orientation. A \emph{cylinder} for $X$ is a connected open set $C_\sigma$ foliated by such a family of parallel closed geodesics and maximal with this property. By maximality, it follows that the boundary of a cylinder $C_\sigma$ around a closed geodesic $\sigma$ is union of saddle connections parallel to $\sigma$. Any cylinder $C_\sigma\subset X$ defines a holonomy vector
$
\hol(C_\sigma) =\int_\sigma w
$,
which is also denoted by $\hol(\sigma)$. We need to restrict to cylinders whose euclidian area is bounded from below by a positive constant. Let $\theta$ be a direction such that there a closed geodesic $\sigma$ in direction 
$
\theta_\sigma=\theta
$ 
whose cylinder satisfies 
$
\area(C_\sigma)>1/m
$. 
Such $\sigma$ is not unique. If  
$
\{\sigma_1,\dots,\sigma_j\}
$ 
is the family of all parallel geodesics in direction $\theta$ with $\area(C_{\sigma_j})>1/m$, we denote by $\sigma^{min}(\theta)$ the shortest element in the family 
$
\{\sigma_1,\dots,\sigma_j\}
$. 
Finally, we define the planar resonant set $\cR^{cyl}$ and the length function 
$
l^{cyl}:\cR^{cyl}\to\RR_+
$ 
by
\begin{eqnarray*}
&&
\cR^{cyl}:=
\left\{\theta=\theta_\sigma
\textrm{ ; }
\sigma
\textrm{ closed geodesic for }
X
\textrm{ with }
\area(C_\sigma)>\frac{1}{m}
\right\}
\\
&&
l^{cyl}(\theta):=|\sigma^{min}(\theta)|.
\end{eqnarray*}
In this second case, in order to state results in the sharpest form, let us define the quantity
$$
\cyl(X)=\min\{l^{cyl}(\theta)\textrm{ ; }\theta\in\cR^{cyl}\}.
$$
In the following, when there is not risk of ambiguity, we will denote both $l^{sc}$ and $l^{cyl}$ simply by $l$. 

\medskip

For the sets $\cR^{sc}$ and $\cR^{cyl}$ we will obtain the diophantine condition stated in Theorem \ref{thmabstractkhinchinjarnik} and Theorem \ref{TheoremAbstractJarnickInequality}, provided that the required assumptions are satisfied, which is ensured by Theorem \ref{TheoremMetricResultsHolonomyResonantSets} below. In order to obtain all the consequences of the three statements combined, consider a direction $\theta\in\cR^{cyl}$ and let 
$
\sigma=\sigma^{min}(\theta)
$, 
so that $\theta_\sigma=\theta$, then let $C_\sigma$ be the corresponding cylinder. The boundary of 
$C_\sigma$ is union of saddle connections $\gamma$ in direction $\theta_\sigma$ with 
$
|\gamma|\leq |\sigma|
$. 
Therefore we have $\cR^{cyl}\subset\cR^{sc}$, moreover if $\iota:\cR^{cyl}\to\cR^{sc}$ denotes the inclusion, then for any $\theta\in\cR^{cyl}$ we have 
\begin{equation}
\label{EquationLengthSaddleConnectionsAndClosedGeodesics}
l^{sc}\big(\iota(\theta)\big)\leq l^{cyl}(\theta).
\end{equation}

It follows that for any approximation function $\psi$ and any $\epsilon>0$ we have
\begin{equation}
\label{EqInclusionResonantSets}
W(\cR^{cyl},\psi)\subset W(\cR^{sc},\psi)
\quad
\textrm{ and }
\quad
\bad(\cR^{sc},\epsilon)\subset \bad(\cR^{cyl},\epsilon).
\end{equation}

The quadratic growth for resonant sets arising from translation surfaces is established by a well-known result of Masur (see \cite{MasurQuadraticGrowth}). In a refined version by Eskin and Masur, namely  Theorem 5.4 in \cite{EskinMasurAsymptoticFormulas}, it is proved that for any translation surface $X$ with $\area(X)=1$ there exists a constant $M>1$ such that for any $L>1$ we have
$$
\frac{\sharp\{\theta_\sigma\in\cR^{cyl};l^{cyl}(\theta_\sigma)\leq L\}}{L^2}
\leq
\frac{\sharp\{\theta_\gamma\in\cR^{sc};l^{sc}(\theta_\gamma)\leq L\}}{L^2}<M.
$$
Moreover, given any compact subset $\cK\subset\cH^{(1)}$, the constant $M=M(X)$ can be chosen uniformly for all $X\in\cK$. In this paper, using previous results of Vorobets \cite{vorobets1}, Chaika \cite{ChaikaHomogeneousApproximationsTranslSurf} and Minsky-Weiss \cite{minskyweiss}, we prove further properties of holonomy resonant sets.

\begin{theorem}
\label{TheoremMetricResultsHolonomyResonantSets}
Let $X$ be a translation surface with $\area(X)=1$ and let $m$ be the total multiplicity at conical singularities of $X$.
\begin{enumerate}
\item
There are positive constants $M>1$, $r_0>0$ and $0<\beta\leq 1$, depending only on $m$, such that for any $\epsilon$ with 
$
0<\epsilon<\min\{r_0,\sys(X)\}
$ 
the set $\cR^{sc}$ is $(\epsilon,\tau)$-decaying with $\tau=M\cdot\epsilon^\beta$. In particular we have
$$
\beta=\frac{1}{3m-1}.
$$
Moreover, if $X$ is a Veech surface, the same result holds with $\beta=1$ and $r_0$ depending only on the closed orbit $\cM=\sltwor\cdot X$ of $X$.
\item
For any $\epsilon$ with $0<\epsilon<1$ the set $\cR^{sc}$ satisfies $(U,\epsilon,\tau)$-Dirichlet property in terms of the constants
$$
U=\frac{12}{m^2\epsilon^2}
\quad
\textrm{ and }
\quad
\tau=
\frac{m\epsilon^2}{\sqrt{48}}.
$$
\item
The set $\cR^{cyl}$ has isotropic quadratic growth in terms of the constant
$$
M:=m(m+1).
$$
\item
The set $\cR^{cyl}$ satisfies ubiquity. In particular, for any 
$
\displaystyle{K\geq\frac{\sqrt{2}}{\cyl(X)}\cdot 2^{2^{4m+1}}}
$, 
Equation \eqref{EqDefUbiquity} is satisfied with constants 
$$
c_1:=\frac{1}{2}
\quad
\textrm{ , }
\quad
c_2:=\frac{K}{2m\cyl(X)}
\quad
\textrm{ , }
\quad
a:=\sqrt{3}K.
$$
\end{enumerate}
\end{theorem}

\begin{remark}
\label{RemarkImmersionResonantSets}
Point (4) in Theorem~\ref{TheoremMetricResultsHolonomyResonantSets} and Equation~\eqref{EquationLengthSaddleConnectionsAndClosedGeodesics} implies directly that 
$\cR^{sc}$ satisfies ubiquity with the same constants as $\cR^{cyl}$. On the other hand, according to Lemma~\ref{LemmaIsotropicQuadraticGrowthFails} in Appendix \S~\ref{SectionIsotropicQuadraticGrowthFails} of this paper, if $X$ is a surface with $\sltwor$-orbit dense in $\cH^{(1)}$, then the set $\cR^{sc}(X)$ does not have isotropic quadratic growth. Moreover, with constructions appearing in \S~5.3 in \cite{athreyachaika}, it is possible to see that for such a surface isotropic quadratic growth fails also for the set of directions $\theta_\sigma$ of \emph{all} closed geodesics $\sigma$, i.e. directions of closed geodesics around any cylinder $C_\sigma$, without any positive lower bound on 
$\area(C_\sigma)$. After the preprint of this paper was available online, closely related results on counting the number of saddle connections in angular sectors where obtained in \cite{Dozier}.
\end{remark}

Theorem \ref{ThmJarnikInequalityTranslSurf} and Theorem \ref{ThmKhinchinJarnickTranslationSurfaces} respectively in \S~\ref{BoundedGeodesicsModuliSpaceProofs} and in \S~\ref{ChapterExcursionsGeodesicsParameterSpace} give explicit statements of some consequences of Theorem \ref{TheoremMetricResultsHolonomyResonantSets} above and of the abstract Theorems \ref{thmabstractkhinchinjarnik} and \ref{TheoremAbstractJarnickInequality}.

\subsection{Further comments and questions}

\subsubsection{Sharpest lower bound in Theorem \ref{thmDynamicalJarnickInequality}}
\label{SecCommentJarnikInequality}

Let $\TT^2:=\RR^2/\ZZ^2$ be the standard torus and $\bad(\epsilon)$ be the set of those $\alpha\in\RR$ such that $q|q\alpha-p|\geq\epsilon^2$ for all but finitely many $p/q$, one can see that  
$$
\bad^{dyn}(\TT^2,\sqrt{2}\cdot\epsilon)=\bad(\epsilon).
$$ 
In \cite{jarnik}, Jarn\'ik gave the first estimates on the dimension of $\bad(\epsilon)$. In \cite{Kurzweil}, Kurzweil proves that for any $\epsilon>0$ small enough, we have 
$$
1-\frac{99}{100}\cdot\epsilon^2
\leq
\dim\big(\bad(\epsilon)\big)
\leq
1-\frac{1}{4}\cdot\epsilon^2.
$$
In \cite{Hensley}, Hensley gives the asymptotic for $\dim\big(\bad(\epsilon)\big)$ up to the term of order $\epsilon^4$. In our case, at least for Veech surfaces, it would be interesting to determine if the lower bound in Theorem \ref{thmDynamicalJarnickInequality} can be improved to get $\beta=2$, as it happens for the very special surface $X=\TT^2$. Nevertheless the gap between the exponent in lower and upper bound does not seem to be removable with our techniques. Recently, in \cite{Simmons}, Simmons computed the first order asymptotic of the dimension of uniformly badly approximable matrices, showing that in this case there is no gap between the exponent in lower and upper bound. This was not evident in previous estimates by Weil in \cite{WeilSteffen} and by Broderick and Kleinbock in \cite{BroderickKleinbock}, even in the extremal case of minimal dimension, where matrices (or vectors, in case of \cite{WeilSteffen}) coincide with real numbers. While the techniques used in \cite{BroderickKleinbock} and in \cite{WeilSteffen} have a counterpart for translation surfaces, namely quantitative non-divergence and Schmidt games, it is not evident that the same is true for the ideas introduced in \cite{Simmons}.

\subsubsection{Limits of the general approach}

Theorem \ref{thmDynamicalKhinchin} and Theorem \ref{thmLogLaws(Diophantine)} are consequences of the metric properties for the resonant sets $\cR^{sc}$ and $\cR^{cyl}$ stated in Theorem \ref{TheoremMetricResultsHolonomyResonantSets} and of the general Theorem \ref{thmabstractkhinchinjarnik}. Although these results can be applied to any pair of approximation function $\psi$ and dimension function $f$ such that $f\circ\psi$ is not increasing, a dynamical estimate for the excursions of $-\log\sys(g_tr_\theta\cdot X)$ requires an explicit choice of $\psi$ and $f$. This is because a comparison between $\sys(g_tr_\theta\cdot X)$ and a given function of time $\psi(t)$ passes through a comparison between the length of a saddle connection $\gamma$ on the surface $X$ and the instant $t=t(\theta,\gamma)$ when such $\gamma$ becomes short on the deformed surface $g_tr_\theta\cdot X$. See \S~\ref{ChapterExcursionsGeodesicsParameterSpace}.

\subsubsection{Unique ergodicity and diophantine type}

Let
$
\lambda:=\dim\big(\textrm{NUE}(X)\big)
$ 
be the dimension of the set of directions $\theta$ on the surface $X$ such that the flow $\phi_\theta$ is not uniquely ergodic. For $\tau\geq 2$ let 
$
\cW(\tau):=
W\big(\cR^{cyl},\psi_\tau\big)
\setminus
\bigcup_{\tau'>\tau}
W\big(\cR^{sc},\psi_{\tau'}\big)
$, 
where $\psi_\tau$ denotes the approximation function $\psi_\tau(r):=r^\tau$. It is easy to see that  $\dim\big(\cW(\tau)\big)=2/\tau$. In order to remove the assumption $\tau<2/\lambda$ in Theorem \ref{ThmRecurrenceBilliads} (see \S~\ref{SecEndProofRecurrenceBilliards}) we ask if we have the strict inequality
$$
\dim
\big(\cW(\tau)\cap\textrm{NUE}(X)
\big)
<\frac{2}{\tau}?
$$
For $\tau=2$ the answer is affirmative and corresponds to the well known fact that 
$$
\dim\big(NUE(X)\big)\leq 1/2<1=\dim\big(\cW(2)\big).
$$

\subsection{Contents of this paper}

In \S~\ref{ChapterProofAbstractKhinchinJarnik} we prove Theorem \ref{thmabstractkhinchinjarnik}. The convergent case follows from a very simple covering argument, which we give in \S~\ref{SecProofConvergentCaseKhinchinJarnik}. In divergent case, Lebesgue and general Hausdorff measure $H^f$ are considered separately. The first case in treated in \S~\ref{SecProofDivergentCaseKhinchinJarnikLebesgue} using Lebesgue density points. The second case is more involved: some general techniques related to \emph{mass transference} are resumed in \S~\ref{SecMassDistrubution}, proofs are completed in \S~\ref{SecProofDivergentAbstractKhinchinJarnik}.

In \S~\ref{SecProofAbstractJarnikInequality} we prove Theorem \ref{TheoremAbstractJarnickInequality}. In \S~\ref{SecProofLowerBoundAbstractJarnikInequality} we prove the lower bound using Decaying property and the general tools from \S~\ref{SecMassDistrubution}. In \S~\ref{SecProofUpperBoundTheoremAbstactJarnikInequality} we prove the upper bound with a covering argument build using Dirichlet property.

In \S~\ref{ChapterProofMetricPropertiesHolonomyResonantSets} we prove Theorem \ref{TheoremMetricResultsHolonomyResonantSets}. The main tools are a Dirichlet Theorem for translation surfaces, namely Proposition \ref{propdirichletvorobets}, and a version of Margulis' \emph{non-divergence of horocycles} adapted to translation surfaces, namely Theorem \ref{TheoremMinskyWeiss}, which is due to Minsky-Weiss. 

In \S~\ref{BoundedGeodesicsModuliSpaceProofs} we prove Theorem \ref{thmDynamicalJarnickInequality}. As an intermediate step, we state and prove a version of the same result for $\bad(\cR^{sc},\epsilon)$, that is Theorem \ref{ThmJarnikInequalityTranslSurf}.

In \S~\ref{ChapterExcursionsGeodesicsParameterSpace} we prove Theorem \ref{thmDynamicalKhinchin} and Theorem \ref{thmLogLaws(Diophantine)}. As an intermediate step we state and prove a version of the abstract Theorem \ref{thmabstractkhinchinjarnik} for the sets $W(\cR^{sc},\psi)$ and $W(\cR^{cyl},\psi)$, namely Theorem \ref{ThmKhinchinJarnickTranslationSurfaces}.

In \S~\ref{ChapterRecurrencePhaseSpace} we prove Theorem \ref{ThmRecurrenceBilliads}.

In Appendix \S~\ref{SectionProofMinskyWeissForVeech} we give the proof of Corollary \ref{CorollaryMinskyWeissVeech}, which is a sharper version of Theorem \ref{TheoremMinskyWeiss} for the specific case of Veech surfaces. 

In Appendix \S~\ref{SectionIsotropicQuadraticGrowthFails} we prove that isotropic quadratic growth of number of saddle connections fails for a generic surface $X$.

\subsection{Acknowledgements}

We are grateful to the anonymous referee for comments and suggestions. We are also grateful to Jon Chaika, Barak Weiss and Sanju Velani for discussions inspiring this work, and to Dong Han Kim and Stefano Marmi for discussions on the rate of recurrence in rational billiards. Part of this work was completed during visits at Tel Aviv University and the Max Planck Institute in Bonn, and we thanks these institutions for their hospitality. Steffen Weil is supported by ERC starter grant DLGAPS 279893. Rodrigo Trevi\~no was partially supported by the NSF under Award No. DMS-1204008, BSF Grant 2010428, and ERC Starting Grant DLGAPS 279893.

\section{Hausdorff measure of $W(\cR,\psi)$: proof of Theorem \ref{thmabstractkhinchinjarnik}}
\label{ChapterProofAbstractKhinchinJarnik}

In this section we prove Theorem \ref{thmabstractkhinchinjarnik}. Some of the constructions developed here, that is the content of \S~\ref{SecMassDistrubution}, will be used also in the next section in the proof of Theorem \ref{TheoremAbstractJarnickInequality}. Statement (1) in Theorem \ref{thmabstractkhinchinjarnik}, that is the case when the series 
$
\sum_{n=1}^{\infty}nf\big(\psi(n)\big)
$ 
converges, is proved in \S~\ref{SecProofConvergentCaseKhinchinJarnik}. Statement (2), that is when the series
$
\sum_{n=1}^{\infty}nf\big(\psi(n)\big)
$ 
diverges, requires a more specific analysis. The case of Lebesgue measure is rather simple and is treated in \S~\ref{SecProofDivergentCaseKhinchinJarnikLebesgue} using Lebesgue density points. The case of Hausdorff measure is more involved: general tools are developed in \S~\ref{SecMassDistrubution}, then  the proof is completed in \S~\ref{SecProofDivergentAbstractKhinchinJarnik}. In all this section 
$
f:\RR_+\to\RR_+
$ 
is a dimension function and $\psi:\RR_+\to\RR_+$ is a positive function such that $l\mapsto lf\circ\psi(l)$ is decreasing monotone. Recall that for us intervals are always considered as subintervals $I\subset[-\pi/2,\pi/2[$. 

\subsection{Separation properties for planar resonant sets}
\label{SecSeparationPropertiesResonantSets}

In this subsection we develop separation properties for a given planar resonant set $(\cR,l)$ which satisfies ubiquity and has isotropic quadratic growth as in Definition \ref{DefMetricPropertiesResonantSets}, that is such that there exist a constant $M>1$, and for any $K>1$ big enough constants $c_1=c_1(K)>0$, $c_2=c_2(K)>0$ and $a=a(K)>0$ with $a/c_1=o(K^2)$ such that for any integer $n$ and any interval $I$ we have
\begin{eqnarray*}
&&
\left|
I
\cap
\bigcup_{j=1}^n 
\bigcup_{\theta\in\cR(K,j)} 
B\left(\theta,\frac{a}{K^{2n}}\right)\right|
\geq c_1 |I|
\quad
\textrm{ provided that }
\quad
|I|>\frac{c_2}{K^{n}}
\\
&&
\sharp
\left\{
\theta\in \cR\cap I
\textrm{ ; }
l(\theta)<K^n
\right\}
<
M\cdot |I|\cdot K^{2n}
\quad
\textrm{ provided that }
\quad
|I|>\frac{1}{K^{2n}}.
\end{eqnarray*}

Observe that since $a/c_1=o(K^2)$ then, modulo increasing $K>1$, we can choose a constant $b=b(K)$ with $b\geq a$ such that 
\begin{equation}
\label{EquationConditionConstantSeparationProperties(1)}
\frac{c_1}{8(a+b)}>\frac{M}{K^2}.
\end{equation}
For example, one can choose $a=b$. We use different names to stress that the two constants $a$ and $b$ play a different role. Once $K>1$ and $a=a(K),c_2=c_2(K)$ are fixed, observe that there exists $n_0=n_0(K)$ such that for any $n\geq n_0$ we have 
\begin{equation}
\label{EquationConditionConstantSeparationProperties(3)}
K^n\geq \frac{4a}{c_2}.
\end{equation}

For any $n$ and any interval $I$ introduce the set of directions
\begin{equation}
\label{EquationConditionConstantSeparationProperties(4)}
\cR(n,I):=
\left\{
\theta\in\cR(K,n)
\textrm{ ; }
B\big(\theta,\frac{a}{K^{2n}}\big)\subset I
\right\}.
\end{equation}

For a fixed $\epsilon>0$ we say that a subset $\cT\subset[-\pi/2,\pi/2[$ is \emph{$\epsilon$-separated} if $|\theta-\theta'|>\epsilon$ for any pair of different points $\theta$ and $\theta'$ of $\cT$. Such a set is necessarily finite. 

\begin{proposition}
\label{PropositionSeparatedPointsInsideInterval}
Let $\cR$ be a planar resonant set satisfying ubiquity and isotropic quadratic growth in terms of the constants above. Assume that Equation~\eqref{EquationConditionConstantSeparationProperties(1)} is satisfied. Assume that $n$ is big enough so that Equation~\eqref{EquationConditionConstantSeparationProperties(3)} is satisfied. Then for any interval $I$ such that 
$$
|I|>2\cdot\frac{c_2}{K^{n}}.
$$ 
there exists a
$
\displaystyle{\frac{b}{K^{2n}}}
$-separated
subset
$
\cT(n,I)\subset\cR(n,I)
$
with cardinality
$$
\sharp\cT(n,I)\geq
\frac{c_1}{8(a+b)}\cdot|I|\cdot K^{2n}.
$$
\end{proposition}

\begin{proof}
Let $I'\subset I$ be the subinterval of maximal size such that we have the implication
$$
B\big(\theta,\frac{a}{K^{2n}}\big)\cap I'\not=\emptyset
\quad
\Rightarrow
\quad
B\big(\theta,\frac{a}{K^{2n}}\big)\subset I.
$$ 
The definition of $I'$ implies 
$
|I'|\geq |I|-4a\cdot K^{-2n}
$. 
Since $|I|>2c_2\cdot K^{-n}$ then Equation~\eqref{EquationConditionConstantSeparationProperties(3)} implies $|I|>8a\cdot K^{-2n}$ and thus 
$$
|I'|\geq 
|I|-\frac{4a}{K^{2n}}\geq
\frac{|I|}{2}.
$$
In particular we have $|I'|>c_2\cdot K^{-n}$, so that we can apply ubiquity to $I'$. Consider the set
$$
\cU(n,I'):=
\left\{
\theta\in\cR
\textrm{ ; }
l(\theta)\leq K^n
\textrm{ ; }
B\big(\theta,\frac{a}{K^{2n}}\big)\cap I'\not=\emptyset
\right\}.
$$
We show that $\cU(n,I')$ contains a $bK^{-2n}$-separated subset $\cU^{sep}(n,I')$ with cardinality at least 
$
c_1\cdot|I|K^{2n}/4(a+b)
$. 
Fix $N\in\NN$ and suppose that $\theta_1,\dots,\theta_N$ are
$bK^{-2n}$-separated points of $\cU(n,I')$ and that $N$ is maximal with such property. It follows that for any $\theta\in \cU(n,I')$ there exists some $j$ with $1\leq j\leq N$ such that 
$
|\theta-\theta_j|<bK^{-2n}
$. 
Ubiquity of $\cR$ implies the claim observing that
$$
\frac{c_1|I|}{2}\leq c_1|I'|<
\bigg|
\bigcup_{\theta\in \cU(n,I')}
B\big(\theta,\frac{a}{K^{2n}}\big)\cap I'
\bigg|
\leq
2N\frac{a+b}{K^{2n}}.
$$
Moreover, since $\cU(n,I')\subset I$, then isotropic quadratic growth implies 
$$
\sharp 
\left\{\theta\in\cU(n,I')\textrm{ ; }l(\theta)\leq K^{n-1}\right\}
\leq
M|I|K^{2(n-1)}.
$$ 
Set 
$$
\cT(n,I):=\{\theta\in\cU^{sep}(n,I')\textrm{ ; }K^{n-1}<l(\theta)\leq K^{n}\}.
$$
We have $\cT(n,I)\subset\cR(n,I)$ and $\cT(n,I)$ is $bK^{-2n}$-separated by construction. Moreover the estimates above and Condition~\eqref{EquationConditionConstantSeparationProperties(1)} imply 
\begin{align*}
&
\sharp\cT(n,I)=
\sharp\cU^{sep}(n,I)-
\sharp\left\{\theta\in\cU(n,I')\textrm{ ; }l(\theta)\leq K^{n-1}\right\}
\geq
\\
&
\frac{c_1\cdot|I|}{4(a+b)}K^{2n}-M\cdot|I|K^{2(n-1)}=
\left(
\frac{c_1}{4(a+b)}-\frac{M}{K^2}
\right)
\cdot 
|I|K^{2n}
\geq
\frac{c_1}{8(a+b)}
\cdot 
|I|K^{2n}.
\end{align*}
\end{proof}

Recall that $a/c_1=o(K^2)$. Modulo taking $K$ bigger, and arguing as for Equaiton~\eqref{EquationConditionConstantSeparationProperties(1)}, assume that we have 
$
\displaystyle{\frac{c_1}{16(a+b)}> \frac{3M}{K^2}}
$ 
strictly. Then consider a constant $\delta>0$ small enough compared to $b$ in order to satisfy the condition
\begin{equation}
\label{EquationConditionConstantSeparationProperties(2)}
\frac{c_1}{16(a+b)}
\geq
\frac{\delta}{b}+\frac{3M}{K^2}.
\end{equation}
Observe that the condition above implies 
$
\displaystyle{\frac{c_1}{16(a+b)}>\frac{\delta}{b}}
$ 
and since $c_1<1$ then we must have also $\delta<1$.

\begin{corollary}
\label{cor1ss2abstractkhinchinjarnick}
Consider $n\in\NN$ which satisfies Equation~\eqref{EquationConditionConstantSeparationProperties(3)} and an interval $I$ such that $|I|>2c_2\cdot K^{-n}$. Let $\cI:=\bigcup_{i=1}^{N}I_i$ be the union of $N$ subintervals $I_i$ of $I$ such that $|\cI|<\delta|I|$ and $N<M|I|K^{2(n-1)}$. Then the
$
\displaystyle{\frac{b}{K^{2n}}}
$-separated
set $\cT(n,I)$ in Proposition \ref{PropositionSeparatedPointsInsideInterval} contains at least
$
\displaystyle{\frac{c_1\cdot|I|}{16(a+b)}K^{2n}}
$
points $\theta$ such that
$$
\cI\cap B\left(\theta,\frac{b}{K^{2n}}\right)=\emptyset.
$$
\end{corollary}

\begin{proof}
Set $\rho:=bK^{-2n}$ and observe that any subinterval $I_i$ contains at most $|I_i|/\rho+1$ points which are $\rho$-separated, so that the union $\cI$ contains at most $|\cI|/\rho+N$ points which are $\rho$-separated. Then the $\rho$-neighborhood of $\cI$ contains at most
$$
\frac{|\cI|}{\rho}+3N
\leq
\left(\frac{\delta}{b}+\frac{3M}{K^2}\right)|I|\cdot K^{2n}
$$
points which are $\rho$-separated. The Corollary follows observing that Proposition \ref{PropositionSeparatedPointsInsideInterval} and Condition~\eqref{EquationConditionConstantSeparationProperties(2)} imply
\begin{align*}
&
\sharp\cT(n,I)-
\sharp\left\{\theta\in\cT(n,I)\textrm{ ; }B(\theta,bK^{-2n})\cap \cI\not=\emptyset\right\}
\geq
\\
&
\left(\frac{c_1}{8(a+b)}-\frac{\delta}{b}-\frac{3M}{K^2}\right)|I|\cdot K^{2n}
\geq
\frac{c_1}{16(a+b)}\cdot|I|\cdot K^{2n}.
\end{align*}
\end{proof}

\subsection{Proof of convergent case in Theorem \ref{thmabstractkhinchinjarnik}} 
\label{SecProofConvergentCaseKhinchinJarnik}

The proof follows from a simple covering argument, that we give below for the sake of completeness.

\medskip

Fix $\epsilon>0$ and $\rho>0$. Since $\psi(l)\to0$ and $l\mapsto lf\circ\psi(l)$ is decreasing monotone for $l\to\infty$, for any $N$ big enough we obtain a $\rho$-covering of $W(\cR,\psi)$ by taking the union
$$
\bigcup_{n=N}^{\infty}
\bigcup_{\theta\in\cR(K,n)}
B\big(\theta,\psi(K^{n-1})\big).
$$
The summability of
$
\sum_{n=1}^{\infty}nf\big(\psi(n)\big)
$
is equivalent to the summability of
$
\sum_{n=1}^{\infty}K^{2n}f\big(\psi(K^n)\big)
$,
thus modulo increasing $N$ one also has
$
\sum_{n=N}^{\infty}K^{2n}f\big(\psi(K^{n-1})\big)<\epsilon
$.
Hence
$$
H^f\big(W(\cR,\psi)\big)<2\epsilon,
$$
and since $\epsilon$ is arbitrarily small we get
$
H^f\big(W(\cR,\psi)\big)=0
$.

\subsection{Proof of divergent case in Theorem \ref{thmabstractkhinchinjarnik} for Lebesgue measure}
\label{SecProofDivergentCaseKhinchinJarnikLebesgue}

We closely follow the argument of \cite{boshernitzanchaika}, pages 7 and 8. In the proof, an interval $I$ is fixed once and for all, around some Lebesgue density point. It is possible to see that in such situation the argument only uses ubiquity and quadratic growth for the resonant set $\cR$, but not isotropic quadratic growth (see \cite{boshernitzanchaika} for details). Our proof assumes isotropic quadratic growth in order to stay in the setting developed in \S~\ref{SecSeparationPropertiesResonantSets}. Isotropic quadratic growth will be strictly necessary in the case of Hausdorff measure, where the construction of some Cantor set will require to consider intervals at smaller and smaller scale.

\medskip

Let $\cR$ be a planar resonant set satisfying ubiquity and isotropic quadratic growth. Let $M$, $K$, $a$, $c_1$ and $c_2$ be constants as in Definition~\ref{DefMetricPropertiesResonantSets}. As in \S~\ref{SecSeparationPropertiesResonantSets}, increase $K$ if necessary and introduce constants $b$ and $\delta$ such that Equation~\eqref{EquationConditionConstantSeparationProperties(2)} is satisfied. 
 
Observe that if $\psi'(l)\leq\psi(l)$ for any $l>0$ then we have
$
W(\cR,\psi')\subset W(\cR,\psi)
$.
Hence it is enough to prove the statement for an approximating sequence satisfying
$$
\psi(l)=\min
\left\{
\psi(K^n),\frac{b}{K^{2n}}
\right\}
\textrm{ for }
K^{n-1}< l\leq K^n.
$$

Fix an interval $I$. Let $N$ be an integer such that any $n\geq N$ satisfies Equation~\eqref{EquationConditionConstantSeparationProperties(3)} and moreover we have also 
$
|I|>2c_2\cdot K^{-n}
$, 
so that Proposition~\ref{PropositionSeparatedPointsInsideInterval} and Corollary~\ref{cor1ss2abstractkhinchinjarnick} can be applied. Then fix any $m>N$ and, recalling the sets 
$\cR(n,I)$ defined in Equation~\eqref{EquationConditionConstantSeparationProperties(4)}, consider the set
$$
\cI(N,m):=
I\cap\bigcup_{n=N}^{m}\bigcup_{\theta\in\cR(n,I)}
B\big(\theta,\psi(K^n)\big).
$$

\begin{lemma}\label{lem1ss3abstractkhinchinjarnick}
There exists $m>N$ such that
$
|\cI(N,m)|\geq\delta|I|
$.
\end{lemma}

\begin{proof}
Fix $m\geq N+1$ and set $\cI:=\cI(N,m-1)$, which is the union of at most $M|I|K^{2(m-1)}$ subintervals of $I$, according to isotropic quadratic growth of $\cR$. If $|\cI|\geq\delta|I|$ then we are done. If $|\cI|<\delta|I|$ then Corollary \ref{cor1ss2abstractkhinchinjarnick} implies that $\cR(m,I)$ contains at least
$
\displaystyle{\frac{c_1\cdot|I|}{16(a+b)}K^{2m}}
$
points $\theta$ which are $bK^{-2m}$-separated and such that
$
B(\theta,bK^{-2m}))\cap\cI=\emptyset
$.
This implies
$
B(\theta,\psi(K^m)\cap\cI=\emptyset
$,
since
$
\psi(K^m)\leq bK^{-2m}
$.
It follows that
$$
\big|\cI(N,m)\big|
\geq
\big|\cI(N,m-1)\big|+
\frac{c_1}{16(a+b)}|I|K^{2m}\psi(K^m).
$$
The Lemma follows from the divergence assumption of
$
\sum K^{2m}\psi(K^m)
$.
\end{proof}

The divergent case in Theorem \ref{thmabstractkhinchinjarnik} for Lebesgue measure follows observing that, according to the Lemma above we have
$
\big|\bigcup_{n=N+1}^\infty\cI(N,m)\big|>\delta|I|
$
for any $N$ big enough, and thus
$$
\big|W(\cR,\psi)\cap I\big|=
\big|
\bigcap_{N\in\NN}
\bigcup_{n=N+1}^\infty
\cI(N,m)
\big|
\geq
\delta|I|.
$$
The estimate above holds for any small interval $I$, therefore the complement of $W(\cR,\psi)$ has no density points, that is $W(\cR,\psi)$ has full measure.

\subsection{Mass distribution $\mu_f$ on the Cantor Set $\KK$}
\label{SecMassDistrubution}

We consider a dimension function with $f(r)/r\to\infty$ for $r\to0$. Given a Cantor set 
$
\KK\subset[-\pi/2,\pi/2[
$ 
we describe a classical construction of a probability measure $\mu_f$ supported on $\KK$ which is somehow natural with respect to the dimension function $f$. For convenience of notation we write simply $\mu$ instead of $\mu_f$.

\medskip

For any positive integer $n$ we define a family $\cK(n)$ of subintervals of $[-\pi/2,\pi/2[$ which are disjoint in their interior. The $n$-th \emph{level} of the Cantor set is given by
$$
\KK(n):=\bigsqcup_{B\in\cK(n)}B.
$$
The families $\cK(n)$ are chosen so that $\KK(n)\subset\KK(n-1)$ for any $n>1$, then the Cantor set is defined by
$
\KK=\bigcap_{n=1}^{\infty}\KK(n)
$.
For any $B\in\cK(1)$ we set
$$
\mu(B):=
\frac{f\big(|B|\big)}
{\sum_{B'\in\cK(1)}f\big(|B'|\big)}
$$
For any $n>1$, any $B_0\in\cK(n-1)$ call $\cK(n,B_0)$ the subfamily of those balls $B\in\cK(n)$ with $B\subset B_0$, then for any $B\in\cK(n,B_0)$ set
$$
\mu(B):=
\frac{f\big(|B|\big)}
{\sum_{B'\in\cK(n,B_0)}f\big(|B'|\big)}\mu(B_0).
$$
The construction of the measure $\mu$ on $\KK$ is completed by the following Lemma, which corresponds to Proposition 1.7 in \cite{falconer}.

\begin{lemma}
\label{lem2ss4abstractkhinchinjarnick}
The function
$
\mu:\bigcup_{n\in\NN}\cK(n)\to\RR_+
$
extends to a Borel probability measure supported on $\KK$ setting
$$
\mu(E)=\mu(E\cap\KK):=
\inf \sum_{B}\mu(B),
$$
where $E$ is any Borel subset of $\RR$ and the $\inf$ is taken over all coverings of $E$ with balls $B$ in $\bigcup_{n\in\NN}\cK(n)$.
\end{lemma}

The following Lemma gives a classical method to obtain lower bounds for $H^f$ of a set $\KK$. A version for the Hausdorff measures $H^s$ corresponding to the dimension function $f_s(x)=x^s$ with $0<s<1$ can be found in \S~4.2 in \cite{falconer}.

\begin{lemma}
[Mass Transference principle]
\label{lem1ss4abstractkhinchinjarnick}
Let $\mu$ be a probability measure supported on a subset $\KK$ of $\RR$. Suppose that there are constants $\eta>0$ and $\rho_0>0$ such that
\begin{equation}
\label{eqmeasureoncantorset}
\mu(B)\leq \frac{f(|B|)}{\eta}
\end{equation}
for any ball with radius $\rho<\rho_0$. Then we have
$
H^f(E)\geq\eta\cdot\mu(E)
$
for any subset $E$ of $\KK$.
\end{lemma}

\begin{proof}
For any $\rho$-cover $\{B_i\}$ of $E$ with $\rho<\rho_0$ we have
$$
\mu(E)=
\mu\big(\bigcup B_i\big)\leq
\sum\mu(B_i)\leq
\eta^{-1}\sum f\big(|B_i|\big).
$$
\end{proof}

\begin{remark}
Fix $n\geq1$ and an interval $B_0\in\cK(n-1)$, where $B_0:=[-\pi/2,\pi/2[$ for $n=1$. Observe that any subinterval $B\in\cK(B_0,n)$ satisfies Condition \eqref{eqmeasureoncantorset} if and only if
\begin{equation}
\label{eqmeasureoncantorset(alternative)}
\sum_{B'\in\cK(n,B_0)}f\big(|B'|\big)>\eta\mu(B_0).
\end{equation}
\end{remark}

The following Lemma follows by an easy computation, which is left to the reader, and gives a criterion to get Condition \eqref{eqmeasureoncantorset} for the intervals in $\bigcup_{n\in\NN}\cK(n)$.

\begin{lemma}
\label{lemmaaereo}
Let $B_0$ be any interval and $\cK$ be a finite family of subintervals $B\subset B_0$ which are pairwise disjoint. Fix constants $0<\delta<1$ and $C>0$. Assume that we have
\begin{equation}
\label{eqlemmaaereo}
\sum_{B\in\cK}|B|>\delta|B_0|.
\end{equation}
and moreover that for any $B\in\cK$ we have also 
$$
\frac{f\big(|B|\big)}{|B|}>\frac{C}{\delta|B_0|}
$$
Then we have 
$
\displaystyle{
\sum_{B\in\cK}f\big(|B|\big)>C
}$.
\end{lemma}

\subsection{Proof of divergent case in Theorem \ref{thmabstractkhinchinjarnik} for Hausdorff measure}
\label{SecProofDivergentAbstractKhinchinJarnik}

We basically follow \cite{velanimasstransfer}. Consider an approximation function $\psi$ such that
$$
\sum_{n=1}^{\infty} K^{2n}f\big(\psi(K^n)\big)=\infty.
$$

Let $(\cR,l)$ be a planar resonant set satisfying ubiquity and isotropic quadratic growth, in terms of the constants $M$, $K$, $a$, $c_1$ and $c_2$ introduced in Definition \ref{DefMetricPropertiesResonantSets}. Fix constants $b>0$ and $\delta>0$ as in \S~\ref{SecSeparationPropertiesResonantSets} and modulo increasing $K$ assume that Condition~\eqref{EquationConditionConstantSeparationProperties(2)} is satisfied, so that Proposition \ref{PropositionSeparatedPointsInsideInterval} and Corollary \ref{cor1ss2abstractkhinchinjarnick} can be applied. In order to simplify the notation, set
$$
c:=\frac{c_1}{16(a+b)}.
$$
Finally, recall from \S~\ref{SecSeparationPropertiesResonantSets} that for any subinterval 
$
B_0\subset[-\pi/2,\pi/2[
$, 
Equation~\eqref{EquationConditionConstantSeparationProperties(4)} defines $\cR(l,B_0)$ as the set of those $\theta\in\cR(K,l)$ such that
$$
B\left(\theta,\frac{a}{K^{2l}}\right)\subset B_0.
$$

\begin{proposition}
[Local construction of measure $\mu$]
\label{propLocalConstructionKhinchinJarnick}
Fix a subinterval $B_0\subset[-\pi/2,\pi/2[$ and a constant $C>0$. 

There exist positive integers $m(B_0)$ and $l(B_0)$ and a finite family $\cK(B_0)$ of disjoint subintervals $B\subset B_0$ of the form
$
B=B\big(\theta,\psi(K^{l})\big)
$
for some $m(B_0)<l\leq l(B_0)$ and some $\theta\in\cR(l,B_0)$ which are pairwise disjoint and such that
\begin{equation}
\label{eq1LocalConstructionKhinchinJarnick}
\sum_{B\in\cK(B_0)}f\big(|B|\big)>C.
\end{equation}

Moreover, there exists an universal constant $\Delta>1$ not depending on $B_0$ such that for any subinterval $I\subset B_0$, denoting $\cK(B_0,I)$ the family of those balls $B\in\cK(B_0)$ with $B\cap I\not=\emptyset$, we have
\begin{equation}
\label{eq2LocalConstructionKhinchinJarnick}
\sum_{B\in \cK(B_0,I)}f\big(|B|\big)<
\Delta\frac{|I|}{|B|}\sum_{B\in \cK(B_0)}f\big(|B|\big).
\end{equation}
\end{proposition}

\begin{proof}
We give first a sketch of the proof. The first step in the proof is to define the integer $m(B_0)$. Once $m(B_0)$ is defined, for $l>m(B_0)$ we consider families $\cK(B_0,l)$ made of disjoint intervals $B$ of the form  
$
B=B\big(\theta,\psi(K^{l})\big)
$
for $\theta\in\cR(l,B_0)$, so that the sum in Condition \eqref{eq1LocalConstructionKhinchinJarnick} takes the form 
$$
\sum_{j=m(B_0)+1}^l
\sum_{B\in\cK(B_0,j)}f\big(|B|\big).
$$
The Lebesgue measure of such families of intervals is big enough if we have
$$
\sum_{j=m(B_0)+1}^l
\sum_{B\in\cK(B_0,j)}|B|
\geq
\delta |B_0|.
$$ 
If the last condition is satisfied we set $l(B_0):=l$ and Condition \eqref{eq1LocalConstructionKhinchinJarnick} follows from Lemma \ref{lemmaaereo}. Otherwise Corollary \ref{cor1ss2abstractkhinchinjarnick} tells us that there exists an extra family $\cK(B_0,l+1)$ containing at least $c|B_0|K^{2(l+1)}$ intervals which are disjoint from all the previous ones. Adding this $(l+1)$-th term the sum in Condition \eqref{eq1LocalConstructionKhinchinJarnick} increases by 
$$
\sum_{B\in\cK(B_0,l+1)}f\big(|B|\big)
\geq
c|B_0|\cdot K^{2(l+1)}\psi(K^{l+1}).
$$
The latter is the $(l+1)$-th term of a divergent series, thus Condition \eqref{eq1LocalConstructionKhinchinJarnick} is eventually satisfied. Then we define $l(B_0)$ as the last term in the finite sum. The second part of the statement will follow easily. We now start the formal proof of the Proposition. 

\smallskip

\emph{Definition of $m(B_0)$}. In order to apply Proposition \ref{PropositionSeparatedPointsInsideInterval} and Corollary \ref{cor1ss2abstractkhinchinjarnick}, fix $m=m(B_0)$ such that for any $l>m(B_0)$ Condition~\eqref{EquationConditionConstantSeparationProperties(3)} is satisfied and moreover we have
$$
|B_0|\geq \frac{2\cdot c_2}{K^l}.
$$
Moreover recall that $\psi(K^l)\to0$ for $l\to+\infty$ and that $f(r)/r\to\infty$ for $r\to0$. Therefore, modulo increasing $m(B_0)$ we can assume also that for any $l> m(B_0)$ and any interval of the form
$
B=B\big(\theta,\psi(K^{l})\big)
$ 
with $\theta\in\cR(l,B_0)$ we have
$$
\frac{f\big(|B|\big)}{|B|}>
\frac{C}{\delta|B_0|}.
$$

\smallskip

\emph{Definition of $l(B_0)$}. The family $\cK(B_0)$ is defined as union of subfamilies $\cK(B_0,l)$ inductively defined for $l=m+1,\dots,l(B_0)$. The integer $l(B_0)$ is defined as the last step of the construction, when the required properties of $\cK(B_0)$ are satisfied. The inductive procedure is described below.

\smallskip

\emph{Initial step}.
According to Proposition \ref{PropositionSeparatedPointsInsideInterval} there exist a subset
$
\cT(B_0,m+1)\subset\cR(m+1,B_0)
$ 
which is $bK^{-2(m+1)}$-separated and has cardinality
$
\sharp\big(\cT(B_0,m+1)\big)\geq 2c|B_0|K^{2m+2}
$.
Define $\cK(B_0,m+1)$ as the family of balls
$
B\big(\theta,\psi(K^{m+1})\big)
$
for $\theta\in \cT(B_0,m+1)$ and set
$$
\KK(B_0,m+1):=
\bigsqcup_{B\in\cK(B_0,m+1)}B.
$$

\smallskip

\emph{Inductive step.}
Assume inductively that for $l\geq m+1$ the families
$
\cK(B_0,m+1),\dots,\cK(B_0,l)
$
are defined, where for any $j=m+1,\dots,l$ any $B\in\cK(B_0,j)$ is an interval of the form 
$
B=B\big(\theta,\psi(K^l)\big)
$ 
for some $\theta\in\cR(l,B_0)$, and assume inductively also that all the intervals in
$
\bigcup_{j=m+1}^l\cK(B_0,j)
$
are disjoint. Then set
$$
\KK(B_0,l):=
\bigsqcup_{j=m+1}^{l}
\bigsqcup_{B\in\cK(B_0,j)}B.
$$
Since intervals in $\KK(B_0,l)$ are balls centered at points 
$
\theta\in\bigcup_{j=m+1}^l\cR(j,B_0)
$, 
then by isotropic quadratic growth these intervals are at most $M|B_0|K^{2l}$.

\begin{enumerate}
\item
If the family
$
\bigcup_{j=m+1}^l\cK(B_0,j)
$
satisfies Condition \eqref{eqlemmaaereo} then we set $l(B_0):=l$ and
$$
\cK(B_0):=\bigcup_{j=m+1}^l\cK(B_0,j).
$$
Lemma \ref{lemmaaereo} implies that Condition \eqref{eq1LocalConstructionKhinchinJarnick} is satisfied too, and the first part of the Proposition is proved. 
\item
If Condition \eqref{eqlemmaaereo} is not satisfied, observe that $\KK(B_0,l)$ is a union of at most $M|B_0|K^{2l}$ disjoint intervals with $|\KK(B_0,l)|<\delta|B_0|$. Then according to Corollary \ref{cor1ss2abstractkhinchinjarnick} there exists a $bK^{-2(l+1)}$-separated subset $\cT(B_0,l+1)$ of $\cR(l+1,B_0)$ with cardinality
$$
\sharp\big(\cT(B_0,l+1)\big)\geq c|B_0|K^{2(l+1)}
$$
such that for any $\theta\in \cT(B_0,l+1)$ we have
$$
B\big(\theta,bK^{-2(l+1)}\big)
\subset
B_0\setminus\KK(B_0,l).
$$
Then define $\cK(B_0,l+1)$ as the family of balls
$
B\big(\theta,\psi(K^{l+1})\big)
$
for $\theta\in \cT(B_0,l+1)
$
and observe that
$
\bigcup_{j=m+1}^{l+1}\cK(B_0,j)
$
is a family of disjoint balls.
\end{enumerate}

\emph{The inductive procedure eventually stops}.
Repeat the analysis in the inductive step replacing $l$ by $l+1$. Eventually at least one of the following two conditions is satisfied.
\begin{enumerate}
\item
The family
$
\bigcup_{j=m+1}^{l}\cK(B_0,j)
$ 
eventually satisfies Condition \eqref{eqlemmaaereo} and thus also Condition \eqref{eq1LocalConstructionKhinchinJarnick}, as in point (1) of the inductive step. The construction of $\cK(B_0)$ is therefore complete. 
\item
Otherwise the family
$
\bigcup_{j=m+1}^{l}\cK(B_0,j)
$ 
eventually satisfies directly Condition \eqref{eq1LocalConstructionKhinchinJarnick}. Indeed, reasoning as in point (2) of the inductive step, we add an extra subfamily $\cK(B_0,l+1)$. Since each $\cK(B_0,j)$ contains at least $c|B_0|K^{2j}$ subintervals of size $\psi(K^j)$, and since all the intervals in all the families $\cK(B_0,j)$ are mutually disjoint, we have
$$
\sum_{j=m+1}^{l}
\sum_{B\in\cK(B_0,l)}f(|B|)\geq
c|B_0|\cdot
\sum_{j=m+1}^{l}K^{2j} f\big(\psi(K^j)\big),
$$
and Condition \eqref{eq1LocalConstructionKhinchinJarnick} follows from the divergence of
$
\sum_{n=1}^{\infty}K^{2n}f\big(\psi(K^n)\big)
$.
\end{enumerate}
In both cases we obtain a family $\cK(B_0)$ satisfying Equation~\eqref{eq1LocalConstructionKhinchinJarnick}. We define $l(B_0)$ as the first $l\geq m+1$ such that this is true. The first part of the Proposition is proved.

\smallskip

\emph{Second part of the statement.} In order to finish the proof, fix any subinterval $I\subset B_0$. Since
$
\sharp\big(\cT(B_0,l)\big)\geq c|B_0|K^{2l}
$
for any integer $l$ with $m(B_0)< l\leq l(B_0)$, then we have
$$
S(B_0,l):=
\sum_{B\in\cK(B_0,l)}f\big(|B|\big)\geq
cf\big(\psi(K^l)\big)
\cdot|B_0|\cdot K^{2l}.
$$
On the other hand, the points $\theta$ in $\cT(B_0,l)$ are $bK^{-2l}$-separated and thus, denoting $\cK(B_0,I,l)$ the set of those balls $B\in\cK(B_0,l)$ with $B\cap I\not=\emptyset$, we have
$$
S(I,l):=
\sum_{B\in\cK(B_0,I,l)}f\big(|B|\big)<
\frac{1}{b}f\big(\psi(K^l)\big)\cdot|I|\cdot K^{2l}.
$$
Equation \eqref{eq2LocalConstructionKhinchinJarnick} follows with $\Delta:=(bc)^{-1}$ observing that
$$
\sum_{B\in \cK(B_0,I)}f\big(|B|\big)=
\sum_{l=m(B_0)+1}^{l(B_0)}S(I,l)<
\frac{1}{bc}\frac{|I|}{|B_0|}
\sum_{l=m(B_0)+1}^{l(B_0)}S(B_0,l)=
\frac{1}{bc}\frac{|I|}{|B_0|}
\sum_{B\in \cK(B_0)}f\big(|B|\big).
$$
\end{proof}

\subsubsection{Construction of the Cantor set $\KK$ with probability measure $\mu$}

Fix any $\eta>0$. Recall the notation of \S~\ref{SecMassDistrubution}, where for any $n$ we consider a family $\cK(n)$ of disjoint subintervals of $[-\pi/2,\pi/2[$, which defines the $n$-th level $\KK(n)$ of a Cantor set $\KK$, so that $\KK(n+1)\subset\KK(n)$ for any $n$ and that $\KK=\bigcap_{n=1}^\infty\KK(n)$. The inductive construction of the levels $\KK(n)$ of $\KK$ is given below.

\begin{description}
\item[First level] 
set $B_0:=[-\pi/2,\pi/2[$ and $C:=\eta$. Let 
$
\KK(1):=\bigsqcup_{B\in\cK(1)}B
$, 
where 
$$
\cK(1):=\cK\big(B_0=[-\pi/2,\pi/2[,C=\eta\big)
$$ 
is the family of disjoint interval corresponding to the interval $B_0=[-\pi/2,\pi/2[$ and to the constant $C=\eta$ which is constructed in Proposition \ref{propLocalConstructionKhinchinJarnick}. Observe that any interval $B\in\cK(1)$ satisfies Condition \eqref{eqmeasureoncantorset}, according to Equation \eqref{eq1LocalConstructionKhinchinJarnick} and Equation \eqref{eqmeasureoncantorset(alternative)}.
\item[General level]
suppose inductively that the levels $\KK(1),\dots,\KK(n-1)$ are defined, or equivalently the families $\cK(1),\dots,\cK(n-1)$. Fix any $B_0\in\cK(n-1)$, where
$
B_0=B\big(\theta,\psi(K^m)\big)
$
for some $m\geq n$ and some $\theta\in\cR(K,m)$. Set $C:=\eta\mu(B_0)$. Let 
$$
\cK(n,B_0):=\cK\big(B_0,C=\eta\mu(B_0)\big)
$$
be the family of disjoint intervals provided by Proposition \ref{propLocalConstructionKhinchinJarnick} corresponding to the interval $B_0\in\cK(n-1)$ and to the constant $C=\eta\mu(B_0)$. According to Equation \eqref{eq1LocalConstructionKhinchinJarnick} and Equation \eqref{eqmeasureoncantorset(alternative)}, any interval $B\in\cK(n,B_0)$ satisfies Condition \eqref{eqmeasureoncantorset}. Finally define the $n$-th level and family of intervals by
\begin{eqnarray*}
&&
\cK(n):=\bigcup_{B_0\in\cK(n-1)}\cK(B_0,n)
\\
&&
\KK(n):=\bigsqcup_{B_0\in\cK(n-1)}\bigsqcup_{B\in\cK(B_0,n)}B.
\end{eqnarray*}
\end{description}

Observe that for any $n$, the intervals in $\cK(n)$ are pairwise disjoint and any $B\in\cK(n)$ is of the form $B=B\big(\theta,\psi(K^l)\big)$ for some $l\geq n$ and some $\theta\in\cR(K,l)$. In particular we have
$$
\KK
\subset
\bigcap_{n=1}^\infty
\left(
\bigcup_{l\geq n}
\bigcup_{\theta\in\cR(K,l)}
B\big(\theta,\psi(K^l)\big)
\right)
\subset W(\cR,\psi).
$$
Moreover, for any $n$ any interval $B\in\cK(n)$ satisfies Condition \eqref{eqmeasureoncantorset}, that is
$$
\mu(B)<\frac{f\big(|B|\big)}{\eta}.
$$

\subsubsection{End of the proof}

For any $n$, any $B\in\cK(n)$ and any subinterval $I\subset B$, denote by $\cK(n+1,I)$ the set of those balls $B'\in\cK(n+1,B)$ such that
$
B'\cap I\not=\emptyset
$.
Define
$$
\rho_0:=
\min
\left\{|t-t'|
\textrm{ ; }
t\in B\in\cK(1)
\textrm{ ; }
t'\in B'\in\cK(1)
\textrm{ ; }
B\not=B'
\right\},
$$
which is positive since $\KK(1)$ is a finite union of disjoint intervals. Let $\Delta$ be the constant appearing in Equation \eqref{eq2LocalConstructionKhinchinJarnick} in Proposition \ref{propLocalConstructionKhinchinJarnick}. 

\begin{lemma}
For any interval with $|I|<\rho_0$ we have
$$
\mu\big(|I|\big)<\frac{\Delta}{\eta}f\big(|I|\big).
$$
\end{lemma}

\begin{proof}
By definition of $\rho_0$, if $|I|<\rho_0$ then there exists $n$ such that $I$ intersects at most one $B\in\cK(n)$. Moreover we can assume that $I\subset B$, since $\mu$ does not give positive measure to subsets $E$ with
$
E\cap\KK(n)=\emptyset
$.
We have
$$
\mu(I)\leq
\frac
{\sum_{B'\in \cK(n+1,I)}f\big(|B'|\big)}
{\sum_{B'\in \cK(n+1,B)}f\big(|B'|\big)}\mu(B)<
\Delta\frac{|I|}{|B|}\mu(B)<
\Delta\frac{f\big(|I|\big)}{f\big(|B|\big)}\mu(B)<
\frac{\Delta}{\eta}f\big(|I|\big),
$$
where the first inequality follows from to the definition of $\mu$ (see Lemma \ref{lem2ss4abstractkhinchinjarnick}), the second follows from Equation \eqref{eq2LocalConstructionKhinchinJarnick} in Proposition \ref{propLocalConstructionKhinchinJarnick}, the third holds because $f(r)/r$ is decreasing monotone (for increasing $r$) and the fourth because any interval
$
B\in\bigcup_{n\in\NN}\cK(n)
$
satisfies Condition \eqref{eqmeasureoncantorset}.
\end{proof}

According to Lemma \ref{lem1ss4abstractkhinchinjarnick} we have
$
\displaystyle{H^f\big(\KK\big)\geq \frac{\eta}{\Delta}}
$. 
For any $\eta>0$ we can define a Cantor set $\KK=\KK_\eta$ with $\KK\subset W(\cR,\psi)$ which satisfies the estimate above. Therefore we have 
$$
H^f\big(W(\cR,\psi)\big)=+\infty.
$$ 
The divergent case of Theorem \ref{thmabstractkhinchinjarnik} for Hausdorff measure is proved. This completes the proof of Theorem \ref{thmabstractkhinchinjarnik}.

\section{Hausdorff dimension of $\bad(\cR,\epsilon)$: proof of Theorem \ref{TheoremAbstractJarnickInequality}}
\label{SecProofAbstractJarnikInequality}

In this section we prove Theorem \ref{TheoremAbstractJarnickInequality}. For any real number $s$ with $0<s<1$, consider the function $f_s:\RR_+\to\RR_+$ defined by $f_s(x)=x^s$. The lower bound for the Hausdorff dimension of $\bad(\cR,\epsilon)$ is proved in \S~\ref{SecProofLowerBoundAbstractJarnikInequality}. The upper bound is proved in \S~\ref{SecProofUpperBoundTheoremAbstactJarnikInequality}.

\subsection{Proof of lower bound}
\label{SecProofLowerBoundAbstractJarnikInequality}

Fix constants $\epsilon>0$ and $\tau>0$ with $\epsilon<1$ and $\tau<1-\epsilon^2$, and set 
$$
\epsilon:=\frac{1}{K}. 
$$
Let $(\cR,l)$ be a planar resonant set, and assume that it is $(\epsilon,\tau)$-decaying. For convenience of notation, for any $n\in\NN$ and any $\delta>0$ set
$$
\Delta(K,n,\delta):=
\bigcup_{\theta\in\cR(K,n)}
B\left(\theta,\frac{\delta}{l(\theta) K^{n}}\right).
$$
According to Definition \ref{DefMetricPropertiesResonantSets}, the $(\epsilon,\tau)$-decaying assumption on $(\cR,l)$ means that for any interval $I$ and any integer $n\geq 1$ satisfying Condition \eqref{EqDefDecayingAssumption}, that is
$$
|I|=\frac{1}{K^{2n}}
\quad\textrm{ and }\quad
I\cap
\bigcup_{j=1}^{n-1}\Delta(K,j,\epsilon^2)=\emptyset
$$
the estimate in Equation \eqref{EqDefDecayingConclusion} is satisfied too, that is
$$
\big|I\cap\Delta(K,n,2\epsilon^2)\big|<\tau|I|,
$$
and moreover there exists an interval $I_0$ satisfying Condition \eqref{EqDefDecayingAssumption} for $n=1$. 

\subsubsection{Construction of a probability measure on a Cantor set}

We apply the constructions of \S~\ref{SecMassDistrubution}. Let $I_0$ be an interval satisfying Condition \eqref{EqDefDecayingAssumption} for $n=1$. Such interval exist by assumption in the definition of $(\epsilon,\tau)$-decaying resonant set. We set $\cK(1):=\{I_0\}$, then for any $n\geq 1$ we define inductively a family $\cK(n)$ of intervals $I_i$ mutually disjoint in their interior and satisfying Condition \eqref{EqDefDecayingAssumption}. Assume that the first $n$ families $\cK(1),\dots,\cK(n)$ are defined and consider any interval $I$ in the family $\cK(n)$, recalling in particular that $|I|=K^{-2n}$. Let $[K^2]$ be the integer part of $K^2$. Consider a family $\big(I_i\big)_{i=1,\dots,[K^2]}$ of subintervals $I_i\subset I$, all of length $|I_i|=|I|\epsilon^2$ for any $i$ and any two of them disjoint in their interior. Such family of subintervals covers $I$ modulo a subset of measure at most $|I|\epsilon^2$. Define the sub-family $\cK(n+1|I)$ of $\big(I_i\big)_{i=1,\dots,[K^2]}$ by 
\begin{equation}
\label{eq1LocalConstructionLowerBoundJarnick}
\cK(n+1|I):=
\left\{I_i
\textrm{ ; }
1\leq i\leq [K^2]
\textrm{ and }
I_i\cap\Delta(K,n,\epsilon^2)=\emptyset
\right\},
\end{equation}
then define the family $\cK(n+1)$ by 
$$
\cK(n+1):=\bigcup_{I\in\cK(n)}\cK(n+1|I).
$$
Define a Cantor set by $\KK=\bigcap_{n=N}^\infty\KK(n)$, where any level is defined by
$
\KK(n):=\bigcup_{I\in\cK(n)}I
$, 
so that $\KK(n+1)\subset\KK(n)$ for any $n$. For any $\theta\in\cR(K,n)$ we have
$
K^{n-1}\leq l(\theta)<K^n
$,
therefore, recalling that $\epsilon=1/K$, we have
$$
B\left(\theta,\frac{\epsilon^3}{l(\theta)^2}\right)
\subset
B\left(\theta,\frac{\epsilon^2}{l(\theta)K^n}\right).
$$
Hence
$$
\KK\subset
I_0\setminus\bigcup_{n=1}^\infty\Delta(K,n,\epsilon^2)
\subset
I_0\setminus\bigcup_{\theta\in\cR}
B\left(\theta,\frac{\epsilon^3}{l(\theta)^2}\right)=
\bad\left(\cR,\epsilon^{3/2}\right)\cap I_0.
$$
Finally, as in \S~\ref{SecMassDistrubution}, a Borel probability measure $\mu$ is defined and supported on $\KK$. We recall that for the intervals in the construction above such measure is defined setting $\mu(I_0):=1$ and, assuming that $\mu(I)$ is defined for any $I$ in $\cK(n)$, setting  
\begin{equation}
\label{EquationLocalConstructionBadLowerBound(Recall2.4)}
\mu\big(|I_i|\big):=
\frac
{f_s\big(|I_i|\big)}
{\sum_{I_j\in\cK(n+1,I)}f_s\big(|I_j|\big)}
\mu\big(|I|\big)
\end{equation}
for any $I_i\in\cK(n+1|I)$. Actually, other than for $I_0$, we will define $\mu(I)$ only for intervals $I$ in $\cK(n)$ with $n\geq N$, where $N$ is a positive integer given by Proposition~\ref{PropLocalConstructionBadLowerBound} below. The estimate on the lower bound in Theorem \ref{TheoremAbstractJarnickInequality} follows from a lower bound for $\dim(\KK)$, which follows itself from the next Proposition.

\begin{proposition}
\label{PropLocalConstructionBadLowerBound}
For any $n\geq 1$ and any interval $I\in\cK(n)$ satisfying Condition \eqref{EqDefDecayingAssumption}, the family $\cK(n+1|I)$ defined in Equation \eqref{eq1LocalConstructionLowerBoundJarnick} has cardinality
$$
\sharp\cK(n+1|I)\geq(1-\tau-\epsilon^2)K^2.
$$
In particular, whenever
\begin{equation}
\label{EqLocalConstructionLowerBoundBadAssumption}
s<1-\frac{|\log(1-\tau-\epsilon^2)|}{2|\log\epsilon|}
\end{equation}
for any $n\geq1$ and any $I\in\cK(n)$ we have
\begin{equation}
\label{EqLocalConstructionLowerBoundBadConclusion}
\sum_{I_i \in \cK(n+1|I)}
f_s\big(|I_i|\big)\geq f_s\big(|I|\big).
\end{equation}
Finally, there exists $N\geq2$ such that for any $s$ as above, for any $n\geq N$ and any interval $I\in\cK(n)$ Equation \eqref{eqmeasureoncantorset} is satisfied with $\eta=1$, that is
$$
\mu\big(I\big)
\leq
f_s\big(|I|\big).
$$
\end{proposition}

\begin{proof}
Observe that every subinterval $I_i$ of $I$ has length 
$
\epsilon^2|I|=\epsilon^{2(n+1)}
$ 
and any interval in $\Delta(K,n,\epsilon^2)$ has length at least $2\epsilon^{2(n+1)}$. Therefore any $I_i$ such that
$
I_i\cap\Delta(K,n,\epsilon^2)\not=\emptyset
$
must be contained in $\Delta(K,n,2\epsilon^2)$. Since $\cR$ is an $(\epsilon,\tau)$-decaying resonant set and by assumption $I$ satisfies Condition \eqref{EqDefDecayingAssumption}, we have
\begin{eqnarray*}
&&
\left(1 - \frac{\sharp\cK(n+1|I)}{K^2}\right)\lvert I \rvert
\leq
\left|
\bigcup_{I_i\cap\Delta(K,n,\epsilon^2)\not=\emptyset}I_i
\right|
+\frac{|I|}{K^2}
\leq
\\
&&
\lvert I \cap \Delta(K,n,2\epsilon^2) \rvert
+\frac{|I|}{K^2}
\leq
(\tau+\epsilon^2)\cdot|I|
\end{eqnarray*}
and hence $\sharp\cK(n+1|I)\geq (1-\tau-\epsilon^2)K^2$. According to this last estimate, Equation \eqref{EqLocalConstructionLowerBoundBadConclusion} follows directly from Condition \eqref{EqLocalConstructionLowerBoundBadAssumption} with a simple computation, recalling that $f_s(|I|)=|I|^s$ and observing that
$$
\sum_{I_i\in\cK(n+1|I)} f_s\big(|I_i|\big)
\geq
(1-\tau-\epsilon^2)K^2\left(\frac{|I|}{K^2}\right)^s
=
(1-\tau-\epsilon^2)\cdot\epsilon^{2(s-1)}|I|^s.
$$

Finally, fix $s$ satisfying Condition \eqref{EqLocalConstructionLowerBoundBadAssumption} and observe that such condition is equivalent to 
$
(1-\tau-\epsilon^2)K^{2(1-s)}>1
$. 
Therefore there exists $N\geq 2$ such that 
$$
(1-\tau-\epsilon^2)^{N-1}K^{2(N-1)(1-s)}
>
\frac{1}{|I_0|^s}.
$$
We proved yet that the family $\cK(N)$ contains at least 
$
(1-\tau-\epsilon^2)^{N-1}K^{2(N-1)}
$ 
intervals $I_i\subset I_0$, each of size $|I_i|=|I_0|K^{-2(N-1)}=K^{-2N}$, hence
$$
\sum_{I_i\in\cK(N)}|I_i|\geq \delta|I_0|
\quad
\textrm{ where }
\quad
\delta:=(1-\tau-\epsilon^2)^{N-1}
$$ 
According to our choice of $N$, for any $I_i\in\cK(N)$ we have
$$
\frac{f_s\big(|I_i|\big)}{|I_i|}=
\left(\frac{|I_0|}{K^{2(N-1)}}\right)^{s-1}
\geq
\frac{1}{(1-\tau-\epsilon^2)^{N-1}|I_0|}=
\frac{\mu(I_0)}{\delta|I_0|}.
$$ 
Let $\mu$ be the mass distribution defined by Equation~\eqref{EquationLocalConstructionBadLowerBound(Recall2.4)}. Equation~\eqref{eqmeasureoncantorset(alternative)} and Lemma \ref{lemmaaereo} imply 
$
\mu\big(I_i\big)<f_s\big(|I_i|\big)
$, 
that is Equation \eqref{eqmeasureoncantorset} is satisfied by any interval $I_i$ in $\cK(N)$. We prove by induction that the same is true for any $n\geq N$, and this will complete the proof of the Proposition. Consider any $n\geq N$ and any interval $I$ in the family $\cK(n)$, and assume that 
$
\mu\big(I\big)<f_s\big(|I|\big)
$. 
For any $I_i\in \cK(n+1,I)$ we have
$$
\mu\big(|I_i|\big)=
\frac
{f_s\big(|I_i|\big)}
{\sum_{I_j\in\cK(n+1,I)}f_s\big(|I_j|\big)}
\mu\big(|I|\big)
\leq
\frac
{f_s\big(|I_i|\big)}
{\sum_{I_j\in\cK(n+1,I)}f_s\big(|I_j|\big)}
f_s\big(|I|\big)
\leq
f_s\big(|I|\big),
$$
where the equality corresponds to the definition of $\mu$, the first inequality corresponds to the inductive assumption and the last inequality follows from Condition  \eqref{EqLocalConstructionLowerBoundBadConclusion}. 
\end{proof}

\subsubsection{End of the proof}

Here we finish the proof of the lower bound in Theorem \ref{TheoremAbstractJarnickInequality}. Consider $s$ satisfying Condition \eqref{EqLocalConstructionLowerBoundBadAssumption}. According to Proposition \ref{PropLocalConstructionBadLowerBound}, Equation~\eqref{eqmeasureoncantorset} is satisfied with 
$\eta=1$ for any $n\geq N$ and any interval $I\in\cK(n)$, where $N$ is the integer in the last part of the Proposition. We will deduce here that Equation~\eqref{eqmeasureoncantorset} is satisfied for any interval $J$ with length $|J|\leq K^{-2N}$ with 
$$
\eta:=\frac{1}{2K^{2s}}.
$$
Consider any such interval $J$ with $J\cap\KK\not=\emptyset$, that is $J\cap\KK(n)\not=\emptyset$ for any $n\geq N$. Let $m\geq N$ be the unique integer such that $K^{-2(m+1)}<|J|\leq K^{-2m}$. Since $|J|\leq K^{-2m}$, then there are at most two intervals $I_1$ and $I_2$ in the family $\cK(m)$ such that $J\cap I_i\not=\emptyset$ for $i=1,2$. We have $\mu(I)\leq\mu(I_1)+\mu(I_2)$, because $\mu$ does not charge sets disjoint to $\KK(m)$. Therefore
$$
\mu(J)
\leq
\mu(I_1)+\mu(I_2)
\leq
f_s\big(|I_1|\big)+f_s\big(|I_2|\big)=
\frac{2K^{2s}}{K^{2s(m+1)}}\leq
2K^{2s}f_s\big(|J|\big)=
\frac{f_s\big(|J|\big)}{\eta},
$$
where the second inequality follows from the last part of Proposition \ref{PropLocalConstructionBadLowerBound}. According to Lemma \ref{lem1ss4abstractkhinchinjarnick} the last inequality implies $H^s(\KK)\geq\eta$ for any $s$ satisfying Condition \eqref{EqLocalConstructionLowerBoundBadAssumption}, therefore
$$
\dim(\KK)
\geq
1-\frac{|\log(1-\tau-\epsilon^2)|}{2|\log(\epsilon)|}.
$$
The lower bound in Theorem \ref{TheoremAbstractJarnickInequality} follows recalling that 
$
\KK\subset\bad(\cR,\epsilon^{3/2})
$ 
by replacing $\epsilon$ by $\epsilon^{3/2}$ in the last estimate.

\subsection{Proof of upper bound}
\label{SecProofUpperBoundTheoremAbstactJarnikInequality}

Fix constants $\epsilon,U,\tau$ with $0<\epsilon<1$, $0<\tau<1$ and $U>1$ and let $(\cR,l)$ be a resonant set satisfying $(\epsilon,U,\tau)$-Dirichlet property. Set 
$$
K:=\frac{4U}{\epsilon^2}.
$$ 
Up to choosing a slightly bigger $U>0$, assume that $K\in \NN$. Recall from Definition \ref{DefMetricPropertiesResonantSets} that $(\epsilon,U,\tau)$-Dirichlet property for $(\cR,l)$ means that there exists some $L_0>0$ such that for any $L\geq L_0$ and any interval $I\subset[-\pi/2,\pi/2[$ with $|I|\geq 2U/L^2$ Equation \eqref{EqDefDirichletUbiquity} is satisfied, that is we have
$$
\left|
I\cap
\bigcup_{l(\theta)\leq L}
B\left(\theta,\frac{\epsilon^2}{2l(\theta)^2}\right)
\right|
\geq
\tau|I|.
$$

\subsubsection{A sequence of coverings}

In order to prove the upper bound in Theorem \ref{TheoremAbstractJarnickInequality}, we fix some positive integer $N$ and define a sequence of coverings $\big(\cC(n)\big)_{n\geq N}$ for $\bad(\cR,\epsilon)$ satisfying the properties below.
\begin{enumerate}
\item
For any $n\geq N$ we have
$$
\bad(\cR,\epsilon)
\subset
\bigcup_{I\in\cC(n)}I.
$$
\item
Any interval $I$ in $\cC(n)$ has length $|I|=\pi\cdot K^{-n}$.
\item
The covering $\cC(n)$ contains at most $(1-\tau)^{n-N} K^{n}$ intervals.
\end{enumerate}

The upper bound follows from the construction of such sequence of coverings, indeed we have
$$
H^s\big(\bad(\cR,\epsilon)\big)
=
\lim_{\delta\to0} 
H^s_\delta\big(\bad(\epsilon)\big)
\leq
\liminf_{n\to\infty}
\sharp \cC(n)\cdot \left(\frac{\pi}{K^n}\right)^s,
$$
where $\sharp\cC(n)$ denotes the number of intervals in the covering $\cC(n)$. According to property (3) above we have
$$
H^s\big(\bad(\cR,\epsilon)\big)
\leq
\liminf_{n\to\infty}
(1-\tau)^{n-N} K^{n}\cdot \left(\frac{\pi}{K^n}\right)^s=
\pi^s
\liminf_{n\to\infty}
\left(
(1-\tau)^{1-N/n}\cdot K^{1-s}
\right)^n.
$$
Therefore 
$
H^s\big(\bad(\cR,\epsilon)\big)<+\infty
$
whenever
$$
(1-\tau)\cdot K^{1-s}<1
\Leftrightarrow
s>
1+\frac{\log(1-\tau)}{\log(K)}
=
1-\frac{|\log(1-\tau)|}{\log(4U/\epsilon^2)}.
$$

\subsubsection{End of the proof}

Here we give the definition of the coverings $\cC(n)$ satisfying the properties (1), (2) and (3) as above. For any $n\in\NN$ let $L_n>0$ be the real number satisfying the relation
$$
\frac{\pi}{K^n}=\frac{\epsilon^2}{2L_n^2}.
$$
Consider the parameter $L_0$ in the definition of Dirichlet property, then let $N$ be the positive integer such that $L_n\geq L_0$ for any $n\geq N$. Observe that with this choice of $L_n$, and recalling that $K=4U/\epsilon^{2}$, we have
$$
\frac{\pi}{K^{n-1}}=\frac{2U}{L_n^2}
$$

For $n=N$ subdivide the interval $[-\pi/2,\pi/2[$ into $K^N$ intervals of length $\pi\cdot K^{-N}$ and define $\cC(N)$ as the family of all these intervals. Such cover obviously satisfies the properties $(1)$, $(2)$ and $(3)$ above. Consider $n>N$ and assume that the families $\cC(i)$ are defined for $i=N,\dots,n-1$. Fix any interval $I$ in $\cC(n-1)$ of length $|I|=\pi/K^{n-1}$. Subdivide $I$ into $K$ intervals $I_1,\dots,I_K$ mutually disjoint in their interior and all of equal length $|I_i|=|I|/K$ for any $i$. Define $\cK(n|I)$ as the family of those intervals $I_i$ which are disjoint to all intervals
$
B\big(\theta,\epsilon^2/2l(\theta)^2\big)
$
with $l(\theta)\leq L_n $, that is
$$
\cK(n|I):=
\left\{
I_i
\textrm{ ; }
I_i\cap
\bigcup_{l(\theta)\leq L_n}
B\left(\theta,
\frac{1}{2}
\frac{\epsilon^2}{l(\theta)^2}\right)
=\emptyset
\right\}.
$$
Then set
$$
\cC(n):=\bigcup_{I\in\cC(n-1)}\cK(n|I).
$$

\begin{proposition}
\label{PropLocalConstructionBadUpperBound}
For any interval $I$ in the cover $\cC(n-1)$ the family $\cK(n|I)$ has cardinality
$$
\sharp\cK(n|I)<(1-\tau)K.
$$
Moreover
$$
\bad(\cR,\epsilon)\cap I
\subset
\bigcup_{I_i\in \cK(n|I)} I_i.
$$
\end{proposition}

\begin{proof}
Consider any $\theta \in \cR$ with $l(\theta) \leq L_n$. The second claim follows observing that for every interval $I_i$ in $\cK(n|I)$ we have
$$
|I_i|=\frac{\pi}{K^n}\leq\frac{\epsilon^2}{2l(\theta)^2}
$$
Hence, every $I_i$ intersecting some interval
$
B\big(\theta,\epsilon^2/2l(\theta)^2\big)
$
with $\theta\in\cR$ and $l(\theta)<L_n$ is contained in $B\big(\theta,\epsilon^2/l(\theta)^2\big)$. Therefore $\cK(n|I)$ is a covering of the set
$$
\cB(I,\epsilon):=
I
\setminus
\bigcup_{l(\theta)\leq L_n}
B\left(\theta,\frac{\epsilon^2}{l(\theta)^2}\right)
$$
and it is evident that 
$
I\cap\bad(\cR,\epsilon)\subset\cB(I,\epsilon)
$. 
Moreover, $I$ satisfies the assumption in the definition of Dirichlet Property for $L=L_n$, indeed we have
$$
|I|=\frac{\pi}{K^{n-1}}=\frac{2U}{L_n^2}.
$$ 
Therefore Dirichlet property for $(\cR,l)$ implies
$$
\left(1 - \frac{\sharp\cK(n|I)}{K}\right)\cdot |I|
=
|I|-\sum_{I_j \in \cK(n|I)}|I_j|
\geq 	
\left|
I\cap
\bigcup_{l(\theta)\leq L_n}
B\left(\theta,\frac{\epsilon^2}{2l(\theta)^2}\right)
\right|
\geq \tau \cdot |I|,
$$
showing that
$
\sharp\cK(n|I) \leq  (1-\tau)K
$
and finishing the proof.
\end{proof}

Property (1) holds for $\cC(n)$ because it holds for $\cC(n-1)$ by inductive assumption and moreover according to the second part of  Proposition \ref{PropLocalConstructionBadUpperBound} we have
$$
\bad(\cR,\epsilon)
\subset
\bigcup_{I \in \cC_{n-1}}I \cap \bad(\cR,\epsilon)
\subset
\bigcup_{I\in\cC_{n-1}}\bigcup_{I_i\in\cK(n|I)}I_i
=
\bigcup_{I \in \cC(n)}I.
$$
Property (2) holds for $\cC(n)$ because it holds for $\cC(n-1)$ by inductive assumption and moreover for any $I\in\cC(n-1)$ and any $I_i\in\cK(n|I)$ we have $|I_i|=|I|/K$. Property (3) holds for $\cC(n)$ because it holds for $\cC(n-1)$ by inductive assumption and moreover, according to the first part of Proposition \ref{PropLocalConstructionBadUpperBound}, we have
\begin{eqnarray*}
&&
\sharp\cC(n)
=
\sum_{I\in\cC(n-1)}\sharp\cK(n|I)
\leq
\sharp\cC(n-1)\cdot(1-\tau)K
\leq
\\
&&
(1-\tau)^{n-1-N}K^{n-1}\cdot(1-\tau)K
=
(1-\tau)^{n-N} K^{n}.
\end{eqnarray*}
The upper bound in Theorem \ref{TheoremAbstractJarnickInequality} is proved.

\section{Planar resonant sets of a translation surface: proof of Theorem \ref{TheoremMetricResultsHolonomyResonantSets}}
\label{ChapterProofMetricPropertiesHolonomyResonantSets}

Fix a translation surface $X$ in some stratum $\cH$ and let $\Sigma$ be the set of its conical singularities $p_1,\dots,p_r$. Let $m$ be the sum of the orders at all conical singularities, that is
$$
m:=2g-2+\sharp\big(\Sigma\big).
$$
In this section we consider the resonant sets $\cR^{sc}$ and $\cR^{cyl}$ defined in \S~\ref{SecIntroHolonomyResonantSets} and we prove Theorem \ref{TheoremMetricResultsHolonomyResonantSets}. Statements (1) and (2) in the Theorem concern the set $\cR^{sc}$. Statement (1) corresponds to Propositions \ref{PropDecaying} and \ref{PropDecayingVeech}, Statement (2) corresponds to Proposition \ref{PropDirichletUbiquity}. Statements (3) and (4) in the Theorem concern the set $\cR^{cyl}$ and they correspond respectively to Proposition \ref{PropIsotropicQuadGrowth} and to Proposition \ref{propUbiquity}.

\subsection{Upper bound for systole and shortest cylinder}

Most of the constants appearing in the metric properties in Theorem \ref{TheoremMetricResultsHolonomyResonantSets} are expressed in terms of the positive integer $m$, which depends only on the stratum $\cH$ of the translation surface $X$. It will be useful to introduce the following constants
$$
S_0:=\frac{\sqrt{2}}{\sqrt{m\sqrt{3}}}
\quad
\textrm{ and }
\quad
T_0:=2^{2^{4m}}.
$$

For us a \emph{flat triangulation} of a translation surface $X$ is a triangulation of $X$ whose vertices are the conical points in $\Sigma$, whose edges are saddle connections and whose triangles do not contain other points of $\Sigma$. The number $v$, $e$ and $t$ respectively of vertices, edges and triangles in such triangulation are topological invariants, and are given by $v=\sharp(\Sigma)$, $e=3m$ and $t=2m$ (see \cite{KerkoffMasurSmillie}). In \cite{BoissyGeninska} it is proved that for any stratum $\cH$ the surface $X_0$ for which $\sys(X_0)$ is maximal admits a flat triangulation whose triangles are all equilateral triangles with side's length $\sys(X_0)$. It follows that for any $X$ in $\cH$ we have
$$
\sys(X)\leq \sys(X_0)=S_0.
$$

Moreover, in Theorem 1.3 in \cite{vorobets1} it is proved that any surface in $\cH$ has closed geodesic $\sigma$ with length $|\sigma|\leq T_0$ and whose cylinder $C_\sigma$ satisfies $\area(C_\sigma)>1/m$. Therefore for any $X$ in $\cH$ we have
$$
\cyl(X)\leq T_0.
$$

Finally, the constant $S_0$ has a second geometrical interpretation, related to Theorem 6.3 in \cite{minskyweiss}. Indeed $3m$ is the maximal number of saddle connections $\gamma_1,\dots,\gamma_{3m}$ on a surface $X$ which are mutually disjoint in their interior, because such a set of saddle connections necessarily gives a flat triangulation of $X$. Therefore $S_0$ is also the smallest bound such that any saddle connection $\gamma_1,\dots,\gamma_{3m}$ in a flat triangulation of $X$ has length $|\gamma_i|\leq S_0$ for any $i=1,\dots,3m$. Equivalently, on a translation surface $X$ there are at most $3m-1$ saddle connections which are mutually disjoint in their interior and all strictly shorter than $S_0$. This motivates the form of the constant $\beta$ appearing in Theorem 6.3 in \cite{minskyweiss}, which is the same as in Proposition \ref{PropDecaying} and is given by
$$
\beta:=\frac{1}{3m-1}.
$$
On the other hand, when $X$ is a Veech surface, we can find a bound $r_0>0$ depending only on the orbit $\sltwor\cdot X$ such that we never have two non-parallel saddle connections shorter than $r_0$ (see Lemma \ref{LemmaSparseCorverVeech}). This explains heuristically why for Veech surfaces we have the better version of decaying, namely Proposition \ref{PropDecayingVeech}, where $\beta=1$.

\subsection{Dirichlet Theorem}
\label{SecDirichletTheorem}

According to classical Dirichlet's Theorem, for any real number $\alpha$ and for any $Q>1$ there exists a rational number $p/q$ with $q\leq Q$ such that 
$$
\left|\alpha-\frac{p}{q}\right|\leq\frac{1}{qQ}.
$$
We develop a version of Dirichlet's Theorem for the resonant sets $\cR^{sc}$ and $\cR^{cyl}$. In particular, for $\cR^{cyl}$ we use a nontrivial result due to Vorobets, namely Theorem 1.3 in \cite{vorobets1}.

\begin{proposition}
\label{propdirichletvorobets}
Let $X$ be any translation surface and $\theta$ be any direction on $X$. 
\begin{enumerate}
\item
For any 
$
\displaystyle{L>\frac{\sqrt{2}S_0^2}{\sys(X)}}
$ 
there exists $\theta_\gamma\in\cR^{sc}$ with $l(\theta_\gamma)\leq L$ such that
$$
|\theta-\theta_\gamma|\leq\frac{\sqrt{2}S^2_0}{l\big(\theta_\gamma\big)L}=
\frac{\sqrt{8}}{m\sqrt{3}}\cdot\frac{1}{l\big(\theta_\gamma\big)L}.
$$
\item
For any 
$
\displaystyle{L>\frac{\sqrt{2}T_0^2}{\cyl(X)}}
$ 
there exists $\theta_\sigma\in\cR^{cyl}$ with $l(\theta_\sigma)\leq L$ such that
$$
|\theta-\theta_\sigma|\leq\frac{\sqrt{2}T_0^2}{l\big(\theta_\sigma\big)L}.
$$
\end{enumerate}
\end{proposition}

\begin{proof}
In order to prove the first statement, set 
$$
e^t:=\frac{L}{S_0}\geq\frac{\sqrt{2}S_0}{\sys(X)}\geq\sqrt{2}. 
$$
There is a saddle connection $\gamma$ on the surface $X$ whose length on the surface $g_t r_\theta X$ satisfies 
$$
|\hol(\gamma,g_t r_\theta X)|\leq S_0.
$$
Let $\theta_\gamma$ be the direction of such $\gamma$ on the surface $X$ and let $|\gamma|$ be its length on $X$. We have 
$$
l(\theta_\gamma)\leq
|\gamma|=|\hol(\gamma,r_\theta X)|
\leq
e^t|\hol(\gamma,g_t r_\theta X)|\leq
e^t S_0=L.
$$
Set 
$
(H,V):=\hol(\gamma,g_t r_\theta X)
$. 
We have obviously 
$
|H|\leq |\hol(\gamma,g_t r_\theta X)|\leq S_0
$ 
and thus, since $L>\sqrt{2} S_0^2/\sys(X)$ by assumption, we get 
$$
H^2e^{-2t}\leq \frac{S_0^2}{e^{2t}}\leq
\frac{S_0^4}{L^2}\leq\frac{\sys(X)^2}{2}.
$$
On the other hand
$$
H^2e^{-2t}+V^2e^{2t}=|\hol(\gamma,r_\theta X)|^2=|\hol(\gamma, X)|^2\geq\sys(X)^2.
$$
The last two estimates imply $V^2e^{2t}\geq H^2e^{-2t}$ and therefore 
$
|V|e^t\geq |\gamma|/\sqrt{2}\geq l(\theta_\gamma)/\sqrt{2}
$, 
so that we get finally
$$
\big|\theta-\theta_\gamma\big|<
\big|\tan(\theta-\theta_\gamma)\big|=
\frac{H}{Ve^{2t}}<
\frac{\sqrt{2}S_0}{l(\theta_\gamma)e^t}=
\frac{\sqrt{2}S_0^2}{l(\theta_\gamma)L}.
$$
The second statement follows with the same argument. Replace $S_0$ by $T_0$ and set $e^t:=L/T_0$. Recall that, according to Vorobets Theorem 1.3 in \cite{vorobets1}, any translation surface in the same stratum as $X$ has a closed geodesic $\sigma$ with length $|\sigma|\leq T_0$ and a corresponding cylinder $C_\sigma$ with $\area(C_\sigma)>1/m$. Thus let $\sigma$ be such geodesic for the surface $g_tr_\theta\cdot X$ and repeat the same argument as above replacing $\gamma$ by $\sigma$.
\end{proof}

\subsection{Dirichlet property} 

Statement (2) in Theorem \ref{TheoremMetricResultsHolonomyResonantSets} follows from Proposition \ref{PropDirichletUbiquity} below.

\begin{proposition}
\label{PropDirichletUbiquity}
For any $\epsilon>0$ the resonant set $(\cR^{sc},l)$ satisfies $(\epsilon,U,\tau)$-Dirichlet property with 
$$
U:=\frac{12}{m^2\epsilon^2}
\quad
\textrm{ and }
\quad
\tau:=\frac{m\epsilon^2}{\sqrt{48}},
$$
that is for any 
$
\displaystyle{L\geq\frac{\sqrt{2}S_0^2}{\sys(X)}}
$ 
and any interval $I$ with 
$
\displaystyle{|I|\geq \frac{2U}{L^2}}
$ 
we have
$$
\left|
I\cap
\bigcup_{l(\theta)\leq L}
B\left(\theta,\frac{\epsilon^2}{2l(\theta)^2}\right)
\right|
\geq
\tau|I|
$$
\end{proposition}

\begin{proof}
Fix $L$ as in the statement, and for any $\theta_\gamma\in\cR^{sc}$ define the \emph{rescaling factor} $r(\theta_\gamma)$ by
$$
r(\theta_\gamma):=\frac{\epsilon^2mL}{l(\theta_\gamma)\sqrt{12}}.
$$
Observe that $r(\theta_\gamma)\geq m\epsilon^2/\sqrt{12}$ for any $\theta_\gamma$ with $l(\theta_\gamma)\leq L$, and moreover we can have $r(\theta_\gamma)>1$ when $l(\theta_\gamma)$ is much smaller than $L$. 
Let $I$ be an interval as in the statement. According to Proposition \ref{propdirichletvorobets} we have
$$
I\subset
\bigcup_{l(\theta_\gamma)\leq L}
B\left(\theta_\gamma,\frac{\sqrt{3}}{ml(\theta_\gamma)L}\right).
$$
Let $\cR^{sc}(L,I)$ be the set of directions $\theta_\gamma\in\cR^{sc}$ with $\theta_\gamma\in I$ and $l(\theta_\gamma)\leq L$, then define 
$$
\nu(I,L):=
\left|
I\cap
\bigcup_{\theta_\gamma\in\cR^{sc}(L,I)}
B\left(\theta_\gamma,\frac{\sqrt{3}}{ml(\theta_\gamma)L}\right)
\right|.
$$
If $\nu(I,L)\geq |I|/2$ then we have
\begin{eqnarray*}
&&
\left|
I\cap
\bigcup_{l(\theta_\gamma)\leq L}
B\left(\theta_\gamma,\frac{\epsilon^2}{2l(\theta_\gamma)^2}\right)
\right|
\geq
\left|
I\cap
\bigcup_{\theta_\gamma\in\cR^{sc}(L,I)}
B\left(\theta_\gamma,\frac{\epsilon^2}{2l(\theta_\gamma)^2}\right)
\right|=
\\
&&
\left|
I\cap
\bigcup_{\theta_\gamma\in\cR^{sc}(L,I)}
B\left(\theta_\gamma,\frac{r(\theta_\gamma)}{ml(\theta_\gamma)L}\right)
\right|
\geq
\frac{m\epsilon^2}{\sqrt{12}}\nu(I,L)\geq
\frac{m\epsilon^2}{\sqrt{48}}|I|.
\end{eqnarray*}
Otherwise, if $\nu(I,L)<|I|/2$, there must be some $\theta_\gamma\in\cR^{sc}$ with $l(\theta_\gamma)\leq L$ and $\theta_\gamma\not\in I$ such that
$$
\left|
I\cap
B\left(\theta_\gamma,\frac{\sqrt{3}}{ml(\theta_\gamma)L}\right)
\right|
>
\frac{|I|}{4}.
$$
We finish the proof showing that such $\theta_\gamma$ must have rescaling factor $r(\theta_\gamma)>1$. Observe first that since $\theta_\gamma\not\in I$, we must have 
$
\sqrt{3}\cdot\big(ml(\theta_\gamma)L\big)^{-1}>|I|/4
$. 
Moreover we have $|I|\geq 2U/L^2$ by assumption, thus it follows
$$
r(\theta_\gamma)=
\frac{\sqrt{3}}{ml(\theta_\gamma)}\cdot\frac{\epsilon^2m^2L}{\sqrt{12}\sqrt{3}}>
\frac{U}{2L}\cdot\frac{\epsilon^2m^2L}{6}=1.
$$
\end{proof}

\subsection{Isotropic quadratic growth}
\label{SecIsotropicQuadraticGrowth}

Statement (3) in Theorem \ref{TheoremMetricResultsHolonomyResonantSets} follows from Proposition \ref{PropIsotropicQuadGrowth} below.

\begin{lemma}
\label{lemmaIsotropicGrowthCylinders}
Let $\sigma$ be a closed geodesics in $X$ and let $C_\sigma$ be the corresponding cylinder. For any other closed geodesic $\sigma'$ intersecting $C_\sigma$ we have
$$
\big|\theta_\sigma-\theta_{\sigma'}\big|>
\frac{\area(C_\sigma)}{|\sigma|\cdot|\sigma'|}
$$
\end{lemma}

\begin{proof}
The width of $C_\sigma$ is
$
\area(C_\sigma)/|\sigma|
$.
Since $\sigma$ and $\sigma'$ are not parallel, than $\sigma'$ is not contained in $C_\sigma$, therefore
$$
|\sigma'|\cdot
|\sin\big(\theta_\sigma-\theta_{\sigma'}\big)|>
\frac{\area(C_\sigma)}{|\sigma|}
$$
and the Lemma follows since
$
\big|\theta_\sigma-\theta_{\sigma'}\big|
>
|\sin\big(\theta_\sigma-\theta_{\sigma'}\big)|
$.
\end{proof}

\begin{proposition}
\label{PropIsotropicQuadGrowth}
For any subinterval $I\subset[-\pi/2,\pi/2[$ and any $L>0$ such that $L^2|I|>1$ we have
$$
\sharp
\{\theta\in I\cap\cR^{cyl}(X,L)\}
<
m(m+1)|I|L^2.
$$
\end{proposition}

\begin{proof}
Consider $\theta_1=\theta(\sigma_1)$ and $\theta_1=\theta(\sigma_1)$ in $\cR^{cyl}(X,L)$ be any two directions of closed geodesics $\sigma_1$ and $\sigma_2$, and let $C_1$ and $C_2$ be the corresponding cylinders, so that in particular $\area(C_i)>1/m$ for $i=1,2$. Assume that $\theta_1$ and $\theta_2$ belong to the same interval $J$ of length $|J|\leq 1/(mL^2)$. According to Lemma \ref{lemmaIsotropicGrowthCylinders} the cylinders $C_1$ and $C_2$ are disjoint, indeed the directions $\theta_1$ and $\theta_2$ satisfy
$$
|\theta_1-\theta_2|<
|J|<
\frac{1}{mL^2}<
\frac{1}{ml(\theta_1)l(\theta_2)}.
$$
Since $\area(X)=1$ then $X$ contains at most $m$ disjoint cylinders with area greater that $1/m$, therefore any interval $J$ with length $|J|\leq 1/(mL^2)$ contains at most $m$ directions $\theta_i$ in $\cR^{cyl}(X,L)$. The Proposition follows covering $I$ with
$$
N:=\big[mL^2|I|\big]+1<
mL^2|I|+1<(m+1)L^2|I|
$$
intervals $J_1,\dots,J_N$ with length $|J_j|\leq1/(mL^2)$ for any $j=1,\dots,N$.
\end{proof}

\subsection{Ubiquity}
\label{SecUbiquity}

Statement (4) in Theorem \ref{TheoremMetricResultsHolonomyResonantSets} follows from Proposition \ref{propUbiquity} below. Fix a translation surface $X$ and fix a positive real number $K>1$ such that
$$
K\geq
\frac{\sqrt{2}T_0^2}{\cyl(X)}=
\frac{\sqrt{2}}{\cyl(X)}\cdot 2^{2^{4m+1}}.
$$
According to such assumption, for any positive integer $n\geq1$ we can apply the second statement in Proposition \ref{propdirichletvorobets} for those $\theta_\sigma\in\cR^{cyl}$ such that 
$
l(\theta_\sigma)\leq K^n
$. 
Observe also that, since $\cyl(X)\leq T_0$ for any $X$, we have
$$
K\geq\sqrt{2}T_0=\sqrt{2}\cdot 2^{2^{4m}}>m\sqrt{48}.
$$
This second property will be used in the end of the proof of Proposition \ref{propUbiquity} below. The proposition is due to J. Chaika, and we follow the argument from \cite{ChaikaHomogeneousApproximationsTranslSurf}.

\begin{proposition}
[Chaika]
\label{propUbiquity}
Let $X$ be a translation surface and consider 
$
\displaystyle{K\geq\frac{\sqrt{2}T_0^2}{\cyl(X)}}
$. 
For any positive integer $n\geq 1$ and any interval
$
I\subset[-\pi/2,\pi/2[
$ 
such that
\begin{equation}
\label{EquationPropUbiquity(hypothesis)}
|I|\geq\frac{1}{2m\cyl(X)K^{n-1}}
\end{equation}
we have
\begin{equation}
\label{EquationPropUbiquity(thesis)}
\left|
I\cap
\bigcup_{l(\theta_\sigma)\leq K^n}
B\left(\theta_\sigma,\frac{\sqrt{3}K}{K^{2n}}\right)
\right|
\geq
\frac{|I|}{2}.
\end{equation}
\end{proposition}

The dependence on $K$ of the radius of the balls in Equation~\eqref{EquationPropUbiquity(thesis)} can be reduced to the simplified expression $\sqrt{3}\cdot K^{2n-1}$. We keep the redundant $K$ in the numerator in order make clear the relation with Point (4) in Theorem~\ref{TheoremMetricResultsHolonomyResonantSets}. The assumption in Equation~\eqref{EquationPropUbiquity(hypothesis)} is not explicitly stated in Chaika's statement of Ubiquity, namely Proposition 2 in \cite{ChaikaHomogeneousApproximationsTranslSurf}. It seems to us that the same assumption is implicitly used in the proof of Corollary 3 (at line 3) in \cite{ChaikaHomogeneousApproximationsTranslSurf}. Anyhow a lower bound on the length $|I|$ of the interval in Proposition~\ref{propUbiquity} is obviously necessary, indeed the Proposition fails for any interval $I$ which is contained in the complement of 
$
\bigcup_{l(\theta_\sigma)\leq K^n}
B\left(\theta_\sigma,\sqrt{3}K/K^{2n}\right)
$.

\subsubsection{Preliminary Lemmas}

\begin{lemma}\label{lemmaPreliminaryUbiquityCylinders}
Fix $r>0$ and $0<\epsilon<1$. Let $I$ be any interval in $[-\pi/2,\pi/2[$ with $|I|>r$. For any $\theta\in[-\pi/2,\pi/2[$ we have
$$
|I\cap B(\theta,\epsilon\cdot r)|
\leq
2\epsilon\cdot |I\cap B(\theta,r)|.
$$
\end{lemma}

\begin{proof}
If $\theta\in I$ then we have
$
|I\cap B(\theta,\epsilon\cdot r)|\leq 2\epsilon\cdot r
$
and
$
|I\cap B(\theta,r)|\geq r
$,
thus the statement follows. If $\theta\not\in I$ then $r>|I\cap B(\theta,r)|$ thus, since $0<\epsilon<1$, we have
$$
\epsilon\big(r-|I\cap B(\theta,r)|\big)
<r-|I\cap B(\theta,r)|,
$$
which is equivalent to
$$
|I\cap B(\theta,\epsilon\cdot r)|=
\epsilon\cdot r -\big(r-|I\cap B(\theta,r)|\big)<
\epsilon\cdot|I\cap B(\theta,r)|.
$$
\end{proof}

\begin{lemma}
\label{lemmaUbiquityCylinders}
Consider $\epsilon$ with 
$
\displaystyle{0<\epsilon <\frac{1}{2m}}
$, 
a positive integer $n$ and a subinterval $I\subset [-\pi/2,\pi/2[$ such that
$$
|I|\geq \frac{1}{2m\cyl(X)K^n}.
$$
We have
$$
\left|
I\cap
\bigcup_{l(\theta_\sigma)\leq K^n}
B\left(
\theta,
\frac{\epsilon}{l(\theta_\sigma)K^n}
\right)
\right|
<
2m^2\epsilon\cdot |I|.
$$
\end{lemma}

\begin{proof}
Fix any direction $\theta_0\in I$. Fix $n$. Consider $\theta_1$ and $\theta_2$ in $\cR^{cyl}$ be any two directions of closed geodesics $\sigma_1$ and $\sigma_2$ with $l(\theta_i)\leq K^n$ for $i=1,2$, and let $C_1$ and $C_2$ be the corresponding cylinders, so that in particular $\area(C_i)>1/m$ for $i=1,2$. Assume that we have
$$
|\theta_0-\theta_i|<
\frac{1}{2mK^nl(\theta_i)}
\textrm{ for }
i=1,2.
$$
According to Lemma \ref{lemmaIsotropicGrowthCylinders} the cylinders $C_1$ and $C_2$ are disjoint, indeed the directions $\theta_1$ and $\theta_2$ satisfy
$$
|\theta_1-\theta_2|<
\frac{1}{mK^n\min\{l(\theta_1),l(\theta_2)\}}.
$$
There exist at most $m$ disjoint cylinders with area greater that $1/m$, therefore there exist at most $m$ directions $\theta_i$ of closed geodesics $\sigma_i$ such that
$$
|\theta_0-\theta_i|<
\frac{1}{2mK^nl(\theta_i)}
\textrm{ for }
i=1,\dots,m.
$$
We get
$$
\sum_{l(\theta_\sigma)\leq K^n}
\bigg|
I\cap
B\bigg(
\theta,
\frac{1}{2ml(\theta_\sigma)K^n}
\bigg)
\bigg|
\leq
m\cdot\big|I\big|.
$$
According to our assumption we have $2m\epsilon<1$ and 
$
\displaystyle{|I|>\frac{1}{2ml(\theta_\sigma)K^n}}
$ 
for any $\theta_\sigma\in\cR^{cyl}$, thus the statement follows from the previous estimate and from Lemma \ref{lemmaPreliminaryUbiquityCylinders}, observing that
\begin{eqnarray*}
&&
\bigg|
I\cap
\bigcup_{l(\theta_\sigma)\leq K^n}
B\bigg(
\theta,
\frac{\epsilon}{l\big(\theta_\sigma\big)K^n}
\bigg)
\bigg|
<
\sum_{l(\theta_\sigma)\leq K^n}
\bigg|
I\cap
B\bigg(
\theta,\frac{\epsilon}{l(\theta_\sigma)K^n}
\bigg)
\bigg|\leq
\\
&&
2m\epsilon\cdot
\sum_{l(\theta_\sigma)\leq K^n}
\left|
I\cap
B\left(
\theta,\frac{1}{2ml(\theta_\sigma)K^n}
\right)
\right|
\leq
2m^2\epsilon\cdot |I|.
\end{eqnarray*}
\end{proof}

\subsubsection{Proof of Proposition \ref{propUbiquity}}

Fix any $n\geq 1$. Recall that according to the second statement in Proposition \ref{propdirichletvorobets} we have
$$
I\subset
\bigcup_{l(\theta_\sigma)\leq K^n}
B\left(\theta_\sigma,\frac{\sqrt{3}}{ml(\theta_\sigma)K^n}\right).
$$
Moreover, recalling that $K\geq m\sqrt{48}$ and applying Lemma \ref{lemmaUbiquityCylinders} with 
$
\displaystyle{\epsilon:=\frac{\sqrt{3}}{mK}<\frac{1}{2m}}
$, 
we get
$$
\left|
I\cap
\bigcup_{l(\theta_\sigma)\leq K^{n-1}}
B\left(\theta_\sigma,\frac{\sqrt{3}}{mK^nl(\theta_\sigma)}\right)
\right|
<
2m^2\epsilon|I|=
\frac{m\sqrt{12}}{K}|I|.
$$
Therefore, recalling that $\cR^{cyl}(K,n)$ is the set of those $\theta_\sigma\in\cR^{cyl}$ such that 
$
K^{n-1}<l(\theta_\sigma)\leq K^n
$, 
we have
\begin{eqnarray*}
&&
\left|
I\cap
\bigcup_{l(\theta_\sigma)\leq K^n}
B\left(\theta_\sigma,\frac{\sqrt{3}K}{K^{2n}}\right)
\right|
\geq
\left|
I\cap
\bigcup_{\theta_\sigma\in\cR^{cyl}(K,n)}
B\left(\theta_\sigma,\frac{\sqrt{3}K}{K^{2n}}\right)
\right|\geq
\\
&&
\left|
I\cap
\bigcup_{\theta_\sigma\in\cR^{cyl}(K,n)}
B\left(\theta_\sigma,\frac{\sqrt{3}}{mK^{n}l(\theta_\sigma)}\right)
\right|\geq
\\
&&
|I|-
\left|
I\cap
\bigcup_{l(\theta_\sigma)\leq K^{n-1}}
B\left(\theta_\sigma,\frac{\sqrt{3}}{mK^{n}l(\theta_\sigma)}\right)
\right|\geq
|I|-
\frac{m\sqrt{12}}{K}|I|\geq\frac{|I|}{2}.
\end{eqnarray*}
Proposition \ref{propUbiquity} is proved.

\subsection{Decaying}
\label{SecDecaying}

Statement (1) in Theorem \ref{TheoremMetricResultsHolonomyResonantSets} corresponds to Proposition \ref{PropDecaying} and Proposition \ref{PropDecayingVeech} below, whose proof is the subject of this section. Let $X$ be a translation surface with $\area(X)=1$. Recall that we set
$$
\beta:=\frac{1}{3m-1}.
$$

\begin{proposition}
\label{PropDecaying}
There are positive constants $M=M(m)$ and $r_0=r_0(m)$ depending only on $m$ such that for any 
$\epsilon$ with 
$
0<\epsilon<\min\{r_0,\sys(X)\}
$ 
the resonant set $\cR^{sc}$ is $(\epsilon,\tau)$-decaying with 
$$
\tau=M\cdot \epsilon^\beta.
$$
In other words, setting $K:=1/\epsilon$, the following holds. For any $n\geq 1$ and any interval $I$ satisfying Condition \eqref{EqDefDecayingAssumption}, that is
$$
|I|=\frac{1}{K^{2n}}
\quad\textrm{ and }\quad
I
\cap
\bigcup_{j=0}^{n-1}
\bigcup_{\theta_\gamma\in\cR^{sc}(K,j)}
B\left(\theta_\gamma,\frac{\epsilon^2}{l(\theta_\gamma)\cdot K^j}\right)
=\emptyset
$$
we have
$$
\left|
I\cap
\bigcup_{\theta_\gamma\in\cR^{sc}(K,n)}
B\left(\theta_\gamma,\frac{2\epsilon^2}{l(\theta_\gamma)\cdot K^n}\right)
\right|
<
M\cdot \epsilon^\beta\cdot |I|.
$$
Moreover there exist at least 
$
(1-M\epsilon^\beta)K^2
$ 
intervals $I_i\subset[-\pi/2,\pi/2[$ which are mutually disjoint in their interior and satisfy Condition \eqref{EqDefDecayingAssumption} for $n=1$. 
\end{proposition}

\subsubsection{Decaying for a Veech surface}

\begin{proposition}
\label{PropDecayingVeech}
Let $X$ be a Veech surface. Then the result in Proposition \ref{PropDecaying} holds with $\beta=1$ and with $r_0$ that can be chosen uniformly on the closed orbit $\cM:=\sltwor\cdot X$ of $X$.
\end{proposition}

The proof of Proposition \ref{PropDecayingVeech} follows exactly the same lines as Proposition \ref{PropDecaying}. The only difference is that the Minsky-Weiss estimate in Theorem \ref{TheoremMinskyWeiss} below will be replaced by the stronger one in Corollary \ref{CorollaryMinskyWeissVeech}, which says that when $X$ is a Veech surface the same estimate holds as in Theorem \ref{TheoremMinskyWeiss} with $\beta=1$. For completeness, in \S~\ref{SectionProofMinskyWeissForVeech} we give a proof of Corollary \ref{CorollaryMinskyWeissVeech}, adapting the argument of \cite{minskyweiss}. All other details of the proof of Proposition \ref{PropDecayingVeech} will be omitted.

\subsubsection{Non-divergence of horocycle}

We report the statement of Theorem 6.3 in \cite{minskyweiss}, which is the main tool in the proof of Decaying property. Fix a stratum $\cH$ and let $C>1$, $\beta>0$ and $\rho_0>0$ be the constants explicated above.

\begin{theorem}
[Minsky-Weiss]
\label{TheoremMinskyWeiss}
For any translation surface $X\in\cH$ the following holds. Assume $J$ is an interval and $\rho$ is a real number with $0<\rho<S_0$ such that for any saddle connection $\gamma$ we have
$$
\sup_{\alpha\in J}\big|\hol(\gamma,u_{-\alpha}\cdot X)\big|\geq\rho.
$$
Then for any $\rho'$ with $0<\rho'\leq\rho$ we have
$$
\left|
\left\{
\alpha\in J
\textrm{ ; }
\sys(u_{-\alpha}\cdot X)\leq\rho'
\right\}
\right|
\leq
C\cdot
\bigg(
\frac{\rho'}{\rho}
\bigg)^\beta
\cdot
|J|.
$$
\end{theorem}

\begin{corollary}
\label{CorollaryMinskyWeissVeech}
Let $X$ be a Veech surface and let $\cM:=\sltwor\cdot X$ be its closed orbit. There exist constants $C>0$ depending only on the stratum of $X$ and $r_0>0$ depending only on $\cM$ such that the following holds. Assume that $J$ is an interval and $0<\rho<r_0$ is a positive real number such that for any saddle connection $\gamma$ we have
$$
\sup_{\alpha\in J}\big|\hol(\gamma,u_{-\alpha}\cdot X)\big|\geq\rho.
$$
Then for any $\rho'$ with $0<\rho'\leq\rho$ we have
$$
\left|
\left\{
\alpha\in J
\textrm{ ; }
\sys(u_{-\alpha}\cdot X)\leq\rho'
\right\}
\right|
\leq
C\cdot
\frac{\rho'}{\rho}
\cdot
|J|.
$$
\end{corollary}

\subsubsection{Notation and basic facts for the horocycle}

For any saddle connection $\gamma$ on the translation surface $X$ we write 
$
\hol(\gamma,X)=\big(\re(\gamma,X),\im(\gamma,X)\big)
$. 
When there is no ambiguity on the surface $X$ we simply write
$
\big(\re(\gamma),\im(\gamma)\big)
$. 
Moreover we denote the slope of $\gamma$ by 
$$
\alpha_\gamma:=\frac{\re(\gamma)}{\im(\gamma)}.
$$
The action of $u_\alpha$ does not change the vertical part of the planar development of any geodesic segment, that is $\im(\gamma,u_\alpha\cdot X)=\im(\gamma,X)$ for any geodesic segment $\gamma$ on $X$. According to the previous remark, we write
$$
\hol(\gamma,u_{-\alpha}\cdot X)=
\big(\re(\gamma,\alpha),\im(\gamma)\big).
$$
Recall that for any $\alpha$ and $t$ we have $g_tu_\alpha=u_{e^{2t}\alpha}g_t$. Recall also that $\epsilon$ with $0<\epsilon<\sys(X)$ is fixed, and that we set $K=1/\epsilon$. In this paragraph, in order to simplify the notation, for any real number $\lambda$ we set
$$
G_\lambda:=g_{\lambda\log K}=
\begin{pmatrix}
K^\lambda & 0 \\
0 & K^{-\lambda}
\end{pmatrix}.
$$

\begin{lemma}
\label{LemmaLengthHorocycle}
Let $\gamma$ be a saddle connection for the surface $X$ and let $\alpha\in\RR$. Then for any $\lambda>0$ we have
$$
\left|
\hol(\gamma,G_\lambda\cdot u_{-\alpha}\cdot X)
\right|
=
\sqrt{
\big(K^\lambda\cdot|\im(\gamma)|\cdot\big|\alpha-\alpha_\gamma\big|\big)^2
+
\big(\frac{|\im(\gamma)|}{K^\lambda}\big)^2
}.
$$
\end{lemma}

\begin{proof}
Let 
$
(x,y):=\hol(\gamma,u_{-\alpha}\cdot X)
$ 
and observe that  
\begin{eqnarray*}
&&
|y|=|\im(\gamma,u_{-\alpha}\cdot X)|=|\im(\gamma)|
\\
&&
|x|=|\re(\gamma,u_{-\alpha}\cdot X)|=|y|\cdot|\alpha-\alpha_\gamma|.
\end{eqnarray*}
The Lemma follows from
$$
\left|
\hol(\gamma,G_\lambda\cdot u_{-\alpha}\cdot X)
\right|^2
=
\left|
\begin{pmatrix}
K^{\lambda} & 0 \\
0 & K^{-\lambda}
\end{pmatrix}
\cdot
\left(
\begin{array}{c}
x \\
y
\end{array}
\right)
\right|^2
=
K^{2\lambda}|x|^2+K^{-2\lambda}|y|^2.
$$
\end{proof}

In order to avoid ambiguity, in this section we denote by $J\subset\RR$ intervals in the horocycle variable $u_\alpha$, whereas we denote by $I$ intervals in the circle variable, which is parametrized by $r_\theta$. The next Lemma gives an estimate on the distortion in the change of variable. The proof is immediate and thus omitted.

\begin{lemma}
\label{LemmaDistortionArctan}
For any $\alpha_1,\alpha_2$ in $[-1,1]$ we have
$$
\frac{|\alpha_1-\alpha_2|}{2}
\leq
|\arctan(\alpha_1)-\arctan(\alpha_2)|
\leq
|\alpha_1-\alpha_2|.
$$
\end{lemma}

\subsubsection{Conditional probability along horocycle segments}

Recall that we fix a translation surface $X$ and $\epsilon>0$ such that $\epsilon<\sys(X)$, and that we set $K=1/\epsilon$.

\begin{lemma}
\label{LemmaInitiationDecaying}
There exist at least $(1-M\epsilon^\beta)\cdot K^2$ intervals $J_i\subset[-1,1]$ such that any two of them are disjoint in their interior and any of them satisfies
\begin{eqnarray*}
&&
|J_i|=\frac{2}{K^2}
\\
&&
J_i\cap
\bigcup_{|\im(\gamma)|\leq1}
B\big(\alpha_\gamma,\frac{2\epsilon^2}{|\im(\gamma)|}\big)
=\emptyset,
\end{eqnarray*}
where $M>0$ is a constant depending only on the stratum of $X$.
\end{lemma}

\begin{proof}
The first step in the proof is to show that for any saddle connection $\gamma$ for the surface $X$ we have
$$
\sup_{-1\leq \alpha\leq 1}
\left|
\hol(\gamma,G_1\cdot u_{-\alpha}\cdot X)
\right|
\geq
\frac{1}{\sqrt{2}}.
$$
For any saddle connection $\gamma$ we have either 
$
|\im(\gamma)|\geq\sys(X)/\sqrt{2}
$ 
or 
$
|\re(\gamma)|\geq\sys(X)/\sqrt{2}
$. 
Moreover, according to Lemma \ref{LemmaLengthHorocycle}, for any $\alpha\in[-1,1]$ we have
$$
\left|
\hol(\gamma,G_{1}\cdot u_{-\alpha}\cdot X)
\right|
\geq
K\cdot|\im(\gamma)|\cdot\big|\alpha-\alpha_\gamma\big|.
$$ 
If $|\im(\gamma)|\geq\sys(X)/\sqrt{2}$, choose $\alpha\in[-1,1]$ with $|\alpha-\alpha_\gamma|\geq1$. For such $\alpha$ we have
$$
\left|
\hol(\gamma,G_{1}\cdot u_{-\alpha}\cdot X)
\right|
\geq
K|\im(\gamma)|
\geq
\frac{K\sys(X)}{\sqrt{2}}\geq \frac{1}{\sqrt{2}}.
$$
Otherwise we have 
$
|\re(\gamma)|=|\im(\gamma)|\cdot|\alpha_\gamma|\geq \sys(X)/\sqrt{2}
$, 
thus for $\alpha=0$ one gets
$$
\left|
\hol(\gamma,G_1\cdot X)
\right|
\geq
K|\re(\gamma)|
\geq
\frac{K\sys(X)}{\sqrt{2}}\geq\frac{1}{\sqrt{2}}.
$$

Once the first step is proved, observe that for any saddle connection $\gamma$ with $|\im(\gamma)|\leq 1$ and any $\alpha\in[-1,1]$ such that
$
\left|\alpha-\alpha_\gamma\right|<2\epsilon^2/|\im(\gamma)|
$, 
according to Lemma \ref{LemmaLengthHorocycle} we have
$$
\sys(G_{1}\cdot u_{-\alpha}\cdot X)
\leq
\left|
\hol(\gamma,G_{1}\cdot u_{-\alpha}\cdot X)
\right|
\leq
\sqrt{8}\epsilon.
$$
According to Minsky-Weiss estimate in Theorem \ref{TheoremMinskyWeiss} we have 
$$
\left|
[-1,1]\cap
\bigcup_{|\im(\gamma)|\leq 1}
B\big(\alpha_\gamma,\frac{2\epsilon^2}{|\im(\gamma)|}\big)
\right|
<
C\cdot
\left(
\frac{\sqrt{8}\epsilon}{S_0}
\right)^\beta
\cdot
\big|[-1,1]\big|.
$$
In the union above, any interval 
$
B(\alpha_\gamma,2\epsilon^2/|\im(\gamma)|)
$ 
has length at least $4\epsilon^2$. Divide $[-1,1]$ into $[K^{2}]$ intervals $J_i$ of equal size $|J_i|=2\epsilon^2$ and a remaining set of measure less than $2\epsilon^2$. Any $J_i$ has length less than half the length of any interval in the union, then the union of those $J_i$ which do not satisfy the required property has measure at most 
$
2(\sqrt{8}\epsilon/\rho_0)^\beta\cdot\big|[-1,1]\big|
$. 
The good ones are therefore at least
$$
\left(
1-
\bigg(
\frac{\sqrt{8}\epsilon}{S_0}
\bigg)^\beta
-
\epsilon^2
\right)
\cdot
\frac{\big|[-1,1]\big|}{2\epsilon^2}
\geq
\left(
1-
\bigg(
\frac{\sqrt{8}\epsilon}{S_0}
\bigg)^\beta
-
\epsilon^2
\right)
\cdot
K^2.
$$
\end{proof}

For convenience of notation, for any $j\geq1$ let $\Gamma(X,j)$ be the set of saddle connections $\gamma$ for the surface $X$ such that 
$
K^{j-1}/\sqrt{2}<|\im(\gamma)|\leq K^{j}/\sqrt{2}
$. 
Moreover let $\Gamma(X,0)$ be the set of saddle connections $\gamma$ with $|\im(\gamma)|\leq1/\sqrt{2}$. Set
$$
S_0':=\min\left\{S_0,\frac{1}{\sqrt{8}}\right\}.
$$

\begin{lemma}
\label{LemmaRecursionDecaying}
Let $J$ be an interval such that
\begin{eqnarray*}
&&
|J|=\frac{1}{K^{2n}}
\\
&&
J\cap
\bigcup_{j=0}^{n-1}
\bigcup_{\gamma\in\Gamma(X,j)}
B\big(\alpha_\gamma,\frac{\epsilon^2}{\sqrt{2}\cdot|\im(\gamma)|\cdot K^j}\big)
=\emptyset.
\end{eqnarray*}
Then we have 
$$
\left|
J\cap
\bigcup_{|\im(\gamma)|\leq K^n}
B\big(\alpha_\gamma,\frac{2\epsilon^2}{|\im(\gamma)|\cdot K^n}\big)
\right|
<
C\cdot
\bigg(
\frac{\sqrt{5}\epsilon}{S'_0}
\bigg)^\beta
\cdot
\big|J\big|.
$$
\end{lemma}

\begin{proof}
As in the previous Lemma, the first step in the proof is to show that for any saddle connection $\gamma$ we have
$$
\sup_{\alpha\in J}
\big|\hol(\gamma,G_{n+1}\cdot u_{-\alpha}\cdot X)\big|\geq 
\min\left\{S_0,\frac{1}{\sqrt{8}}\right\}.
$$
Let $\gamma$ be any saddle connection for $X$. According to Lemma \ref{LemmaLengthHorocycle}, for any $\alpha\in J$ we have
$$
\left|
\hol(\gamma,G_{n+1}\cdot u_{-\alpha}\cdot X)
\right|
\geq
K^{n+1}\cdot|\im(\gamma)|\cdot\big|\alpha-\alpha_\gamma\big|.
$$
Suppose that $\gamma\in\Gamma(X,j)$ for some $j$ with $0\leq j\leq n-1$. For any $\alpha\in J$ we have
$$
|\alpha-\alpha_\gamma|>\frac{1}{\sqrt{2}\cdot|\im(\gamma)|\cdot K^{j+2}}
$$ 
and thus 
$$
\left|
\hol(\gamma,G_{n+1}\cdot u_{-\alpha}\cdot X)
\right|
\geq
\frac{K^{n+1}}{\sqrt{2}\cdot K^{j+2}} 
\geq \frac{1}{\sqrt{2}}\geq S_0.
$$
Otherwise, if $|\im(\gamma)|> K^{n-1}/\sqrt{2}$, choose $\alpha\in J$ such that 
$
|\alpha-\alpha_\gamma|\geq |J|/2
$. 
For such $\alpha$ we have
$$
\left|
\hol(\gamma,G_{n+1}\cdot u_{-\alpha}\cdot X)
\right|
\geq
K^{n+1}\cdot|\im(\gamma)|\cdot \frac{|J|}{2}
>
|\im(\gamma)|\cdot \frac{K^{n+1}}{2K^{2n}}\geq 
\frac{1}{\sqrt{8}}.
$$

Once the first step is completed, observe that for any saddle connection $\gamma$ such that 
$|\im(\gamma)|\leq K^n$ and for any real number $\alpha$ such that
$$
|\alpha-\alpha_\gamma|\leq\frac{2\epsilon^2}{K^n|\im(\gamma)|},
$$
according to Lemma \ref{LemmaLengthHorocycle} we have
\begin{eqnarray*}
&&
\sys\big(G_{n+1}\cdot u_{-\alpha}\cdot X\big)
\leq
\left|
\hol(\gamma,G_{n+1}\cdot u_{-\alpha}\cdot X)
\right|
=
\\
&&
\sqrt{
\big(
K^{n+1}\cdot|\im(\gamma)|\cdot\big|\alpha-\alpha_\gamma\big|
\big)^2
+
\big(
\frac{|\im(\gamma)|}{K^{n+1}}
\big)^2
}
\leq\sqrt{5}\epsilon.
\end{eqnarray*}
Then the Lemma follows according to Theorem \ref{TheoremMinskyWeiss}.
\end{proof}

\subsubsection{Proof of Proposition \ref{PropDecaying}}

Let $J_i\subset[-1,1]$ be the intervals given by Lemma \ref{LemmaInitiationDecaying}, which are at least 
$
(1-M\epsilon^\beta)\cdot K^2
$, 
and for any such $J_i$ let $I_i\subset[-\pi/4,\pi/4[$ be its image under the function 
$
\alpha\mapsto\arctan(\alpha)
$. 
Observe that if $\theta$ is a direction of a saddle connection $\gamma$ with $|\im(\gamma)|\leq1$ then we have $l(\theta)\leq1$. According to the properties of the intervals $J_i$ and to Lemma \ref{LemmaDistortionArctan}, any $I_i$ satisfies Condition \eqref{EqDefDecayingAssumption} for $n=1$. 

Consider any interval $I$ satisfying the same Condition for some $n\geq 1$, that is
$$
|I|=\frac{1}{K^{2n}}
\quad\textrm{ and }\quad
I
\cap
\bigcup_{j=0}^{n-1}
\bigcup_{\theta_\gamma\in\cR^{sc}(K,j)}
B\left(\theta_\gamma,\frac{\epsilon^2}{l(\theta_\gamma)\cdot K^j}\right)
=\emptyset
$$
Let $J$ be the image of $I$ under the function $\theta\mapsto\tan(\theta)$ and observe that $|J|\geq K^{-2n}$, since the function $\tan(\cdot)$ has derivative bigger than $1$. Consider any $\theta\in I$ and let $\alpha:=\tan(\theta)$. If there exist some $j$ with $0\leq j\leq n-1$ and some $\gamma\in\Gamma(X,j)$ such that 
$$
|\alpha-\alpha_\gamma|\leq 
\frac{\epsilon^2}{\sqrt{2}|\im(\gamma)|K^j}
$$ 
then the direction $\theta_\gamma=\arctan(\alpha_\gamma)$ of $\gamma$ satisfies  
$
l(\theta_\gamma)\leq |\gamma|\leq \sqrt{2}|\im(\gamma)|\leq K^j
$, 
so that $\theta_\gamma\in\cR^{sc}(K,i)$ for some $i\leq j$, and moreover we have
$$
|\theta-\theta_\gamma|<
|\alpha-\alpha_\gamma|\leq
\frac{\epsilon^2}{\sqrt{2}|\im(\gamma)|K^j}
\leq
\frac{\epsilon^2}{l(\theta_\gamma)K^i},
$$
which is absurd by the assumption on $I$. According to Lemma \ref{LemmaRecursionDecaying} we have
$$
\left|
J\cap
\bigcup_{|\im(\gamma)|\leq K^n}
B\big(\alpha_\gamma,\frac{2\epsilon^2}{|\im(\gamma)|\cdot K^n}\big)
\right|
<
C\cdot
\bigg(
\frac{\sqrt{5}\epsilon}{S_0'}
\bigg)^\beta
\cdot
\big|J\big|
$$
Observe that the set of directions $\cR^{sc}(K,j)$ is contained into the set of all the directions $\theta_\gamma=\arctan(\alpha_\gamma)$ of saddle connections with $|\im(\gamma)|\leq K^n$. According to Lemma \ref{LemmaDistortionArctan} we have
$$
\left|
I\cap
\bigcup_{\theta_\gamma\in\cR^{sc}(K,n)}
B\big(\theta_\gamma,\frac{\epsilon^2}{l(\theta_\gamma)\cdot K^n}\big)
\right|
<
2
C\cdot
\bigg(
\frac{\sqrt{5}\epsilon}{S'_0}
\bigg)^\beta
\cdot
\big|I\big|
$$
Proposition \ref{PropDecaying} is proved.

\section{Bounded geodesics in moduli space}
\label{BoundedGeodesicsModuliSpaceProofs}

\subsection{Proof of Theorem \ref{thmDynamicalJarnickInequality}}
\label{SecBoundedGeodesicsProofs}

Fix a translation surface $X$ and let $(\cR^{sc},l)$ be the resonant set corresponding to saddle connections of $X$ as in \S~\ref{SecIntroHolonomyResonantSets}. In this section we prove Theorem \ref{thmDynamicalJarnickInequality}. An intermediate step in the proof is an analogous statement for the set $\bad(\cR^{sc},\epsilon)$, namely Theorem \ref{ThmJarnikInequalityTranslSurf} in \S~\ref{SecProofThmJarnikInequalityTranslSurf} below, which is itself an immediate application of Theorem \ref{TheoremMetricResultsHolonomyResonantSets} and Theorem \ref{TheoremAbstractJarnickInequality}. The second step in the proof is Lemma \ref{LemmaDaniCorrespondenceFlat} below, which is an adaptation of Proposition 1.1 in \cite{HubertMarcheseUlcigrai} to the language of resonant sets and gives a relation between the sets $\bad^{dyn}(X,\epsilon)$ and $\bad(\cR^{sc},\epsilon)$. Fix a direction $\theta$ on the translation surface $X$. For any saddle connection $\gamma$ on the surface $X$ we write 
$$
\hol(\gamma,r_\theta\cdot X)=
\big(\re(\theta,\gamma),\im(\theta,\gamma)\big).
$$

\begin{lemma}
\label{LemmaDaniCorrespondenceFlat}
Let $X$ be any translation surface and $\theta$ be a direction on $X$. We have
$$
\liminf_{l(\theta_\gamma)\to\infty}|\theta-\theta_\gamma|\cdot l(\theta_\gamma)^2=
\liminf_{|\im(\theta,\gamma)|\to\infty}|\re(\theta,\gamma)|\cdot|\im(\theta,\gamma)|=
\frac{1}{2}
\liminf_{t\to+\infty}\sys(g_tr_\theta\cdot X)^2.
$$
\end{lemma}

\begin{proof}
The second equality corresponds to Proposition 1.1 in \cite{HubertMarcheseUlcigrai}. In order to see the first equality, let $\theta_\gamma\in\cR^{sc}(X)$ be the direction of a saddle connection $\gamma$ and observe that we have 
$
|\re(\theta,\gamma)|=
|\gamma|\cdot\left|\sin(\theta-\theta_\gamma)\right|
$ 
and
$
|\im(\theta,\gamma)|=
|\gamma|\cdot\left|\cos(\theta-\theta_\gamma)\right|
$, 
therefore 
$$
|\re(\theta,\gamma)|\cdot|\im(\theta,\gamma)|=
|\gamma|^2\sin\big(|\theta-\theta_\gamma|\big)
\cos\big(|\theta-\theta_\gamma|\big)=
|\gamma|^2\frac{\sin\big(2|\theta-\theta_\gamma|\big)}{2}.
$$ 
Moreover fix a constant $a_0>0$ and consider $L>0$ arbitrarily big. If $\gamma$ is a saddle connection  such that we have at the same time $|\im(\theta,\gamma)|\geq L$ and  
$
|\re(\theta,\gamma)|\cdot|\im(\theta,\gamma)|\leq a_0
$, 
then we must have obviously 
$
|\gamma|\geq|\im(\theta,\gamma)|\geq L
$ 
and 
$$
|\theta-\theta_\gamma|
\leq
\tan\big(|\theta-\theta_\gamma|\big)
=
\frac{|\re(\theta,\gamma)|}{|\im(\theta,\gamma)|}
\leq
\frac{a_0}{L^2}.
$$
On the other hand if we have at the same time $|\gamma|\geq L$ and 
$
|\theta-\theta_\gamma|\cdot |\gamma|^2\leq a_0
$, 
we must have 
$
|\im(\theta,\gamma)|\geq|\gamma|/\sqrt{2}>L/\sqrt{2}
$ 
and 
$$
|\re(\theta,\gamma)|\cdot|\im(\theta,\gamma)|
\leq
|\theta-\theta_\gamma|\cdot|\gamma|^2\leq a_0.
$$
The first equality follows. 
\end{proof}

Theorem \ref{thmDynamicalJarnickInequality} follows immediately from Theorem \ref{ThmJarnikInequalityTranslSurf} in the next subsection and from Corollary \ref{CorollaryInclusionsNotionBad} below.

\begin{corollary}
\label{CorollaryInclusionsNotionBad}
Let $X$ be any translation surface. For any $\epsilon>0$ we have
$$
\bigcup_{\epsilon'>\epsilon/\sqrt{2}}
\bad\left(\cR^{sc},\epsilon'\right)
\subset
\bad^{dyn}(X,\epsilon)
\subset
\bad\left(\cR^{sc},\frac{\epsilon}{\sqrt{2}}\right).
$$
\end{corollary}

\subsection{Hausdorff dimension of $\bad(\cR^{sc},\epsilon)$ for a translation surface $X$}
\label{SecProofThmJarnikInequalityTranslSurf}

Consider a translation surface $X$ with $\area(X)=1$ and total multiplicity at conical singularities $m$ and let $\cH$ be its stratum. If $X$ is a Veech surface, let $\cM:=\sltwor\cdot X$ be the its closed orbit under the action of $\sltwor$. Let $(\cR^{sc},l)$ be the resonant set corresponding to saddle connections of $X$ as in \S~\ref{SecIntroHolonomyResonantSets}. 

\begin{theorem}
\label{ThmJarnikInequalityTranslSurf}
There exist constants $\epsilon_0>0$, $0<\beta\leq 1$, $c_u>0$ and $c_l>0$ which depend only on the integer $m$, such that for any $\epsilon$ with $0<\epsilon< \min\{\epsilon_0,\sys(X)\}$ we have
$$
1-c_l\cdot\frac{\epsilon^{\beta}}{|\log\epsilon|}
\leq
\dim\big(\bad(\cR^{sc},\epsilon)\big)
\leq
1-c_u \cdot\frac{\epsilon}{|\log\epsilon|}.
$$
In general we have $\beta=(3m-1)^{-1}$. Moreover, if $X$ is a Veech surface the same estimate holds with $\beta=1$ and with some $\epsilon_0$ depending only on $\cM$.
\end{theorem}

\begin{proof}
Fix any $\epsilon$ as in the statement. The statement follows combining Theorem \ref{TheoremMetricResultsHolonomyResonantSets} and Theorem \ref{TheoremAbstractJarnickInequality}. In order to prove the upper bound, observe that $\tau\leq|\log(1-\tau)|$ for any $\tau\geq0$ by convexity of the logarithm function. Therefore, since the resonant set $\cR^{sc}$ satisfies $(\epsilon,K,\tau)$-Decaying with constants $U=12/(m^2\epsilon^2)$ and $\tau=m\epsilon^2/\sqrt{48}$, we have
$$
\dim\big(\bad(\cR^{sc},\epsilon)\big)
\leq
1-
\frac
{|\log(1-\tau)|}
{|\log(\epsilon^2/(5U))|}
\leq
1-\frac{m\epsilon^2/\sqrt{48}}{|\log(m^2\epsilon^4/60)|}
\leq
1-\frac{m}{4\sqrt{48}}\frac{\epsilon^2}{|\log(\epsilon)|}.
$$
In order to prove the lower bound, observe that there is some universal $\tau_0>0$ such that $|\ln(1-\tau)|>\tau/2$ for $0\leq\tau\leq \tau_0$. Since the resonant set $(\cR^{sc},l)$ also satisfies $(\epsilon,\tau)$-decaying with $\tau=M\epsilon^\beta$, for $\epsilon$ small enough (in terms of $\tau_0$ and of the constants in the explicit form of $\tau$) we have
$$
\dim
\big(\bad(\cR^{sc},\epsilon)\big)
\geq
1-\frac{|\log(1-\tau-\epsilon^{4/3})|}{4/3|\log(\epsilon)|}
\geq
1-\frac{M}{2\cdot 4/3}\frac{\epsilon^\beta}{|\log(\epsilon)|}.
$$
The proof of the last inequality in case of Veech surfaces, where $\beta=1$, is similar.
\end{proof}

\section{Unbounded geodesics in moduli space}
\label{ChapterExcursionsGeodesicsParameterSpace}

Fix a translation surface $X$ and a direction $\theta$ on $X$. Recall that for any saddle connection/closed geodesic $\gamma$ on the surface $X$ we write 
$$
\hol(\gamma,r_\theta\cdot X)=
\big(\re(\theta,\gamma),\im(\theta,\gamma)\big).
$$
In this section we often consider the positive instant $t(\theta,\gamma)\in\RR$ such that
$$
e^{t(\theta,\gamma)}|\re(\gamma,\theta)|=
e^{-t(\theta,\gamma)}|\im(\gamma,\theta)|.
$$
The length 
$
|\hol(\gamma,g_tr_\theta\cdot X)|
$ 
is minimal for $t=t(\theta,\gamma)$, and the minimal value is  
$$
|\hol(\gamma,g_{t(\theta,\gamma)}r_\theta\cdot X)|
=
\sqrt{2|\re(\theta,\gamma)|\cdot|\im(\theta,\gamma)|}.
$$

\subsection{Khinchin-Jarn\'ik Theorem for cylinders and saddle connections}

\begin{theorem}
\label{ThmKhinchinJarnickTranslationSurfaces}
Let $X$ be any translation surface and consider an approximation function $\psi$ and a dimension function $f$ such that $t\mapsto tf\circ\psi(t)$ is decreasing monotone for $t>0$. Let $\cR$ denote indifferently $\cR^{sc}$ of $\cR^{cyl}$.
\begin{enumerate}
\item
If
$
\int_0^{+\infty}tf\big(\psi(t)\big)dt
$ 
converges as $t\to+\infty$, then $H^f\big(W(\cR,\psi)\big)=0$. Consequently, for any 
$
\theta\not\in W(\cR,\psi)
$ 
and for all saddle connections/closed geodesic $\gamma$ long enough we have
$$
|\re(\theta,\gamma)|>|\gamma|\psi\big(|\gamma|\big).
$$
\item
If
$
\int_0^{+\infty}tf\big(\psi(t)\big)dt
$ 
diverges as $t\to+\infty$, then 
$
H^f\big(W(\cR,\psi)\big)=H^f\big([-\pi/2,\pi/2]\big)
$. 
Consequently, for any $\theta\in W(\cR,\psi)$ there exist infinitely many saddle connections/closed geodesic $\gamma$ such that
$$
|\re(\theta,\gamma)|<
|\gamma|\psi\big(|\gamma|\big)
$$
\end{enumerate}
\end{theorem}

\begin{proof}
Both $\cR^{sc}$ and $\cR^{cyl}$ satisfy quadratic growth, thus in the first part of the statement we have 
$
H^f\big(W(\cR,\psi)\big)=0
$ 
both for $\cR=\cR^{sc}$ and $\cR=\cR^{cyl}$, according to Theorem \ref{thmabstractkhinchinjarnik}. For any $\theta\not\in W(\cR,\psi)$ and any saddle connection/closed geodesic $\gamma$ long enough we have
$$
|\re(\gamma,\theta)|=|\gamma|\cdot\sin\big(|\theta-\theta_\gamma|\big)\geq
0.5\cdot |\gamma|\cdot|\theta-\theta_\gamma|\geq
0.5\cdot |\gamma|\cdot\psi\big(l(\theta_\gamma)\big)\geq
0.5\cdot |\gamma|\cdot\psi\big(|\gamma|\big),
$$
where the first inequality holds because $|\sin(x)|\geq 0.5\cdot|x|$ whenever $|x|\leq \pi/2$, and the last one holds because $\psi(\cdot)$ is decreasing monotone and 
$
l(\theta_\gamma)\leq|\gamma|
$ 
for any $\gamma$. The first part of the statement follows replacing the approximation function $\psi(\cdot)$ by $2\psi(\cdot)$, which satisfies the same convergence assumption.

In order to prove the second part of the statement, observe that Theorem \ref{thmabstractkhinchinjarnik} and Theorem \ref{TheoremMetricResultsHolonomyResonantSets} imply 
$
H^f\big(W(\cR^{cyl},\psi)\big)=H^f\big([-\pi/2,\pi/2[\big)
$. 
Then according to Equation \eqref{EqInclusionResonantSets} we have also
$
H^f\big(W(\cR^{sc},\psi)\big)=H^f\big([-\pi/2,\pi/2[\big)
$. 
Both for $\cR=\cR^{sc}$ and $\cR=\cR^{cyl}$, and for any $\theta\in W(\cR,\psi)$ there exist infinitely many saddle connections/closed geodesics $\gamma$ in direction $\theta_\gamma\in\cR$ such that 
$$
|\re(\gamma,\theta)|=|\gamma|\cdot\sin\big(|\theta-\theta_\gamma|\big)\leq
|\gamma|\cdot|\theta-\theta_\gamma|\leq
|\gamma|\cdot\psi\big(l(\theta_\gamma)\big)=
|\gamma|\cdot\psi\big(|\gamma|\big).
$$
Here the last equality holds because we can assume $|\gamma|=l(\theta_\gamma)$ for all $\gamma$. The second part of the statement is proved.
\end{proof}

\subsection{Proof of Theorem \ref{thmDynamicalKhinchin}}
\label{SecProofDynamicalKhinchin}

In this subsection we use the following elementary Lemma, whose proof is left to the reader.

\begin{lemma}
\label{LemChangeVariableLogLaws}
Consider a decreasing function $\varphi:\RR_+\to\RR_+$ and $a>0$, then define a function 
$
\psi:(1,+\infty)\to\RR_+
$ 
by
$$
\psi(s):=\frac{1}{s^2}\cdot\varphi\left(\frac{\ln s}{a}\right).
$$
Then the function $s\mapsto s^2\psi(s)$ is decreasing monotone and for any $t>0$ we have
$$
\varphi(t)=e^{2at}\psi(e^{at}).
$$ 
Moreover $\int_0^{+\infty}\varphi(t)dt$ diverges at $t=+\infty$ if and only if 
$
\int_1^{+\infty}s\psi(s)ds
$ 
diverges at $s=+\infty$.
\end{lemma}

\subsubsection{Proof of convergent case}

\begin{lemma}
\label{LemmaDynamicalKhinchinConvergent}
Let $\varphi:\RR_+\to\RR_+$ be a decreasing function such that $\int_0^{+\infty}\varphi(t)dt$ converges at $t=+\infty$. Then for almost any $\theta$ we have
$$
\liminf_{t\to\infty}
\frac{\sys(g_tr_\theta\cdot X)}{\sqrt{\varphi(t)}}\geq1.
$$
\end{lemma}

\begin{proof}
Consider the function $\psi$ associated to $\varphi$ and to the parameter $a=1$ by Lemma \ref{LemChangeVariableLogLaws}. Since
$
\int_0^{+\infty}\varphi(t)dt
$
converges at $t=+\infty$ then
$
\int_1^{+\infty}s\psi(s)ds
$ 
converges at $s=+\infty$. Let $\cW$ be the set of directions $\theta$ such that there exist arbitrarily big instants $t>0$ with
$$
\sys(g_tr_\theta\cdot X)^2\leq\varphi(t).
$$
Fix any $\theta\in\cW$. For any $t$ as above, let $\gamma$ be the saddle connection for the surface $X$ such that
$$
\sys(g_tr_\theta\cdot X)=|\hol(\gamma,g_tr_\theta\cdot X)|.
$$
For $t>0$ big enough we have $\varphi(t)<1$, thus it follows that $e^t>|\gamma|$, indeed we have 
$$
1>\sqrt{\varphi(t)}\geq|\hol(\gamma,g_tr_\theta\cdot X)|\geq|\gamma|e^{-t}.
$$
Fix $t$ and $\gamma$ as above. Recalling the minimality property of the instant $t(\theta,\gamma)$, we have
\begin{eqnarray*}
&&
|\re(\theta,\gamma)|\cdot |\gamma|
<
2|\re(\theta,\gamma)|\cdot |\im(\theta,\gamma)|
=
|\hol(\gamma,g_{t(\theta,\gamma)}r_\theta\cdot X)|^2
\leq
\\
&&
|\hol(\gamma,g_tr_\theta\cdot X)|^2\leq\varphi(t)
=
e^{2t}\psi(e^t)
\leq
|\gamma|^2\psi(|\gamma|),
\end{eqnarray*}
where the last inequality holds because $e^t>|\gamma|$ and the function $s\mapsto s\varphi(s)$ is decreasing monotone. Observe finally that for any $\theta$ and $\gamma$ as above we have 
$
|\re(\theta,\gamma)|\leq \varphi(t)/|\gamma|
$. 
Thus, since $t$ is arbitrarily big, the saddle connection $\gamma$ must be arbitrarily long, by discreteness of the set of values $\hol(\gamma,r_\theta\cdot X)$. It follows that for any $\theta$ as above there exists infinitely many saddle connections $\gamma_n$ such that
$
|\re(\theta,\gamma_n)|<|\gamma_n|\cdot\psi(|\gamma_n|)
$. 
Theorem \ref{ThmKhinchinJarnickTranslationSurfaces} implies $\leb(\cW)=0$. The Lemma is proved.
\end{proof}

Here we finish the proof of the convergent case of Theorem \ref{thmDynamicalKhinchin}. Let $\varphi$ be a function as in Lemma \ref{LemmaDynamicalKhinchinConvergent} and for any integer $n\geq 1$ consider the function $\varphi_n:=n\cdot\varphi$, which also satisfies the assumption of the Lemma. It follows that for any $n$ there exists a full measure set of directions $\theta$ such that 
$$
\frac{1}{n}
\liminf_{t\to\infty}
\frac{\sys(g_tr_\theta\cdot X)}{\sqrt{\varphi(t)}}=
\liminf_{t\to\infty}
\frac{\sys(g_tr_\theta\cdot X)}{\sqrt{\varphi_n(t)}}\geq1.
$$
The convergent case of Theorem \ref{thmDynamicalKhinchin} follows because the countable intersection of full measure sets has full measure.

\subsubsection{Proof of divergent case}

\begin{lemma}
\label{LemmaDynamicalKhinchinDivergent}
Let $\varphi:\RR_+\to\RR_+$ be a decreasing function such that $\int_0^{+\infty}\varphi(t)dt$ diverges at $t=+\infty$. Then for almost any $\theta$ we have
$$
\limsup_{t\to\infty}
\frac{\sys(g_tr_\theta\cdot X)}{\sqrt{\varphi(t)}}\leq\sqrt{2}.
$$
\end{lemma}

\begin{proof}
Observe first that according to the convergent case of Theorem \ref{ThmKhinchinJarnickTranslationSurfaces}, for almost any $\theta$ and for any saddle connection $\gamma$ long enough we have
$
|\re(\theta,\gamma)|>|\gamma|^{-1.02}
$. 
Fix any such $\theta$, let $\gamma$ be a saddle connection long enough and consider the instant $t(\theta,\gamma)$. We have
$$
e^{2t(\theta,\gamma)}=
\frac{|\im(\theta,\gamma)|}{|\re(\theta,\gamma)|}<
|\im(\theta,\gamma)|\cdot|\gamma|^{1.02}\leq \big(|\gamma|^{1.01}\big)^2.
$$
Consider the function $\psi$ associated to $\varphi$ and to the parameter $a:=1.02^{-1}\sim0.98$ by Lemma \ref{LemChangeVariableLogLaws}. Since
$
\int_0^{+\infty}\varphi(t)dt
$
diverges at $t=+\infty$ then
$
\int_1^{+\infty}s\psi(s)ds
$
diverges at $s=+\infty$. According to Theorem \ref{ThmKhinchinJarnickTranslationSurfaces}, for almost any $\theta$ there exist infinitely many saddle connections $\gamma$ such that
$$
|\re(\theta,\gamma)|<|\gamma|\cdot\psi\big(|\gamma|\big).
$$
According to the discussion at the beginning of the proof, we can also assume that for any such $\theta$ and $\gamma$, at the instant $t(\theta,\gamma)$ we have
$$
|\gamma|\geq e^{at(\theta,\gamma)}.
$$
For any such $\theta$ and $\gamma$, recalling that the function $s\mapsto s^2\psi(s)$ is decreasing monotone, we have
\begin{eqnarray*}
&&
\sys(g_{t(\theta,\gamma)}r_\theta\cdot X)^2\leq
|\hol(\gamma,g_{t(\theta,\gamma)}r_\theta\cdot X)|^2
=
2|\re(\theta,\gamma)|\cdot |\im(\theta,\gamma)|<
\\
&&
2|\gamma|^2\psi(|\gamma|)<
2e^{2at(\theta,\gamma)}\varphi(e^{at(\theta,\gamma)})=
2\varphi\big(t(\theta,\gamma)\big).
\end{eqnarray*}
Finally, observe that since $\gamma$ is arbitrarily long, then $|\re(\theta,\gamma)|$ is arbitrarily small, thus $t(\theta,\gamma)$ is arbitrarily big. The Lemma is proved.
\end{proof}

Here we finish the proof of the divergent case of Theorem \ref{thmDynamicalKhinchin}. Let $\varphi$ be a function as in Lemma \ref{LemmaDynamicalKhinchinDivergent} and for any integer $n\geq 1$ consider the function $\varphi_n:=n^{-1}\cdot\varphi$, which also satisfies the assumption of the Lemma. It follows that for any $n$ there exists a full measure set of directions $\theta$ such that 
$$
\sqrt{n}\cdot
\liminf_{t\to\infty}
\frac{\sys(g_tr_\theta\cdot X)}{\sqrt{\varphi(t)}}=
\liminf_{t\to\infty}
\frac{\sys(g_tr_\theta\cdot X)}{\sqrt{\varphi_n(t)}}\leq\sqrt{2}.
$$
The divergent case of Theorem \ref{thmDynamicalKhinchin} follows because the countable intersection of full measure sets has full measure.

\subsection{Proof of Theorem \ref{thmLogLaws(Diophantine)}}
\label{SecProofGenLogLaws}

In this section we follow \S~3.1 of \cite{velaniubiquity}. Fix a translation surface $X$. Fix $\tau>2$ and $\epsilon\geq0$ and consider the set of directions $W(\tau,\epsilon)$ defined by
$$
W(\tau,\epsilon)=
W(\cR^{sc},\psi_{\tau,\epsilon})
$$
for the approximation function
$$
\psi_{\tau,\epsilon}(t):=
\frac{1}{t^\tau(\ln t)^{(1+\epsilon)\tau/2}}.
$$
In particular denote $W(\tau):=W(\tau,\epsilon=0)$. Consider also the dimension function $f(r):=r^{2/\tau}$, so that $H^f=H^{2/\tau}$, that is the standard Hausdorff measure of parameter $2/\tau$, and moreover
$$
r\cdot f\circ\psi_{\tau,\epsilon}(r)=\frac{1}{r(\ln r)^{1+\epsilon}},
$$
so that 
\begin{eqnarray*}
&&
\int_0^{+\infty}r\cdot f\circ\psi_{\tau,\epsilon}(r)dr
\quad
\textrm{ diverges at }
\quad
r=+\infty
\quad
\textrm{ for any }
\epsilon>0
\\
&&
\int_0^{+\infty}r\cdot f\circ\psi_{\tau,\epsilon=0}(r)dr
\quad
\textrm{ converges at }
\quad
r=+\infty.
\\
\end{eqnarray*}

\begin{lemma}
\label{LemLenghtTimeLowLaws}
If $\theta\not\in W(\tau,\epsilon)$ then for any saddle connection $\gamma$ long enough we have
$$
2t(\theta,\gamma)
<
\tau\ln|\gamma|
+
(1+\epsilon)\frac{\tau}{2}\ln\big(\ln|\gamma|\big).
$$
\end{lemma}

\begin{proof}
According to the definition of $W(\tau,\epsilon)$ and to Theorem \ref{ThmKhinchinJarnickTranslationSurfaces}, for any saddle connection $\gamma$ long enough we have
$$
|\re(\gamma,\theta)|>
\frac{1}{|\gamma|^{\tau-1}(\ln|\gamma|)^{(1+\epsilon)\tau/2}}.
$$
therefore the Lemma follows observing that for such $\gamma$ we have
$$
e^{2t(\theta,\gamma)}
=
\frac
{|\im(\theta,\gamma)|}
{|\re(\theta,\gamma)|}
<
\frac
{|\gamma|}
{|\re(\theta,\gamma)|}
<
|\gamma|^{\tau}(\ln|\gamma|)^{(1+\epsilon)\tau/2}.
$$
\end{proof}

\begin{corollary}
\label{CoroLenghtTimeLogLaw}
If
$
\theta\not\in W(\tau,\epsilon)
$
then for any $\gamma$ long enough we have
$$
\ln|\gamma|
>
\frac{t(\theta,\gamma)}{\tau}.
$$
\end{corollary}

\begin{proof}
We have
$
\ln|\gamma|>(1+\epsilon)/2\ln(\ln|\gamma|)
$
for any saddle connection $\gamma$. According to Lemma \ref{LemLenghtTimeLowLaws}, for any saddle connection $\gamma$ long enough we have
$$
\ln|\gamma|
>
\frac{1}{2}
\bigg(
\ln|\gamma|+\frac{1+\epsilon}{2}\ln(\ln|\gamma|)
\bigg)=
\frac{1}{2\tau}
\bigg(
\tau\ln|\gamma|+(1+\epsilon)\frac{\tau}{2}\ln(\ln|\gamma|)
\bigg)
>
\frac{t(\theta,\gamma)}{\tau}.
$$
\end{proof}

\subsubsection{End of the proof}

Recall that according to Theorem \ref{ThmKhinchinJarnickTranslationSurfaces} we have
$
H^{1-\alpha}\big(W(\tau)\big)=+\infty
$
and
$
H^{1-\alpha}\big(W(\tau,\epsilon)\big)=0
$
for any $\epsilon>0$, where we introduce the parameter
$$
\alpha:=1-\frac{2}{\tau}.
$$
Theorem \ref{thmLogLaws(Diophantine)} follows from Proposition \ref{PropLogLawsDivergent} and Proposition \ref{PropLogLawsConvergent} below.

\begin{proposition}
\label{PropLogLawsDivergent}
Consider
$
\theta\in
W(\tau)
\setminus
\bigcup_{\epsilon>0}W(\theta,\epsilon)
$.
We have
$$
\limsup_{t\to\infty}
\frac{-\ln\big(\sys(g_tr_\theta\cdot X)\big)-(1-2/\tau)t}
{\ln t}
\geq\frac{1}{2}.
$$
\end{proposition}

\begin{proof}
Since $\theta\in W(\tau)$ then, according to Theorem \ref{ThmKhinchinJarnickTranslationSurfaces}, there exists an arbitrarily long saddle connection $\gamma$ such that
$$
|\re(\gamma,\theta)|<
\frac{1}{|\gamma|^{\tau-1}(\ln|\gamma|)^{\tau/2}}.
$$
Since $\gamma$ is arbitrarily long, then $t(\theta,\gamma)$ is also arbitrarily big. Moreover the minimizing property of the instant $t=t(\theta,\gamma)$ gives
$$
\frac{1}{2}|\hol(\gamma,g_tr_\theta\cdot X)|^2=
|\re(\theta,\gamma)|\cdot |\im(\theta,\gamma)|<
\frac{1}{|\gamma|^{\tau-2}(\ln|\gamma|)^{\tau/2}}.
$$
Fix $\epsilon>0$ and recall that $\theta\not\in W(\tau,\epsilon)$. Without loss of generality $\gamma$ can be assumed to be long enough to satisfy part (2) of Theorem \ref{ThmKhinchinJarnickTranslationSurfaces}. Then, according to the previous inequality and to Lemma \ref{LemLenghtTimeLowLaws}, for $t=t(\theta,\gamma)$ we get
\begin{eqnarray*}
&&
-2\ln\big(|\hol(\gamma,g_tr_\theta\cdot X)|\big)
>
(\tau-2)\ln|\gamma|+\frac{\tau}{2}\ln(\ln|\gamma|)-\ln2
\geq
\\
&&
\frac{\tau-2}{\tau}
\bigg(
2t(\theta,\gamma)-(1+\epsilon)\frac{\tau}{2}\ln(\ln|\gamma|)
\bigg)
+
\frac{\tau}{2}\ln(\ln|\gamma|)-\ln2=
\\
&&
\left(1-\frac{2}{\tau}\right)2t(\theta,\gamma)
+
\bigg(
1+\epsilon-\frac{\tau\epsilon}{2}
\bigg)
\ln(\ln|\gamma|)-\ln2.
\end{eqnarray*}
Finally, according to Corollary \ref{CoroLenghtTimeLogLaw}, for $t=t(\theta,\gamma)$ we obtain
$$
-\ln\big(|\hol(\gamma,g_tr_\theta\cdot X)|\big)
>
\left(1-\frac{2}{\tau}\right)t
+
\bigg(
\frac{1}{2}-\frac{\epsilon}{2}+\frac{\tau\epsilon}{4}
\bigg)
\ln t+c,
$$
where $c$ is a constant depending only on $\theta$ and $\tau$. Since $t=t(\theta,\gamma)$ is arbitrarily big we have
$$
\limsup_{t\to+\infty}
\frac{-\ln\big(\sys(g_tr_\theta\cdot X)\big)-(1-2/\tau)t}
{\ln t}
\geq
\frac{1}{2}-\frac{\epsilon}{2}+\frac{\tau\epsilon}{4}.
$$
The Proposition follows because the last estimate holds for all $\epsilon>0$.
\end{proof}

\begin{proposition}
\label{PropLogLawsConvergent}
Fix $\tau\geq2$ and $\epsilon>0$. Let $\theta$ be a direction such that
\begin{equation}
\label{Eq1PropProofLogLawUpperEstimate}
\limsup_{t\to\infty}
\frac
{-\ln\big(\sys(g_tr_\theta\cdot X)\big)-\big(1-2/\tau\big)t}
{\ln t}
>\frac{1+\epsilon}{2}.
\end{equation}
Then we have
$$
\theta\in W(\tau,\epsilon).
$$
\end{proposition}

\begin{proof}
Consider a direction $\theta$ such that Equation \eqref{Eq1PropProofLogLawUpperEstimate} holds. According to the assumption, there exists $t$ arbitrarily big such that
\begin{equation}
\label{Eq2PropProofLogLawUpperEstimate}
-\ln\big(\sys(g_tr_\theta\cdot X)\big)>
\bigg(1-\frac{2}{\tau}\bigg)t+\frac{1+\epsilon}{2}\ln t.
\end{equation}
Consider a saddle connection $\gamma=\gamma(t)$ such that
$
\sys(g_tr_\theta\cdot X)=|\hol(\gamma,g_tr_\theta\cdot X)|
$.
Observe that for such saddle connection $\gamma$ we have
$$
|\hol(\gamma,g_tr_\theta\cdot X)|
\geq
e^{-t}|\gamma|,
$$
that is
$
\ln\big(|\hol(\gamma,g_tr_\theta\cdot X)|\big)>\ln|\gamma|-t
$,
therefore we get
$$
\frac{2t}{\tau}-\ln|\gamma|
>
\frac{1+\epsilon}{2}\ln t.
$$
In particular we have $e^t>|\gamma|$, since $|\hol(\gamma,g_tr_\theta\cdot X)|<1$. It follows that
\begin{eqnarray*}
&&
(\tau-2)
\bigg(
\frac{2t}{\tau}-\ln|\gamma|
\bigg)
+
(1+\epsilon)\ln t
-
(1+\epsilon)\frac{\tau}{2}\ln\ln|\gamma|>
\\
&&
\frac{(\tau-2)(1+\epsilon)}{2}\ln t
+
(1+\epsilon)\ln t
-
\frac{\tau+\epsilon\tau}{2}\ln\ln|\gamma|
=
\frac{\tau+\epsilon\tau}{2}\big(\ln t-\ln\ln|\gamma|\big)>0.
\end{eqnarray*}
Resuming we have
\begin{equation}
\label{Eq3PropProofLogLawUpperEstimate}
\bigg(1-\frac{2}{\tau}\bigg)2t
+
(1+\epsilon)\ln t
>
(\tau-2)\ln|\gamma|
+
(1+\epsilon)\frac{\tau}{2}\ln\ln|\gamma|.
\end{equation}
Therefore, for a direction $\theta$ and an instant $t$ as in Equation \eqref{Eq2PropProofLogLawUpperEstimate} and for a saddle connection $\gamma$ such that
$
|\hol(\gamma,g_tr_\theta\cdot X)|=\sys(g_tr_\theta\cdot X)
$,
according to Equation \eqref{Eq3PropProofLogLawUpperEstimate} above, we have
$$
-2\ln|\hol(\gamma,g_tr_\theta\cdot X)|
>
(\tau-2)\ln|\gamma|
+
(1+\epsilon)\frac{\tau}{2}\ln\ln|\gamma|
$$
which implies
$$
|\im(\theta,\gamma)|\cdot|\re(\theta,\gamma)|
\leq
\frac{|\hol(\gamma,g_tr_\theta\cdot X)|^2}{2}
<
\frac{1}{2}
\frac{1}
{|\gamma|^{\tau-2}
\big(\ln|\gamma|\big)^{(1+\epsilon)\tau/2}},
$$
that is, observing that $|\im(\theta,\gamma)|>|\gamma|/2$, we have
$$
|\re(\theta,\gamma)|
\leq
\frac{1}{4}
\frac{1}
{|\gamma|^{\tau-1}
\big(\ln|\gamma|\big)^{(1+\epsilon)\tau/2}}.
$$
The last condition holds for a saddle connection $\gamma$ which can be chosen arbitrarily long, therefore we have $\theta\in W(\tau,\epsilon)$.
\end{proof}

\section{Fast recurrence in rational billiard: proof of Theorem \ref{ThmRecurrenceBilliads}}
\label{ChapterRecurrencePhaseSpace}

\subsection{The recurrence rate function}
\label{SecRecurrenceRateFunction}

Let $X$ be a translation surface and $\Sigma$ be the set of its conical singularities. Let $\theta$ be a direction on the surface $X$ such that there are not saddle connections in direction $\theta$, and thus nor closed geodesics. The recurrence rate function 
$
\omega_\theta:X\to[0,+\infty]
$ 
is defined for any $p\in X$ by
$$
\omega_\theta(p):=
\liminf_{r\to0}
\frac{\log\big(R_\theta(p,r)\big)}{-\log r}
$$
where
$
R_\theta(p):=\min\{t>r\textrm{ ; }|\phi^t_\theta(p)-p|<r\}
$.
More precisely, the function $\omega_\theta$ is defined on points which are recurrent for $\phi^t_\theta$ with $t>0$, and the set of such points is equal to the set of points such that $\phi^t_\theta(p)$ is defined for all $t>0$ (see \S~3 in \cite{yoccozIET}). In particular $\omega_\theta(p)$ is defined on an open subset of $X$ with full Lebesgue measure, which is invariant under the flow $\phi^t_\theta$. Moreover, fix $t>0$ and consider a point $p\in X$ such that $\omega_\theta(p)$ is defined. Since the domain of $\omega_\theta$ is open, then there exists $r>0$ such that $\phi^t_\theta$ acts as a translation on the ball $B(p,r)$, that is
$
\phi^s_\theta\big(B(p,r)\big)
$
does not contain any conical singularity for $0\leq s\leq t$.  Therefore we have
$
R_\theta(\phi^t_\theta(p),r)=R_\theta(p,r)
$
and thus
\begin{equation}
\label{EqFlowInvarianceRecurrenceRate}
\omega_\theta\big(\phi_\theta^t(p)\big)
=
\liminf_{r\to0}
\frac
{\log\big(R_\theta(\phi^t_\theta(p),r)\big)}
{-\log r}
=
\liminf_{r\to0}
\frac
{\log\big(R_\theta(p,r)\big)}
{-\log r}
=
\omega_\theta(p).
\end{equation}

\subsection{Recurrence rate and diophantine approximations}

Let $\cR^{sc}$ and $\cR^{cyl}$ be the resonant sets defined in \S~\ref{SecIntroHolonomyResonantSets} for the translation surface $X$. For $\tau>2$ consider the function $\psi_\tau(r):=r^{-\tau}$ and define the sets of directions
$$
\cW^{sc}(\tau):=W\big(\cR^{sc},\psi_\tau\big)
\quad
\textrm{ and }
\quad
\cW^{sc}(\tau):=W\big(\cR^{sc},\psi_\tau\big).
$$

\begin{lemma}
\label{lem1relationnontypes}
Fix a direction $\theta$ on the surface $X$. Consider a point $p\in X\setminus\Sigma$ and an instant $T>0$ such that $\phi^T_\theta(p)$ is connected to $p$ by an horizontal segment $H$ with length $|H|\leq \sys(X)$. Then there exists a saddle connection $\gamma$ such that
$$
|\re(\theta,\gamma)|\leq |H|
\quad
\textrm{ and }
\quad
|\im(\theta,\gamma)|<T.
$$
Moreover, if $T_n\to+\infty$ is a sequence of instants as above and $\gamma_n$ is the sequence of the corresponding saddle connections, we have
$$
|\im(\theta,\gamma_n)|\to+\infty.
$$
\end{lemma}

\begin{proof}
Let $H$ be an horizontal segment as in the first part of the statement and assume without loss of generality that $\phi_\theta^T(p)$ is its left endpoint and $p$ is its right endpoint. Then let $H'$ be the horizontal segment with length $|H'|=|H|$ and left endpoint $p$. Assume that for any point $p'\in H$ and any $t$ with
$
0\leq t\leq T\cdot|p'-p|/|H|
$
we have $\phi^{-t}(p')\not\in\Sigma$, where $|p'-p|$ denotes the distance on $H$ from $p'$ to $p$. In this case $p$ belongs to a closed geodesic $\sigma$ whose direction $\theta_\sigma$ satisfies
$
|\theta-\theta_\sigma|=\arcsin\big(|H|/T\big)
$,
thus the boundary of the cylinder $C_\sigma$ is union of saddle connection satisfying the required property. Similarly, if $\phi^{t}(p')\not\in\Sigma$ for any point $p'\in H'$ and for any $t$ with
$
0\leq t\leq T\cdot|p'-p|/|H|
$, 
where $|p'-p|$ denotes the distance in $H'$ from $p'$ to $p$, then again $p$ belongs to a closed geodesic $\sigma$ in direction $\theta_\sigma$ as above, and the same argument gives a saddle connection with the required properties. In the only remaining case we have two conical singularities $p_i$ and $p_j$ of $\Sigma$ and instants $0\leq s\leq T$ and $0\leq t\leq T$ such that $\phi_\theta^t(p_i)\in H$ and $\phi_\theta^{-s}(p_j)\in H'$. Then $p_i$ and $p_j$ can be connected by a saddle connection satisfying the required property. The first part of the Lemma is proved. The second part just holds because the set of vectors $\hol(\gamma)$ for $\gamma$ saddle connection is a discrete subset of $\RR^2$.
\end{proof}

\begin{lemma}
\label{lem2relationnontypes}
Let $\theta$ be any direction without saddle connections on the surface $X$. Fix $\eta>2$ and suppose that there exists a point $p\in X$ such that
$$
\omega_\theta(p)<\frac{1}{\eta-1}.
$$
Then we have
$
\theta\in \cW^{sc}(\eta)
$.
\end{lemma}

\begin{proof}
According to the definition of $\omega_\theta(p)$ there exists $r$ arbitrarily small with
$
R_\theta(p,r)^{\eta-1}<1/r
$.
Set $T:=R_\theta(p,r)$, so that we have $|\phi_\theta^T(p)-p|<r$, and assume without loss of generality that $\phi_\theta^T(p)$ is connected to $p$ by an horizontal segment of length less than $r$. According to Lemma \ref{lem1relationnontypes} there exists a saddle connection $\gamma$ such that $|\re(\theta,\gamma)|<r$ and $|\im(\theta,\gamma)|<T$, that is
$$
|\re(\theta,\gamma)|<
r<
\frac{1}{R_\theta(p,r)^{\eta-1}}=
\frac{1}{T^{\eta-1}}<
\frac{1}{|\im(\theta,\gamma)|^{\eta-1}}.
$$
The Lemma follows observing that, since $r$ is arbitrarily small, $\gamma$ is arbitrarily long, that is there exist infinitely many saddle connections satisfying the condition above.
\end{proof}

\begin{lemma}
\label{LemRecBilliard(UpperBound)}
Let $\theta$ be a direction without saddle connections on the surface $X$. Assume that $\theta$ is uniquely ergodic and that $\theta\in \cW^{cyl}(\eta)$. Then for almost every $p\in X$ we have
$$
\omega_\theta(p)\leq\frac{1}{\eta-1}.
$$
\end{lemma}

\begin{proof}
It is not a loss of generality to assume that $\area(X)=1$. According to the definition of $\cW^{cyl}(\eta)$ there exists an arbitrarily long closed geodesic $\sigma$, whose corresponding cylinder $C_\sigma$ satisfies $\area(C_\sigma)>a$, such that
$
|\re(\theta,\sigma)|<|\im(\theta,\sigma)|^{-(\eta-1)}
$.
Set $T:=|\im(\theta,\sigma)|$ and let $\textrm{Rec}(C_\sigma)$ be the set of points $p\in C_\sigma$ such that
$
\phi_\theta^T(p)\in C_\sigma
$.
Since $\area(C_\sigma)>a$ then the horizontal transversal $H_\sigma$ to $C_\sigma$ has length $|H_\sigma|>a/T$. Without loss of generality we can assume that $aT^{\eta-2}>2$, thus we have
$$
\leb\big(\textrm{Rec}(C_\sigma)\big)>
\big(1-\frac{|\re(u)|}{|H_\sigma|}\big)
\cdot
\leb\big(C_\sigma\big)>
\big(1-\frac{1}{aT^{\eta-2}}\big)
\cdot
\leb\big(C_\sigma\big)>a/2.
$$
Moreover, setting
$
r:=T^{-(\eta-1)}
$
and observing that for any
$
p\in \textrm{Rec}(C_\sigma)
$
we have
$$
|\phi^T(p)-p|=
|\re(\theta,\sigma)|<
|\im(\theta,\sigma)|^{-(\eta-1)}=
T^{-(\eta-1)}=r
$$
we get
$
R_\theta(p,r)^{\eta-1}=T^{\eta-1}=r^{-1}
$
and thus
\begin{equation}
\label{EqLemRecBilliard(UpperBound)}
\frac
{\log\big(R_\theta(p,r)\big)}
{-\log r}=
\frac{1}{\eta-1}
\textrm{ for any }
p\in \textrm{Rec}(C_\sigma).
\end{equation}
Since $\theta\in \cW^{cyl}(\eta)$, repeat the construction for a sequence of closed geodesics $\sigma_n$ whose corresponding cylinder $C_{\sigma_n}$ satisfies $\area(C_{\sigma_n})>a$ and such that
$
|\re(\theta,\sigma_n)|<|\im(\theta,\sigma_n)|^{-(\eta-1)}
$.
Equation \eqref{EqLemRecBilliard(UpperBound)} is satisfied for any $
p\in \textrm{Rec}(C_{\sigma_n})
$
and for
$
r_n:=|\im(\theta,\sigma_n)|^{-(\eta-1)}
$.
If follows that
$$
\liminf_{r\to0}
\frac
{\log\big(R_\theta(p,r)\big)}
{-\log r}\leq
\frac{1}{\eta-1}
\textrm{ for any }
p\in\bigcap_{N\in\NN}
\bigcup_{n>N}\textrm{Rec}(C_{\sigma_n}).
$$
Finally observe that
$$
\leb\bigg(
\bigcap_{N\in\NN}
\bigcup_{n>N}\textrm{Rec}(C_{\sigma_n})
\bigg)\geq
\limsup_{n\to\infty}
\leb\bigg(
\textrm{Rec}(C_{\sigma_n})
\bigg)
\geq\frac{a}{2}.
$$
The Lemma follows because $\omega_\theta:X\to\RR_+$ is constant almost everywhere, since it is a invariant under $\phi^t_\theta$, and there is a set of positive measure where
$
\omega_\theta(p)\leq 1/(\eta-1)
$.
\end{proof}

\subsection{End of the proof}
\label{SecEndProofRecurrenceBilliards}

Here we finish the proof of Theorem \ref{ThmRecurrenceBilliads}. Of course it is enough to prove the analogous statement for the flow $\phi_\theta$ on a translation surface $X$.

\medskip

Let $\textrm{NUE}(X)$ be the set of non-uniquely ergodic directions $\theta$ on the translation surface $X$ and let
$
\lambda:=\dim\big(\textrm{NUE}(X)\big)
$.
Recall that we have $0\leq\lambda\leq 1/2$ and consider $\tau$ with $2<\tau<2/\lambda$, that is $1>2/\tau>\lambda$. Set
$$
S_\tau:=
\cW^{cyl}(\tau)
\setminus
\bigg(\textrm{NUE}(X)
\cup
\bigcup_{\eta'>\tau}\cW^{sc}(\eta')
\bigg)
$$
According to Theorem \ref{ThmKhinchinJarnickTranslationSurfaces} we have
$
H^{2/\tau}\big(\cW^{cyl}(\tau)\big)=+\infty
$
and for any $\eta$ with $\eta>\tau$ we have
$
H^{2/\tau}\big(\cW^{sc}(\eta)\big)=0
$.
Moreover, for any $\eta_1$ and $\eta_2$ with $\eta_1>\eta_2>\tau$ we have
$
\cW^{sc}(\eta_1)\subset\cW^{sc}(\eta_2)
$,
and thus
$
H^{2/\tau}
\big(\bigcup_{\eta>\tau}\cW^{sc}(\eta)\big)
=0
$.
Since $2/\tau>\lambda$ we have
$
H^{2/\tau}\big(\textrm{NUE}(X)\big)=0
$.
It follows that
$
H^{2/\tau}\big(S_\tau\big)=+\infty.
$
Applying again Theorem \ref{ThmKhinchinJarnickTranslationSurfaces}, one gets
$
H^s\big(\cW^{cyl}(\tau)\big)=0
$
for any $s>2/\tau$, therefore
$
\dim\big(S_\tau\big)=2/\tau
$.
Lemma \ref{LemRecBilliard(UpperBound)} implies
$$
\omega_\theta(p)\leq\frac{1}{\tau-1}
\textrm{ for any }
\theta\in S_\tau
\textrm{ and for almost any }
p\in X.
$$
Finally, if there exists $p\in X$ and some $\eta'>\tau$ such that
$
\omega_\theta(p)<1/(\eta'-1)
$
then Lemma \ref{lem2relationnontypes} implies
$
\theta\in\cW^{sc}(\eta')
$
and thus $\theta\not\in S_\tau$. Theorem \ref{ThmRecurrenceBilliads} is proved.

\appendix

\section{Proof of Corollary \ref{CorollaryMinskyWeissVeech}}
\label{SectionProofMinskyWeissForVeech}

Let $X$ be any translation surface and let $\Gamma(X)$ be the set of its saddle connections. For any $\gamma\in\Gamma(X)$ consider the function 
$$
L_\gamma:\RR\to\RR_+
\quad;\quad
\alpha\mapsto L_\gamma(\alpha):=\|\hol(\gamma,u_{-\alpha}\cdot X)\|_\infty,
$$
where $\|(x,y)\|_\infty:=\max\{|x|,|y|\}$ for any $(x,y)\in\RR^2$. Let $\cG(X)$ be the family of functions 
$
\cG(X):=\{L_\gamma(\cdot);\gamma\in\Gamma(X)\}
$. 
Consider any $\gamma\in\Gamma(X)$, any interval $J\subset\RR$ and any $\lambda>0$, then let $J(\gamma,\lambda)$ be the subinterval of $J$ defined by
$$
J(\gamma,\lambda):=
\{\alpha\in I;L_\gamma(\alpha)\leq\lambda\}.
$$
Define also 
$
J(X,\lambda):=\bigcup_{\gamma\in\Gamma(X)}J(\gamma,\lambda)
$, 
that is the set of those $\alpha\in J$ such that there exists some $\gamma\in\Gamma(X)$ with 
$
L_\gamma(\alpha)\leq\lambda
$. 
For any interval $J$ and any $\gamma\in\Gamma(X)$ set also 
$$
\|L_\gamma\|_J:=\sup_{\alpha\in J}L_\gamma(\alpha).
$$ 

For any Borel set $E\subset\RR$ denote by $|E|$ its Lebesgue measure. According to Proposition 4.5 in \cite{minskyweiss}, for any translation surface $X$ the family of functions $\cG(X)$ is \emph{$(2,1)$-good}, that is for any $\lambda>0$, any interval $J\subset\RR$ and any $\gamma\in\Gamma(X)$ we have
$$
\frac{|J(\gamma,\lambda)|}{|J|}
\leq
2\cdot\frac{\lambda}{\|L_\gamma\|_J},
$$
where $1$ in $(2,1)$-good refers to the exponent of the term $\lambda/\|L_\gamma\|_J$, which in the general definition of $(C,\beta)$-good families of functions is allowed to be smaller. The general Proposition 3.2 in \cite{minskyweiss}, adapted in our setting to the family of functions $\cG(X)$, says that if there exists constants $\rho>0$ and $M>0$ such that for any interval $J\subset \RR$ we have 
\begin{enumerate}
\item
$\|L_\gamma\|_J\geq \rho$ for any $\gamma\in\Gamma(X)$
\item
$
\sharp\{\gamma\in\Gamma(X);L_\gamma(\alpha)\leq \rho\}\leq M
$
for any $\alpha\in J$
\end{enumerate}
then for any $0<\epsilon\leq \rho$ we have
\begin{equation}
\label{EqMinskyWeissAppendix}
\frac{|J(X,\epsilon)|}{|J|}
\leq2M\cdot \frac{\epsilon}{\rho}.
\end{equation}
For completeness, we give a proof of Equation \eqref{EqMinskyWeissAppendix}. Observe first that Condition (2) implies
$$
\cI:=
\int_J\sharp\{\gamma\in\Gamma(X);L_\gamma(\alpha)\leq \rho\}d\alpha\leq M|J|.
$$
On the other hand, since the family $\cG(X)$ is $(2,1)$-good, Condition (1) implies that for any $\gamma\in\Gamma(X)$ we have
$$
|J(\gamma,\epsilon)|\leq 
2|J(\gamma,\rho)|\frac{\epsilon}{\rho}.
$$
Therefore Equation \eqref{EqMinskyWeissAppendix} follows observing that
$$
\cI=\sum_{\gamma\in\Gamma(X)}|J(\gamma,\rho)|
\geq
\left(2\frac{\epsilon}{\rho}\right)^{-1}\sum_{\gamma\in\Gamma(X)}|J(\gamma,\epsilon)|
\geq
\left(2\frac{\epsilon}{\rho}\right)^{-1}|J(X,\epsilon)|.
$$

In general, Condition (2) is not satisfied for any translation surface $X$. When $X$ is a Veech surface Condition (2) is satisfied according to Lemma \ref{LemmaSparseCorverVeech} below. In order to prove Corollary \ref{CorollaryMinskyWeissVeech}, let $X$ be a Veech surface. Fix any interval $J\subset\RR$ and some $\rho>0$. Assume that for any $\gamma\in\Gamma(X)$ we have
$$
\sup_{\alpha\in J}|\hol(\gamma,u_{-\alpha}\cdot X)|\geq \rho.
$$ 
It follows that 
$
\|L_\gamma\|_J\geq\rho/\sqrt{2}
$, 
according to the comparison between the norm $\|\cdot\|_\infty$ and the euclidian norm $|\cdot|$ on $\RR^2$. For any $0<\epsilon<\rho$ Equation \eqref{EqMinskyWeissAppendix} implies 
$$
\frac{|J(X,\epsilon)|}{|J|}
\leq2\sqrt{2}M\cdot \frac{\epsilon}{\rho}.
$$
Finally the comparison between the norms $\|\cdot\|_\infty$ and $|\cdot|$ gives
$$
|\{\alpha\in J;\sys(u_{-\alpha}\cdot X)\leq \epsilon\}|
\leq
4M\cdot\frac{\epsilon}{\rho}\cdot|J|.
$$

We complete the proof of Corollary \ref{CorollaryMinskyWeissVeech} stating and proving Lemma \ref{LemmaSparseCorverVeech} below.

\begin{lemma}
\label{LemmaSparseCorverVeech}
Let $X$ be a Veech surface and let $\cM:=\sltwor\cdot X$ be its closed orbit under the action of $\sltwor$. Then there exists some $r_0>0$, depending only on $\cM$, such that for any $G\in\sltwor$ we have
$$
\sharp\{\gamma\in\Gamma(X);|\hol(\gamma,G\cdot X)\leq r_0|\}\leq 4g-4.
$$
\end{lemma} 

\begin{proof}
Let $\theta_\gamma$ be the direction of any saddle connection. It is well-known (see \cite{ForniMatheus}
) that there exists a decomposition of $X$ into cylinders $C_1,\dots,C_n$ in direction $\theta_\gamma$, where $n=n(\theta_\gamma)$. Moreover there exists a finite number of saddle connection directions $\theta_1\dots,\theta_N$, where $N$ is the number of cusps of $\cM$, such that the cylinder decomposition in any saddle connection direction $\theta_\gamma$ is the affine image of the cylinder decomposition in one of the saddle connection directions $\theta_1\dots,\theta_N$ under some element of the Veech group of $X$. It follows that there exists some $a=a(\cM)>0$ such for any $G\in\sltwor$, any cylinder $C_\sigma$ for $G\cdot X$ has area
$$
\area(C_\sigma)\geq a^2.
$$
Moreover there exists some $M=M(\cM)>1$ such that is $\sigma$ is a closed geodesic and $\gamma$ is a saddle connection parallel to $\sigma$, then we have
$$
|\sigma|\leq M\cdot|\gamma|,
$$ 
where the last condition obviously holds for any affine deformation $G\cdot X$ of the surface $X$, where $G\in\sltwor$. We will prove the statement with 
$$
r_0:=\frac{a}{M}.
$$
Let $\gamma$ be a saddle connection with
$
\displaystyle{|\hol(\gamma,G\cdot X)|\leq\frac{a}{M}}
$. 
Let $\sigma$ be a closed geodesic parallel to $\gamma$ and let $C_\sigma$ be the corresponding cylinder. Since $|\sigma|\leq a$ then $C_\sigma$ must have transversal component $W_\sigma\geq a$. Therefore any saddle connection $\gamma'$ which crosses $C_\sigma$ must have length $|\hol(\gamma',G\cdot X)|\geq a$. The Lemma follows observing that for any saddle connection $\gamma$ there are at most $4g-4$ saddle connections parallel to $\gamma$, and all the other must cross at least one of the cylinders parallel to $\gamma$. 
\end{proof}

\section{Isotropic quadratic growth fails for saddle connections}
\label{SectionIsotropicQuadraticGrowthFails}

In this appendix we show that isotropic quadratic growth fails for saddle connections directions.

\begin{lemma}
\label{LemmaIsotropicQuadraticGrowthFails}
Let $X$ be a translation surface whose orbit under $\sltwor$ is dense in the connected component of its stratum $\cH$. Then the set $\cR^{sc}(X)$ does not satisfy isotropic quadratic growth.
\end{lemma}

\begin{proof}
Consider a translation surface $X$ and let $\hol(X)$ be the discrete set of all holonomy vectors 
$v=\hol(\gamma,X)$, where $\gamma$ varies among the set of all saddle connections of $X$. If $A\subset \RR^2$ is a bounded open subset set $N(X,A):=\sharp \big(A\cap\hol(X)\big)$. More generally, if $f:\RR^2\to\RR$ is a bounded function with compact support its \emph{Siegel-Veech} transform is the map $\widehat{f}:\cH\to \RR$ 
defined by 
$$
\widehat{f}(X):=\sum_{v\in \hol(X)}f(v).
$$
In particular we have $N(X,A)=\widehat{f_A}(X)$, where $f_A$ denotes the indicatrix function of $A$, that is $f_A(v)=1$ if $v\in A$ and $f_A(v)=0$ otherwise. Let 
$
B:=\{v\in\RR^2\textrm{ ; }|v|<1\}
$ 
be the unit euclidian ball and let $f_B$ its indicatrix function. According to Corollary~5.11 in \cite{athreyachaika}, the Siegel-Veech transform $\widehat{f_B}$ is not in $L^3(\cH,\mu)$, where $\mu$ is the absolutely continuous $\sltwor$-invariant measure on $\cH$ (see also Line 4, page 3 in \cite{AthreyaCheungMasur}). In particular $\widehat{f_B}$ is not bounded. Now let $\Delta=\Delta_0$ be the equilateral triangle with vertices at $(0,0)$, $(3^{-1/4},3^{1/4})$ and $(-3^{-1/4},3^{1/4})$ and let 
$f_\Delta$ be its characteristic function. Observe that $\area(\Delta)=1$. Let also $\Delta_1,\dots,\Delta_5$ the rotated copies of $\Delta$, so that the union gives an hexagon containing $B$. Fix any $N\in\NN$. Since $\widehat{f_B}$ is not bounded, modulo a rotation, the pigeonhole principle implies that there exists some $X_0\in\cH$ such that 
$
\widehat{f_{B\cap\Delta}}(X_0)>N
$. 
Considering a smooth approximation $g$ of $f_\Delta$ and using the continuity of $\widehat{g}$, one can see that there exists an open set $\cV\subset\cH$ with $X_0\in\cV$ such that 
$$
N(X,B\cap\Delta)\geq N/2
\quad
\textrm{ for any }
\quad
X\in\cV.
$$
Now let $X\in\cH$ be surface as in the statement, so that $SL(2,R)\cdot X$ is dense in $\cH$. Then by \cite{ChaikaEskin} there exists a direction $\theta$ and $t>>1$ such that $g_t r_{-\theta} X \in \cV$; so that 
$$
N(g_t r_{-\theta} X,B\cap\Delta)>N/2
$$
The isosceles triangle
$$
\Delta':=(g_t r_{-\theta})^{-1}(\Delta)=r_\theta g_{-t}\Delta
$$
has shortest side with length $e^{-t}\cdot 2\cdot 3^{-1/4}$, while the altitude with respect to such shortest side is $e^t\cdot 3^{1/4}$. Let $I \subset S^1$ be the angular sector spanned by $\Delta'$ and set $R:=e^{t}\cdot 3^{1/4}$. If $t>>1$ is big enough we have 
$$
1=\area(\Delta)=\area(\Delta')\leq |I|\cdot R^2\leq 2
$$
but on the other hand
\begin{align*}
&
\sharp
\{\theta\in I\cap\cR^{sc}(X)\textrm{ ; }l(\theta)\leq R\}
\geq 
\frac{1}{3m}
\sharp\{v\in\hol(X)\textrm{ ; }|v|<R\textrm{ ; }\theta_v\in I\}
\geq
\\
&
\frac{1}{3m}N\big(X,r_\theta g_{-t}(B\cap\Delta)\big)=
\frac{1}{3m}N(g_t r_{-\theta} X,B\cap\Delta)>
\frac{N}{6m}>\frac{|I|\cdot R^2}{12m},
\end{align*}
where the first inequality holds because on any translation surface $X\in\cH$ there are at most $3m$ parallel saddle connections, the second holds because for $t>>1$ big enough we have 
$
g_{-t}(B\cap\Delta)\subset \{v\in\RR^2\textrm{ ; }|v|<R\textrm{ ; }\theta_v\in I\}
$, 
and the last holds because $|I|\cdot R^2<2$. The statement follows because $N$ is arbitrarily big.
\end{proof}


\begin{thebibliography}{0000}

\bibitem[At1,Ch4]{athreyachaika}
J. S. Athreya, J. Chaika:
\emph{The distribution of gaps of saddle connection directions}.
Geom. Funct. Anal. 22 (2012), no. 6, 1491-1516.

\bibitem[At2,Che3,Ma]{AthreyaCheungMasur}
J. S. Athreya, Y. Cheung, H. Masur:
\emph{Siegel-Veech transforms are in $L^2$}. 
ArXiv:1711.08537.

\bibitem[Be1,Ve1]{velanimasstransfer}
V. Beresnevich, S. Velani:
\emph{A Mass Transference Principle and the Duffin-Schaeffer conjecture for Hausdorff measures}.
Annals of Mathematics, 164, (2006), 971-992.

\bibitem[Be2,Di,Ve2]{BeresnevichDickinsonVelani}
V. Beresnevich, D. Dickinson, S. Velani:
\emph{Measure Theoretic Laws for Limsup Sets}.
Memoirs of the American Mathematical Society, Volume 179, Number 846, 2006.

\bibitem[Be3,Ve3]{velaniubiquity}
V. Beresnevich, S. Velani:
\emph{Ubiquity and a general logarithmic law for geodesics}.
Dynamical systems and diophantine approximations, 21-36, S\'emin. Congr., 19, Soc. Math. France, Paris, 2009.

\bibitem[Boi,Ge]{BoissyGeninska}
C. Boissy, S. Geninska:
\emph{Systoles in translation surfaces}. ArXiv:1707.05060.

\bibitem[Bo,Ch1]{boshernitzanchaika}
M. Boshernitzan, J. Chaika:
\emph{Borel-Cantelli sequences}.
J. Anal. Math. 117 (2012), 321-345.

\bibitem[Br,Kl]{BroderickKleinbock}
R. Broderick, D. Kleinbock:
\emph{Dimension estimates for sets of uniformly badly approximable systems of linear forms}.
Int. Journal of Number Theory (2015), Vol. 11, No. 07, 2037-2054.

\bibitem[Ch2]{ChaikaHomogeneousApproximationsTranslSurf}
J. Chaika:
\emph{Homogeneous approximations on translation surfaces}. Arxiv:1110.6167.

\bibitem[Ch3,Che1,Ma4]{chaikacheungmasur}
J. Chaika, Y. Cheung, H. Masur:
\emph{Winning games for bounded geodesics in moduli spaces of quadratic differentials}.
J. Mod. Dyn. 7 (2013), no. 3, 395-427.

\bibitem[Ch4,Tr]{ChaikaTrevino}
J. Chaika, R. Trevi\~no:
\emph{Logarithmic laws and unique ergodicity}.
J. Mod. Dyn. 7 (2017), no. 11, 563-588.

\bibitem[Che2,Hu1,Ma5]{CheungHubertMasur}
Y. Cheung, P. Hubert, H. Masur:
\emph{Dichotomy for the Hausdorff dimension of the set of nonergodic quadratic differentials}.
Invent. Math., 183, (2011), 337-383.

\bibitem[Ch5,Es2]{ChaikaEskin}
J. Chaika, A. Eskin:
\emph{Every flat surface is Birkhoff and Oseledets generic in almost every direction}.
Journal of Modern Dynamics (2015), 9, 1-23.

\bibitem[Do]{Dozier}
B. Dozier:
\emph{Equidistribution of saddle connections on translation surfaces}. 
ArXiv:1705.10847.

\bibitem[Es1,Ma5]{EskinMasurAsymptoticFormulas}
A. Eskin, H. Masur:
\emph{Asymptotic formulas on flat surfaces}.
Ergod. Th. and Dynam. Sys. (2001), 21, 443-478.

\bibitem[Fa]{falconer}
K. Falconer:
\emph{Fractal geometry: Mathematical Foundations and Applications}.
John Wiley and sons, New York (1990).

\bibitem[Fo,Mat]{ForniMatheus}
G. Forni, C. Matheus:
\emph{Introduction to Teichm\"uller theory and its applications to dynamics of interval exchange transformations flows on surfaces and billiards}.
Arxiv:1311.2758.

\bibitem[Hu2,Mar1,Ul]{HubertMarcheseUlcigrai}
P. Hubert, L. Marchese, C. Ulcigrai:
\emph{Lagrange Spectra in Teichm\"uller Dynamics via renormalization}.
Geom. Funct. Anal., Vol 25 (2015), 180-255.

\bibitem[Ja]{jarnik}
V. Jarn\'ik:
\emph{Zur Theorie der diophantischen Approximationen} (German).
Monatsh. Math. Phys. 39 (1932), no. 1, 403-438.

\bibitem[He]{Hensley}
D. Hensley:
\emph{Continued fractions, Cantor sets, Hausdorff dimension and functional analysis}.
J. Number Theory, 40 (1992), 336-358.

\bibitem[Ke,Ma,Sm]{KerkoffMasurSmillie}
S. Kerkoff, H. Masur, J. Smillie:
\emph{Ergodicity of Billiard Flows and Quadratic Differentials}. Ann. Math. 124 (1986), 293-311.

\bibitem[Kim,Mar,Marmi]{KimMarcheseMarmi}
D. H. Kim, L. Marchese, S. Marmi:
\emph{Long hitting time for translation flows and L-shaped billiards}. ArXiv:1705.03328.

\bibitem[Kim,Marmi]{KimMarmi}
D. H. Kim, S. Marmi:
\emph{The recurrence time for interval exchange maps}.
Nonlinearity, 21, (2008), 2201-2210.

\bibitem[Kl1,We1]{KleinbockWeiss}
D. Kleinbock, B. Weiss:
\emph{Bounded geodesics in moduli space}.
Int. Math. Res. Not. 2004, no 30, 1551-1560.

\bibitem[Kl2,Li1,We2]{KleinbockLinderstraussWeiss}
D. Kleinbock, E. Linderstrauss, B. Weiss:
\emph{On fractal measures and diophantine approximation}.
Selecta Mathematica, New Series, 10 (2004), 479-523.

\bibitem[Kur]{Kurzweil}
J. Kurzweil:
\emph{A contribution to the metric theory of diophantine approximations}. 
Czechoslovak Math. J. 1, (1951), 149-178.

\bibitem[Mar2]{MarcheseAsymptoticLaws}
L. Marchese:
\emph{Khinchin type condition for translation surfaces and asymptotic laws for the Teichm\"uller flow}.
Bull. Soc. Math. France, 140, fascicule 4, 2012, 485-532.

\bibitem[Ma1]{MasurUniqErgDir}
H. Masur:
\emph{Hausdorff dimension of the set of nonergodic foliations of a quadratic differential}.
Duke Mathematical Journal, 66, (1992), no. 3, 387-442.

\bibitem[Ma2]{MasurQuadraticGrowth}
H. Masur:
\emph{The growth rate of trajectories of a quadratic differential}.
Ergodic Theory and Dynamical Systems (1990), 10, 151-176.

\bibitem[Ma3]{MasurLogLaw}
H. Masur:
\emph{Logarithmic law for geodesics in moduli space}.
Contemporary Mathematics, 150 (1993), 229-245.

\bibitem[Mc]{McMullen}
C. McMullen:
\emph{Winning sets, quasiconformal maps and Diophantine approximations}.
Geom. Funct. Anal. 20 (2010), no. 3, 726-740.

\bibitem[Mi,We3]{minskyweiss}
Y. Minsky, B. Weiss:
\emph{Non-divergence of the horocyclic flow on moduli space}.
J. Reine Angew. Math. 552 (2002), 131-177.

\bibitem[Si]{Simmons}
D. Simmons:
\emph{A Hausdorff measure version of the Jarn\'ik-Schmidt theorem in diophantine approximations}. 
Arxiv:1509.03885.

\bibitem[Su]{SullivanLogLaw}
D. Sullivan:
\emph{Disjoint spheres, approximations by imaginary quadratic numbers, and the logarithm law for geodesics}.
Acta Math. 149 (1982), 215-238.

\bibitem[Vor]{vorobets1}
Y. Vorobets:
\emph{Periodic geodesics on generic translation surfaces}.
Algebraic and topological dynamics, 205-258, Contemp. Math., 385, Amer. Math. Soc., Providence, RI, 2005.

\bibitem[Wei]{WeilSteffen}
S. Weil:
\emph{Jarn\'ik type inequalities}.
Proceedings of London Math. Soc..
(2015) 110 (1), 172-212.

\bibitem[Yo]{yoccozIET}
J. C. Yoccoz:
\emph{Interval exchange maps and translation surfaces}.
Homogeneous flows, moduli spaces and arithmetic, 1-69, Clay Math. Proc., 10, Amer. Math. Soc., Providence, RI, 2010.

\bibitem[Zo]{ZorichFlat}
A. Zorich:
\emph{Flat surfaces}.
Frontiers in number theory, physics, and geometry. I, 437-583, Springer, Berlin, 2006.


\end{thebibliography}
\end{document}